%% file: template.tex
\documentclass[12pt]{article}
\usepackage[margin =1in]{geometry}
\usepackage{amssymb,amsthm,amsmath}
\usepackage{enumerate}
\usepackage{hyperref}
\usepackage{cite}
\usepackage[capitalize]{cleveref}
\usepackage{enumitem}
\usepackage{color}
\usepackage{cancel}

\usepackage[dvipsnames]{xcolor}
\definecolor{lblue}{HTML}{908cc0}
\definecolor{mblue}{HTML}{519cc8}
\definecolor{hblue}{HTML}{1d5996}
\definecolor{lred}{HTML}{cb5501}
\definecolor{hred}{HTML}{b3001e}

\usepackage{tikz}
\usepackage{pgf}
\usepackage{svg}
\usepackage{pgfplots}
\usepackage{import}

\usepackage{graphicx}

\usepackage[title]{appendix}
\numberwithin{equation}{section}
\newcommand{\tr}{\mathrm{Tr}}

\newcommand{\cI}{\mathcal{I}}
\newcommand{\cX}{\mathcal{X}}

\newcommand{\cE}{\mathcal{E}}
\newcommand{\ip}[1] {\langle #1 \rangle }
\newcommand{\norm}[1]{\left \| #1 \right \|}
\newcommand{\opnorm}[1]{\norm{#1}_{\mathrm{op}}}
\newcommand{\inclu}[0] {\ar@{^{(}->}}

\newcommand{\dist}{{\rm dist}}
\newcommand{\R}{\mathbb{R}}

\newcommand{\CC}{\mathbb{C}}
\newcommand{\cA}{\mathcal{A}}

\newcommand{\cS}{\mathcal{S}}

\newcommand{\cN}{\mathcal{N}}
\newcommand{\EE}{\mathbb{E}}

\newcommand{\sign}{\mathrm{sign}}

\newcommand{\Var}{\mathrm{Var}}
\newcommand{\RR}{\mathbb{R}}
\newcommand{\PP}{\mathbb{P}}
\newcommand{\pfail}{p_{\mathrm{fail}}}
\newcommand{\Qfail}{Q_{\mathrm{fail}}}
\newcommand{\rank}{\mathrm{rank}}
\newcommand{\op}{\mathrm{op}}

\newcommand{\outliers}{\cI_{\mathrm{out}}}
\newcommand{\inliers}{\cI_{\mathrm{in}}}
\newcommand{\sel}{\cI^{\mathrm{sel}}}
\newcommand{\seloutliers}{\cI^{\mathrm{sel}}_{\mathrm{out}}}
\newcommand{\selinliers}{\cI^{\mathrm{sel}}_{\mathrm{in}}}
\newcommand{\dirw}{{\bar w_\star}} 
\newcommand{\dirx}{{\bar x_\star}}

\newcommand{\estrad}{\widehat M}

\newcommand{\Sin}{S_\mathrm{in}}
\newcommand{\Sout}{S_\mathrm{out}}
\newcommand{\omegafail}{\omega_{\mathrm{fail}}}

\renewcommand{\SS}{{\mathbb{S}}}

\newcommand{\abs}[1]{\left| #1 \right|}

\newcommand{\proj}{\mathrm{proj}}

\newcommand{\argmin}{\operatornamewithlimits{argmin}}

\newcommand{\Prob}[1]{\mathbb{P}\left( #1 \right)}

\newcommand{\matrx}[1]{\begin{bmatrix} #1 \end{bmatrix}}
\newtheorem{thm}{Theorem}[section]
\newtheorem{theorem}[thm]{Theorem}

\newtheorem{proposition}[thm]{Proposition}
\newtheorem{lem}[thm]{Lemma}
\newtheorem{lemma}[thm]{Lemma}

\newtheorem{corollary}[thm]{Corollary}

\newtheorem{assumption}{Assumption}

\crefname{claim}{claim}{claims}
\Crefname{claim}{Claim}{Claims}
\crefname{lem}{lemma}{lemmas}
\Crefname{lem}{Lemma}{Lemmas}
\crefname{algorithm}{algorithm}{algorithms}
\Crefname{algorithm}{Algorithm}{Algorithms}

\newtheorem{example}{Example}[section]

\theoremstyle{remark}
\newtheorem{claim}{Claim}
\newcommand{\cT}{\mathcal{T}}
\newcommand{\cD}{\mathcal{D}}
\newcommand{\cIs}{\cI^\mathrm{sel}}
\newcommand{\cIo}{\cI_\mathrm{out}}
\newcommand{\cIi}{\cI_\mathrm{in}}
\newcommand{\cIis}{\cIi^\mathrm{sel}}
\newcommand{\cIos}{\cIo^\mathrm{sel}}
\newcommand{\1}{\mathbf{1}}
\newcommand{\mmid}{\ \middle|\ }
\newcommand{\set}[1]{\left\{ #1 \right\}}

\newcommand{\Linit}{L^{\mathrm{init}}}
\newcommand{\Rinit}{R^{\mathrm{init}}}
\newcommand{\what}{\widehat{w}}
\newcommand{\xhat}{\widehat{x}}
\newcommand{\wbar}{{\bar{w}}}
\newcommand{\xbar}{{\bar{x}}}
\newcommand{\wstar}{{w^{\star}}}
\newcommand{\xstar}{{x^{\star}}}
\newcommand{\quant}{\mathtt{quant}}

\newcommand{\med}{\mathtt{med}}

\newcommand{\qfail}{q_{\mathrm{fail}}}

\newcommand{\normal}{\mathsf{N}}

\usepackage{mathtools}
\DeclarePairedDelimiter{\dotp}{\langle}{\rangle}
\usepackage[boxruled, noline]{algorithm2e}

\begin{document}
	
	\title{Composite optimization for robust blind deconvolution}
	%\subtitle{~}
	
	%\titlerunning{~}        % if too long for running head
	
	\author{Vasileios Charisopoulos\thanks{School of Operations Research
			and Information Engineering, Cornell University, Ithaca, NY 14850,
			USA; \texttt{people.orie.cornell.edu/vc333/}}
	\qquad Damek Davis\thanks{School of Operations Research and Information 
	Engineering, Cornell University,
	Ithaca, NY 14850, USA;
	\texttt{people.orie.cornell.edu/dsd95/}} \qquad Mateo D\'iaz\thanks{Center 
	for Applied Mathematics, Cornell University Ithaca, NY 14850, USA; 
	\texttt{people.cam.cornell.edu/md825/}} \\ Dmitriy Drusvyatskiy\thanks{Department of Mathematics, U. Washington, 
Seattle, WA 98195; \texttt{www.math.washington.edu/{\raise.17ex\hbox{$\scriptstyle\sim$}}ddrusv}. Research of Drusvyatskiy was supported by the NSF DMS   1651851 and CCF 1740551 awards.}	}

	\date{}
	\maketitle
	
	\begin{abstract}

	The blind deconvolution problem seeks to recover a pair of vectors from a set of rank one bilinear measurements. We consider a natural nonsmooth formulation of the problem and show that under standard statistical assumptions, its moduli of weak convexity, sharpness, and Lipschitz continuity are all dimension independent. This phenomenon persists even when up to half of the measurements are corrupted by noise. Consequently, standard algorithms, such as the subgradient and prox-linear methods, converge at a rapid dimension-independent rate when initialized within constant relative error of the solution. We then complete the paper with a new initialization strategy, complementing the local search algorithms. The initialization procedure is both provably efficient and robust to outlying measurements. Numerical experiments, on both simulated and real data, illustrate the developed theory and methods.
\end{abstract}

\section{Introduction}

A variety of tasks in data science amount to solving a nonlinear system  $F(x)=0$, where $F\colon\R^d\to\R^m$ is a highly structured smooth map. The setting when $F$ is a quadratic map already subsumes important problems such as phase retrieval \cite{7078985,wirt_flow,ma2017implicit}, blind deconvolution \cite{ahmed2014blind,li2016rapid,MR3424852,proc_flow}, matrix completion \cite{rec_exa,davenport2016overview,MR3565131}, and covariance matrix estimation \cite{MR3367819,li2018nonconvex}, to name a few. Recent works have  suggested a number of two-stage procedures for globally solving such problems. The first stage---{\em initialization}---yields a rough estimate $x_0$ of an optimal solution, often using spectral techniques. The second stage---{\em local refinement}---uses a local search algorithm that rapidly converges to an optimal solution,  when initialized at $x_0$. For a detailed discussion, we refer the reader to the recent survey \cite{chi2018nonconvex}. %Our current work also centers around a two-stage procedure.

The typical starting point for local refinement is to form an optimization problem
\begin{equation}\label{eqn:target_problem}
\min_{x\in \cX}~ f(x):=h(F(x)),
\end{equation}
where $h(\cdot)$ is a carefully chosen penalty function and $\cX$ is a constraint set. Most widely-used penalties are smooth and convex; e.g., the squared $\ell_2$-norm $h(z)=\tfrac{1}{2}\|z\|^2_2$ is ubiquitous in this context. Equipped with such penalties, the problem \eqref{eqn:target_problem} is smooth and therefore gradient-based  methods become immediately applicable. The main analytic challenge is that the condition number $\tfrac{\lambda_{{\rm max}}(\nabla^2 f)}{\lambda_{{\rm min}}(\nabla^2 f)}$ of the problem \eqref{eqn:target_problem} often grows with the dimension of the ambient space $d$. This is the case for example for phase retrieval, blind deconvolution, and matrix completion problems; see e.g. \cite{chi2018nonconvex} and references therein.
Consequently, generic nonlinear programming guarantees yield efficiency estimates that are far too pessimistic. Instead, a fruitful strategy is to recognize that the Hessian may be well-conditioned along the ``relevant'' set of directions, which suffice to guarantee rapid convergence. This is where new insight and analytic techniques for each particular problem come to bare (e.g. \cite{proc_flow,ma2017implicit,MR3025133}).

Smoothness of the penalty function $h(\cdot)$ in \eqref{eqn:target_problem} is crucially used by the aforemention techniques. A different recent line of work \cite{duchi_ruan_PR,davis2018subgradient,davis2017nonsmooth,bai2018subgradient} has instead suggested the use of nonsmooth convex penalties---most notably the $\ell_1$-norm $h(z)=\|z\|_1$. Such a nonsmooth formulation will play a central role in our work. A number of algorithms are available for nonsmooth compositional problems \eqref{eqn:target_problem}, most notably the subgradient method %\cite{davis2018subgradient,goffin}
$$x_{t+1}=\proj_{\cX}(x_t-\alpha_t v_t)\qquad\textrm{with}\qquad v_t\in \partial f(x_t),$$
and the prox-linear algorithm
$$x_{t+1}=\argmin_{x\in\cX}~ h\Big(F(x_t)+\nabla F(x_t)(x-x_t)\Big)+\frac{1}{2\alpha_t}\|x-x_t\|^2_2.$$
The local convergence guarantees of both methods can be succinctly described as follows. Set $\cX^*:=\argmin_{\cX} f$ and suppose there exist constants $\rho,\mu,L>0$ satisfying:
\begin{itemize}
	\item {\bf (approximation)} $\left|h(F(y))-h\Big(F(x)+\nabla F(x)(y-x)\Big)\right|\leq \frac{\rho}{2}\|y-x\|^2_2$ for all $x\in \cX$,
	%\item $f$ is $L_{\gamma}$
	\item {\bf (sharpness)}  $f(x)-\inf f\geq \mu\cdot \dist(x,\cX^*)$  for all $x\in \cX$,
	\item {\bf (Lipschitz bound)} $\|v\|_2\leq L$ for all $v\in \partial f(x)$ with $\dist(x,\cX^*)\leq \frac{\rho}{\mu}$.
\end{itemize}
Then when equipped with an appropriate sequence $\alpha_t$ and initialized at a point $x_0$ satisfying $\dist(x_0,\cX^*)\leq \frac{\rho}{\mu}$, both the subgradient and prox-linear iterates  will converge to an optimal solution of the problem. The prox-linear algorithm converges quadratically, while the subgradient method converges at a linear rate governed by the ratio $\frac{\mu}{L}\in (0,1)$.

A possible advantage of nonsmooth techniques can be gleaned from the phase retrieval problem. The papers \cite[Corollary 3.1,3.2]{duchi_ruan_PR}, \cite[Corollary 3.8]{davis2017nonsmooth} recently, showed that for the phase retrieval problem, standard statistical assumptions imply that with high probability all the constants $\rho,\mu,L>0$ are dimension independent. Consequently, completely generic guarantees outlined above, without any modification, imply that both methods converge at a dimension-independent rate, when initialized within constant relative error of the optimal solution. This is in sharp contrast to the smooth formulation of the problem, where a more nuanced analysis is required, based on restricted smoothness and convexity.
%The prox-linear method, in turn, enjoys superior quadratic convergence guarantees, with the same initialization requirements. %The caveat is that the prox-linear method requires resolution of a large-scale convex optimization problem in each iteration. 
Moreover, this approach is robust to outliers in the sense that  analogous guarantees persist even when up to half of the measurements are corrupted by noise. 

In light of the success of the nonsmooth penalty approach for phase retrieval, it is intriguing to determine if nonsmooth techniques can be fruitful for a wider class of large-scale problems. Our current work fits squarely within this research program.
In this work, we  analyze a nonsmooth penalty technique for the problem of {\em blind deconvolution}. Formally, we consider the task of robustly recovering a pair $(\bar w,\bar x) \in \RR^{d_1} \times \RR^{d_2}$ from $m$ bilinear measurements:
\begin{equation}\label{eq:observations}
y_i = \dotp{\ell_i, \bar w} \dotp{r_i,\bar x}+\eta_i,
\end{equation}
where $\eta$ is an arbitrary noise corruption with frequency $\pfail:=\frac{|{\rm supp }~ \eta|}{m}$  that is at most one half, and $\ell_i\in\R^{d_{1}}$ and $r_i\in\R^{d_{2}}$ are known measurement vectors. Such bilinear systems and their complex analogues arise often in biological systems, control theory, coding theory, and image deblurring, among others. Most notably such problems appear when recovering a pair $(u,v) \in \CC^m \times \CC^m$ from the convolution measurements $y = (Lu) \ast (Rv) \in \CC^m.$
When passing to the Fourier domain this problem is equivalent to that of solving a complex bilinear system of equations; see the pioneering work \cite{ahmed2014blind}. All the arguments we present can be extended to the complex case. We focus on the real case for simplicity. 

In this work we analyze the following nonsmooth formulation of the problem:
\begin{equation}\label{formulation}
\min_{\|w\|_2,\,\|x\|_2\leq \nu\sqrt{M}} f(w,x) :=  \frac{1}{m}\sum_{i=1}^m |\dotp{\ell_i,  w} \dotp{r_i, x} - y_i|,
\end{equation}
where $\nu\geq 1$ is a user-specified constant and $M=\|\bar w\bar x^\top\|_F$. Our contributions are two-fold:
\begin{enumerate}
	\item {(\bf  Local refinement)} Suppose that the vectors $\ell_i$ and $r_i$ are both i.i.d. Sub-Gaussian and satisfy a mild growth condition (which is automatically satisfied for Gaussian random vectors). We will show that as long as the number of measurements satisfies $m\gtrsim \frac{d_1+d_2}{(1-2\pfail)^2}\ln(\frac{1}{1-2\pfail})$, the formulation \eqref{formulation} admits dimension independent constants $\rho$, $L$, and $\mu$ with high probability. Consequently, subgradient and prox-linear methods rapidly converge to the optimal solution at a dimension-independent rate when initialized at a point $x_0$ with constant relative error $\frac{\|w_0x_0^\top-\bar w\bar x^\top\|_F}{\|\bar w\bar x^T\|F}\lesssim 1$. Analogous results also hold under more general incoherence assumptions.
	\item {(\bf Initialization)} Suppose now that $\ell_i$ and $r_i$ are both i.i.d. Gaussian and are independent from the noise $\eta$. We develop an initialization procedure that in the regime $m\gtrsim d_1+d_2$ and $\pfail\in [0,1/10]$,
	will find a point $x_0$ satisfying $\frac{\|w_0x_0^\top-\bar w\bar x^\top\|_F}{\|\bar w\bar x^T\|F}\lesssim 1$, with high probability. To the best of our knowledge, this is the only available initialization procedure  with provable guarantees in presence of gross outliers. We also develop complementary guarantees under the weaker assumption that the vectors $(\ell_i,r_i)$ corresponding to exact measurements are independent from the noise $\eta_i$ in the outlying measurements. This noise model allows one to plant outlying measurements from a completely different pair of signals, and is therefore computationally  more challenging.  
	
\end{enumerate}

The literature studying bilinear systems is rich. From the information-theoretic
perspective~\cite{li2016identifiability,choudhary2014sparse,kech2017optimal}, the optimal sample complexity in the noiseless regime is $m\gtrsim d_1 + d_2$ if no further assumptions (e.g. sparsity) are imposed on the signals. Therefore, from a sample complexity viewpoint, our guarantees are optimal. Incidentally, to our best knowledge, all alternative approaches are either suboptimal by a polylogarithmic factor in $d_1,d_2$ or require knowing the sign pattern of one of the underlying signals
\cite{ahmed2018blind,ahmed2014blind}.

Recent algorithmic advances for blind deconvolution can be classified into two main approaches: works based on convex relaxations and those employing gradient descent on a smooth nonconvex function. The influential convex techniques of \cite{ahmed2018blind,ahmed2014blind} ``lift'' the objective
to a higher dimension, thereby necessitating the resolution of a high-dimensional semidefinite program.
The more recent work of~\cite{aghasi2017branchhull,AghAhmHanJos18} instead relaxes the feasible region in the natural parameter space, under the assumption that the coordinate signs of either $\bar w$ or $\bar x$ are known a priori.
%We note that Moreover, with the exception of~\cite{aghasi2017branchhull}, the sample complexity of these approaches is suboptimal by a polylogarithmic factor. In~\cite{aghasi2017branchhull,AghAhmHanJos18}, it is also assumed that the sign
%of one of the unknown signals is known, which can be too restrictive in some situations. In constrast, our approach only requires an upper bound on the norm of the signals. 
Finally, with the exception of~\cite{ahmed2014blind}, the aforementioned works do not provide
guarantees in the noisy regime.

Nonconvex approaches for blind deconvolution typically apply  gradient descent to a smooth formulation
of the problem~\cite{li2016rapid,ma2017implicit,HuangHand17}. Since the condition number of the problem scales with
dimension, as we mentioned previously, these works introduce a nuanced analysis that is specific to the gradient method. The authors of~\cite{li2016rapid} propose applying gradient descent on a regularized objective function, and 
identify a ``basin of attraction'' around the solution.
The paper \cite{ma2017implicit} instead analyzes gradient descent on the unregularized objective.
They use the leave-one-out technique and prove that the iterates remain within a region
where the objective function satisfies restricted strong convexity and smoothness conditions. The sample complexities of the methods in \cite{li2016rapid,ma2017implicit,HuangHand17,ma2017implicit} are optimal up to polylog factors. 

The nonconvex strategies mentioned above all use spectral methods for  initialization. These methods are not robust to outliers, since they rely on the leading
singular vectors/values of a potentially noisy measurement operator. Adapting the spectral initialization of~\cite{duchi_ruan_PR}
to bilinear inverse problems enables us to deal with gross outliers of arbitrary magnitude. Indeed, high variance noise makes it easier for our initialization to ``reject'' outlying measurements.

The outline of the paper is as follows. Section~\ref{sec:notation} records basic notation we will use throughout the paper. Section~\ref{sec:all_algos} reviews the impact of sharpness and weak convexity on the rapid convergence of numerical methods. Section~\ref{sec:ass_model} establishes estimates of weak convexity, sharpness, and Lipschitz moduli for the blind deconvolution problem under both deterministic and statistical assumptions on the data. Section~\ref{sec:damek_init} introduces the initialization procedure and proves its correctness even if a constant fraction of measurements is corrupted by gross outliers. The final Section~\ref{sec:experiments} presents numerical experiments illustrating the theoretical results in the paper.
 
\section{Notation}\label{sec:notation}

The section records basic notation that we will use throughout the paper.
To this end, we always endow $\RR^d$ with the dot product, $\dotp{x,y} = x^\top y$, and the induced norm $\|x\|_2 = \sqrt{\dotp{x,x}}$. The symbol $\SS^{d-1}$ denotes the unit sphere in $\RR^d$, while $\mathbb B$ denotes the open unit ball.  When convenient, we will use the notation ${\mathbb B}^d$ to emphasize the dimension of the ambient space. More generally, ${\mathbb B}_r(x)$ will stand for the open ball around $x$ of radius $r$. 
We define the distance and the nearest-point projection of a point $x$ onto a closed set $Q\subseteq\R^d$ by
$$\dist(x,Q) = \inf_{y \in Q} \|x - y\|_2\quad \textrm{and}\quad \proj_Q(x)=\argmin_{y\in Q}\|x-y\|_2,$$
respectively.
For any pair of real-valued functions $f, g\colon \RR^d \rightarrow \RR$, the notation $f \lesssim g $ means that there exists a positive constant $C$ such that $f(x) \leq C g(x)$ for all $x\in \R^d$. We write $f \asymp g$ if both $f \lesssim g$ and $g \lesssim f.$

We will always use the trace inner product $\langle X,Y\rangle=\tr(X^TY)$ on the space of matrices $\R^{d_1\times d_2}$. The symbols $\|A\|_\op$ 
and $\|A\|_F$ will denote the operator and Frobenius norm of $A$, respectively. 
Assuming $d\leq m$, the map $\sigma\colon \RR^{d \times m} \rightarrow \RR^{d}_+$ returns the vector of 
ordered singular values $\sigma_1(A) \geq \sigma_2(A) \geq  
\dots \geq \sigma_{d}(A)$. Note the equalities  $\|A\|_F=\|\sigma(A)\|_2$ and $\|A\|_\op =\sigma_1(A)$.

Nonsmooth functions will appear throughout this work. Consequently will use some basic constructions of generalized differentiation, as set out for example in the monographs \cite{RW98,mord1,Borwein-Zhu,penot_book}. Consider a function $f\colon \RR^d \rightarrow \R\cup\{+\infty\}$ and a point $ x$, with $f( x)$ finite. Then the \emph{Fr\'echet subdifferential} of $f$ at $\bar x$, denoted by  $\partial f(x)$, is the set of all vectors $v\in\R^d$ satisfying 
\begin{equation}\label{eqn:subgrad_defn}
f(y) \geq f(x) + \dotp{v, y- x} + o(\|y-x\|)\qquad \text{as }y\rightarrow x. 
\end{equation}
Thus, a vector $v$ lies in the subdifferential $\partial f(x)$ precisely when the function $y\mapsto f(x) + \dotp{v, y- x}$ locally minorizes $f$ up to first-order. We say that a point $x$ is 
\emph{stationary for }$f$ whenever the inclusion, $0 \in \partial f(x)$, holds. Standard results show for convex functions $f$ the subdifferential $\partial f(x)$ reduces to the subdifferential in the sense of convex analysis, while for differentiable functions $f$ it consists only of the gradient $\partial f(x)=\{\nabla f(x)\}$.

Notice that in general, the little-o term in \eqref{eqn:subgrad_defn} may depend on the base-point $x$, and the estimate \eqref{eqn:subgrad_defn} therefore may be nonuniform. In this work, we will only encounter functions whose subgradients automatically satisfy a uniform type of lower-approximation property.
We say that a function $f\colon\RR^d \rightarrow \RR\cup\{+\infty\}$ is {\em $\rho$-weakly convex}\footnote{Weakly convex functions also go by other names such as lower-$C^2$, uniformly prox-regularity, paraconvex, and semiconvex.} if the perturbed function $x \mapsto f(x) + \frac{\rho}{2}\|x\|^2_2$ is convex. It is straightforward to see that for any $\rho$-weakly convex function $f$, subgradients automatically satisfy the uniform bound:
\[f(y) \geq f(x) + \dotp{v, y-x} - \frac{\rho}{2} \|y-x\|^2_2 \qquad \forall x,y \in \RR^d,\forall v\in\partial f(x).\]
We will comment further on the class of weakly convex functions in Section~\ref{sec:all_algos}.

We say a that a random vector $X$ in $\R^d$ is {\em $\eta$-sub-gaussian} whenever $\EE \exp\left( \frac{\dotp{u, X}^2}{\eta^2} \right) \leq 2$ for all vectors $u\in  \SS^{d-1}$. The {\em sub-gaussian norm} of a real-valued random variable $X$ is defined to be $\|X\|_{\psi_2}=\inf\{t>0:\EE\exp\left( \frac{X^2}{t^2} \right) \leq 2\}$, while the {\em sub-exponential norm} is defined by $\|X\|_{\psi_1}=\inf\{t>0:\EE\exp\left( \frac{|X|}{t} \right) \leq 2\}$. Given a sample $y = (y_1, \dots, y_n),$ we will write $\text{\texttt{med}}(y)$ to denote its median.

\section{Algorithms for sharp weakly convex problems}\label{sec:all_algos}
The central thrust of this work is that  under reasonable statistical assumptions, the penalty formulation \eqref{formulation} satisfies two key properties: (1) the objective function is weakly convex and (2) grows at least linearly as one moves away from the solution set. In this section, we review the consequences of these two properties for local rapid convergence of numerical methods. The discussion mostly follows the recent work \cite{davis2018subgradient}, though elements of this viewpoint can already be seen in the two papers \cite{duchi_ruan_PR,davis2017nonsmooth} on robust phase retrieval.

Setting the stage, we introduce the following assumption.
\begin{assumption}\label{ass:aa}
	{\rm
		Consider the optimization problem,
		\begin{equation}\label{eqn:target_weak_conv}
		\min_{x\in \cX}~ f(x).
		\end{equation}
		Suppose that the following properties hold for some real $\mu,\rho>0$.
		\begin{enumerate}
			\item {\bf (Weak convexity)} The set $\cX$ is closed and convex, while the function $f\colon\R^d\to\R$ is $\rho$-weakly convex.
			%, meaning that the perturbed function $x\mapsto f(x)+\frac{\rho}{2}\|x\|^2$ is convex.
			\item {\bf (Sharpness)} The set of minimizers $\displaystyle \cX^*:=\argmin_{x\in\cX} f(x)$ is nonempty and the inequality	
			\[f(x) - \inf f \geq \mu \cdot  \dist\left(x,\cX^\ast \right) \qquad \text{holds for all } x \in \cX. \]
	\end{enumerate}}
\end{assumption}

The class of weakly convex functions is broad and its importance in optimization  is well documented \cite{fav_C2,prox_reg,Nurminskii1973,paraconvex,semiconcave}. It trivially includes all convex functions and all $C^1$-smooth functions with Lipschitz gradient. More broadly, it includes all  compositions $$f(x)=h(F(x)),$$ where $h(\cdot)$ is convex and $L$-Lipschitz, and $F(\cdot)$ is $C^1$-smooth with $\beta$-Lipschitz Jacobian. Indeed then the composite function $f=h\circ F$ is weakly convex with parameter $\rho=L\beta$; see e.g. \cite[Lemma 4.2]{comp_DP}. In particular, our target problem \eqref{formulation} is clearly weakly convex, being a composition of the $\ell_1$ norm and a quadratic map. The estimate $\rho=L\beta$ on the weak convexity constant is often much too pessimistic, however. Indeed, under statistical assumptions, we will see that the target problem \eqref{formulation} has a much better weak convexity constant. 
The notion of sharpness, and the related error bound property, is now ubiquitous in nonlinear optimization. Indeed, sharpness underlies much of perturbation theory and rapid convergence guarantees of various numerical methods.
For a systematic treatment of error bounds and their applications, we refer the  reader to the monographs of Dontchev-Rockafellar \cite{imp} and Ioffe \cite{ioffe_book}, and the article of Lewis-Pang \cite{MR1646951}.

Taken together, weak convexity and sharpness provide an appealing framework for deriving local rapid convergence guarantees for numerical methods. In this work, we specifically focus on two such procedures: the subgradient and prox-linear algorithms. To this end, we aim to estimate both the radius of rapid converge around the solution set and the rate of convergence. Our ultimate goal is to show that when specialized to our target problem  \eqref{formulation}, with high probability, both of these quantities are independent of the ambient dimensions $d_1$ and $d_2$ as soon as the number of measurements is sufficiently large.

Both the subgradient and prox-linear algorithms have the property that when initialized at a stationary point of the problem, they could stay there for all subsequent iterations.  Since we are interested in finding global minima, and not just stationary points, we must therefore estimate the neighborhood of the solution set that has no extraneous stationary points. This is the content of the following simple lemma \cite[Lemma 3.1]{davis2018subgradient}.

\begin{lem}\label{lem:no_extr_stat}
	Suppose that Assumption~\ref{ass:aa} holds.
	Then the problem \eqref{eqn:target_weak_conv} has no stationary points $x$ satisfying 
	$$0<\dist(x;\cX^*)<\frac{2\mu}{\rho}.$$
\end{lem}
\begin{proof}
	Fix a critical point $x\in \cX\notin\cX^*$. Letting $x^*:=\proj_{\cX^*}(x)$, we deduce $\mu\cdot\dist(x,\cX^*)\leq  f(x)-f(x^*)\leq \frac{\rho}{2}\cdot\|x-x^*\|^2=\frac{\rho}{2}\cdot\dist^2(x,\cX^*)$. Dividing  by $\dist(x,\cX^*)$, the result follows.  
\end{proof}

The estimate $\frac{2\mu}{\rho}$ of the radius in Lemma~\ref{lem:no_extr_stat} is tight. To see this, consider minimizing the univariate function $f(x)=|\lambda^2x^2-1|$ on the real line $\cX=\R$. Observe that the set of minimizers is $\cX^*=\left\{\pm \tfrac{1}{\lambda}\right\}$, while $x=0$ is always an extraneous stationary point. A quick computation shows that the smallest valid weak convexity is $\rho=2\lambda^2$ while the largest valid sharpness constant is $\mu=\lambda$. 

We therefore deduce 
$\dist(0,\cX^*)=\frac{1}{\lambda}=\frac{2\mu}{\rho}$. Hence the radius of the region $\frac{2\mu}{\rho}$ that is devoid of extraneous stationary points is tight.

In light of Lemma~\ref{lem:no_extr_stat}, let us define
for any $\gamma>0$ the tube 
\begin{align}\label{eq:tube_region}
\cT_{\gamma} := \left\{ z \in \RR^{d} \colon \dist(z, \cX^\ast) \leq \gamma \cdot \frac{\mu}{\rho}\right\}.
\end{align}
Thus we would like to search for algorithms whose basin of attraction is a tube $\cT_{\gamma}$ for some numerical constant $\gamma>0$. Due to the above discussion, such a basin of attraction is in essence optimal.

We next discuss two rapidly converging algorithms. The first is the Polyak subgradient method, outlined in Algorithm~\ref{alg:polyak}. Notice that the only parameter that is needed to implement the procedure is the minimal value of the problem \eqref{eqn:target_weak_conv}. This value is sometimes known; case in point, the minimal value of the   
penalty formulation \eqref{formulation} is zero when the bilinear measurements are exact.

\smallskip
\begin{algorithm}[H]
	\KwData{$x_0 \in \RR^d$}
	
	{\bf Step $k$:} ($k\geq 0$)\\
	$\qquad$ Choose $\zeta_k \in \partial f(x_k)$. {\bf If} $\zeta_k=0$, then exit algorithm.\\
	$\qquad$ Set $\displaystyle x_{k+1}=\proj_{\cX}\left(x_{k} - \frac{f(x_k)-\min_{\mathcal{X}} f}{\|\zeta_k\|^2}\zeta_k\right)$.

	\caption{Polyak Subgradient Method}
	\label{alg:polyak}
\end{algorithm}
\smallskip

The rate of convergence of the method relies on the Lipschitz constant and the condition measure:
\begin{equation*}
L:=\sup\{\|\zeta\|:\zeta\in \partial f(x),x\in \cT_1\}\qquad \textrm{and} \qquad \tau:=\frac{\mu}{L}.
\end{equation*}
A straightforward argument \cite[Lemma 3.2]{davis2018subgradient} shows $\tau\in [0,1]$.
The following theorem appears as \cite[Theorem 4.1]{davis2018subgradient}, while its application to phase retrieval was investigated in \cite{davis2017nonsmooth}.

\begin{thm}[Polyak subgradient method]\label{thm:qlinear}
	Suppose that Assumption~\ref{ass:aa} holds and fix a real $\gamma \in (0,1)$. %and define  the tube 
	%	$$\mathcal{T}:=\left\{x\in \R^d: \dist(x;S)\leq \gamma\cdot\frac{\mu}{\rho}\right\},$$
	%	and the corresponding Lipschitz constant 
	%	$$\displaystyle L_g:=\sup_{x\in \mathcal{T},\,\zeta\in \partial g(x) } \|\zeta\|.$$
	Then  Algorithm~\ref{alg:polyak} initialized at any point $x_0\in \mathcal{T}_{\gamma}$ produces iterates that 
	converge $Q$-linearly to $\cX^*$, that is 
	\begin{align}\label{eqn:lin_rate1}
	\dist^2(x_{k+1},\cX^*) \leq \left(1-(1-\gamma) \tau^2\right)\dist^2(x_{k},\cX^*)\qquad \forall k\geq 0.
	\end{align}
\end{thm}

When the minimal value of the problem \eqref{eqn:target_weak_conv} is unknown, there is a straightforward modification of the subgradient method that converges R-linearly. The idea is to choose a geometrically decaying control sequence for the stepsize. The disadvantage is that the convergence guarantees rely on being able to tune estimates of $L$, $\rho$, and $\mu$.

\smallskip
\begin{algorithm}[H]
	\KwData{Real $\lambda>0$ and $q\in (0,1)$.}
	
	{\bf Step $k$:}  $(k\geq 0)$ \\
	\qquad Choose $\zeta_k \in \partial g(x_k)$. {\bf If} $\zeta_k=0$, then exit algorithm.
	
	\qquad Set stepsize $\alpha_{k} = \lambda\cdot q^{k}$.\\
	\qquad Update iterate  $x_{k+1}=\proj_{\cX}\left(x_{k} - \alpha_k \frac{\zeta_k}{\norm{\zeta_k}}\right)$.\\
	
	\caption{Subgradient method with
		geometrically decreasing stepsize }
	\label{alg:geometrically_step}
\end{algorithm}
\smallskip

The following theorem appears as \cite[Theorem 6.1]{davis2018subgradient}. The convex version of the result dates back to Goffin \cite{goffin}.

\begin{thm}[Geometrically decaying subgradient method] \label{thm:geometric} Suppose that Assumption~\ref{ass:aa} holds, fix a real  $\gamma \in (0,1)$,  and suppose  $\tau  \le
	\sqrt{ \frac{1}{2-\gamma} }$. 
	Set 
	$\lambda:=\frac{\gamma \mu^2}{\rho L} \textrm{ and } q:=\sqrt{1-(1-\gamma) \tau^2}.$
	Then the iterates $x_k$ generated by
	Algorithm~\ref{alg:geometrically_step}, initialized at a
	point $x_0 \in \mathcal{T}_{\gamma}$, satisfy:
	\begin{equation} \label{eq:geometric_rate}
	\dist^2(x_k;\cX^*) \leq \frac{\gamma^2 \mu^2}{\rho^2}
	\left(1-(1-\gamma)\tau^2\right)^{k}\qquad \forall k\geq 0.
	\end{equation}
\end{thm}

Notice that both subgradient algorithms~\ref{alg:polyak} and \ref{alg:geometrically_step} are at best locally linearly convergent, with a relatively cheap per-iteration cost. As the last example we discuss an algorithm that is specifically designed for convex compositions, which is locally {\em quadratically convergent}. The caveat is that the method may have a high per-iteration cost, since in each iteration one must solve an auxiliary convex optimization problem.

Setting the stage, let us introduce the following assumption.
\begin{assumption}\label{ass:bb}
	{\rm
		Consider the optimization problem,
		\begin{equation}\label{eqn:target_comp}
		\min_{x\in \cX} f(x):=h(F(x)).
		\end{equation}
		Suppose that the following properties holds for some real $\mu,\rho>0$.
		\begin{enumerate}
			\item {\bf (Convexity and smoothness)} The function $h(\cdot)$ and the set $\cX$ are convex and $F(\cdot)$ is differentiable.
			\item {\bf (Approximation accuracy)} The convex models $f_x(y):=h(F(x)+\nabla F(x)(y-x))$ satisfy the estimate:
			$$|f(y)-f_x(y)|\leq \frac{\rho}{2}\|y-x\|^2_2\qquad \forall x,y\in \cX.$$  
			\item {\bf (Sharpness)} The set of minimizers $\displaystyle \cX^*:=\argmin_{x\in\cX} f(x)$ is nonempty and the inequality	
			\[f(x) - \inf f \geq \mu \cdot  \dist\left(x,\cX^\ast \right) \qquad \text{holds for all } x \in \cX. \]
	\end{enumerate}}
\end{assumption}

It is straightforward to see that Assumption~\ref{ass:bb} implies that $f$ is $\rho$-weakly convex; see e.g. \cite[Lemma 7.3]{comp_DP}. Therefore Assumption~\ref{ass:bb} implies Assumption~\ref{ass:aa}.

Algorithm~\ref{alg:prox_lin} describes the prox-linear method---a close variant of Gauss-Newton. For a historical account of the prox-linear method, see e.g.,  \cite{burke_com, prox, comp_DP} and the references therein.

\smallskip
\begin{algorithm}[H]
	\KwData{Initial point $x_0 \in \RR^{d}$, proximal parameter $\beta>0$.}
	
	{\bf Step $k$:}  $(k\geq 0)$ \\
	\qquad Set  $\displaystyle x_{k+1} \leftarrow \argmin_{x \in \cX} \left\{h\left(F(x_k)+\nabla F(x_k)(x-x_k)\right) + \frac{\beta}{2} \|x-x_k\|^2\right\}.$
	
	\caption{Prox-linear algorithm}
	\label{alg:prox_lin}
\end{algorithm}
\smallskip

The following theorem proves that under Assumption~\ref{ass:bb}, the prox-linear method converges quadratically, when initialized sufficiently close to the solution set. Guarantees of this type have appeared, for example, in \cite{duchi_ruan_PR,prox_error,quad_conv,prox_error}. For the sake of completeness, we provide a quick argument.

\begin{thm}[Prox-linear algorithm]
	Suppose Assumption~\ref{ass:bb} holds. Choose any $\beta \geq \rho$ and set $\gamma:=\rho/\beta$. Then Algorithm~\ref{alg:prox_lin} initialized at any point $x_0 \in \cT_{\gamma}$ converges quadratically:
	$$\dist(x_{k+1},\cX^*)\leq  \tfrac{\beta}{\mu}\cdot\dist^2(x_{k},\cX^*)\qquad \forall k\geq 0.$$
	
\end{thm}
\begin{proof}
	Consider an iterate $x_k$ and choose any $x^*\in \proj_{\cX^*}(x_k)$. Taking into account that the function $x\mapsto f_{x_k}(x)+\frac{\beta}{2}\|x-x_k\|^2$ is strongly convex and $x_{k+1}$ is its minimizer, we  deduce
	$$\left(f_{x_k}(x_{k+1})+\frac{\beta}{2}\|x_{k+1}-x_k\|^2\right)+\frac{\beta}{2}\|x_{k+1}-x^*\|^2\leq f_{x_k}(x^*)+\frac{\beta}{2}\|x^*-x_k\|^2.$$
	Using Assumption~\ref{ass:bb}.2, we therefore obtain 
	$$f(x_{k+1})+\frac{\beta}{2}\|x_{k+1}-x^*\|^2\leq f(x^*)+\beta\|x^*-x_k\|^2.$$
	Rearranging and using sharpness (Assumption~\ref{ass:bb}.3), we conclude 
	$$\mu\cdot\dist(x_{k+1},\cX^*)\leq f(x_{k+1})-f(x^*)\leq \beta\cdot\dist^2(x_{k},\cX^*),$$
	as claimed. 
\end{proof}

\section{Assumptions and Models}\label{sec:ass_model}
In this section, we aim to interpret the efficiency of the subgradient and prox-linear algorithms discussed in Section~\ref{sec:all_algos}, when applied to our target problem \eqref{formulation}. To this end, we must estimate the three parameters $\rho,\mu, L>0$. These quantities control both the size of the attraction neighborhood around the optimal solution set and the rate of convergence within the neighborhood. In particular, we will show that these quantities are independent of the ambient dimension $d_1,d_2$ under natural assumptions on the data generating mechanism.

It will be convenient for the time being to abstract away from  the formulation \eqref{formulation}, and instead consider the function
$$~ g(w,x):=\frac{1}{m}\|\cA(wx^\top) - y\|_1,$$
where $\mathcal{A}\colon\R^{d_1\times d_2}\to\R^m$ is an arbitrary linear map and $y\in \R^m$ is an arbitrary vector. The formulation~\eqref{formulation} corresponds to the particular linear map $\cA(X)=(\ell_i^\top Xr_i)_{i=1}^m$.
Since we will be interested in the prox-linear method, let us define the  convex model
$$g_{(w,x)}(\hat w,\hat x):=\frac{1}{m}\|\cA(wx^\top+w(\hat x-x)^\top+(\hat w-w)x^\top)-y\|_1.$$
Our strategy is as follows.  Section~\ref{sec:determ_prop} identifies  deterministic assumptions on the data, $\mathcal{A}$ and $y$, that yield favorable estimates of $\rho,\mu, L>0$. Then Section~\ref{sec:generative_model} shows that these deterministic assumptions hold with high probability under natural statistical assumptions on the data generating mechanism.

\subsection{Favorable Deterministic Properties}\label{sec:determ_prop}
The following property, widely used in the literature, will play a central role in our analysis.

\begin{assumption}[Restricted Isometry Property (RIP)]\label{assump:RIP}
	There exist constants $c_1, c_2 > 0$ such that for all matrices $X \in \RR^{d_1 \times d_2}$ of rank at most two the following bound holds:
	\begin{equation*}%\label{eq:RIP}
	c_1 \|X\|_F \leq \frac{1}{m}\|\cA(X)\|_1 \leq c_2 \|X\|_F.
	\end{equation*}
\end{assumption}

The following proposition estimates the two constants $\rho$ and $L$, governing the performance of the subgradient and prox-linear methods under Assumption~\ref{assump:RIP}.
\begin{proposition}[Approximation accuracy and Lipschitz continuity] {\hfill \\ }
	Suppose Assumption~\ref{assump:RIP} holds and let $K>0$ be arbitrary.
	Then the following estimates hold:
	\begin{align*}
	|g(\hat w, \hat x) - g_{(w, x)}(\hat w, \hat x)| &\leq \frac{c_2}{2}\cdot \|(w,x) - (\hat w,\hat x)\|^2_2  \qquad \forall x,\hat x \in \RR^{d_1},\forall w,\hat w \in \RR^{d_2},\\
	|g(w,x)-g(\hat w,\hat x)|&\leq \sqrt{2}c_2 K\cdot \|(w,x)-(\hat w,\hat x)\|_2\qquad \forall x,\hat x\in K\mathbb{B},w,\hat w\in K\mathbb{B}.
	\end{align*}
	
\end{proposition}
\begin{proof}
	To see the first estimate, observe
	\begin{align*}
	|g(\hat w, \hat x) - g_{(w, x)}(\hat w, \hat x)|&=\left|\frac{1}{m}\left\|\cA(\hat w\hat x^\top) - y \right\|_1 - \frac{1}{m}\left\|\cA(wx^\top+w(\hat x-x)^\top+(\hat w-w)x^\top) - y\right\|_1 \right|\notag \\
	&\leq \frac{1}{m}\left\| \cA(\hat w\hat x^\top -wx^\top-w(\hat x-x)^\top-(\hat w-w)x^\top)\right\|_1  \notag\\
	&= \frac{1}{m} \left \| \cA\big((w-\hat w)(x-\hat x)^\top \big) \right\|_1\notag\\
	&\leq  c_2 \left\| (w-\hat w)(x-\hat x)^\top \right\|_F  \\
	&\leq \frac{c_2}{2} \left(\|w-\hat w\|_2^2 + \|x - \hat x\|_2^2\right),
	\end{align*}
	where the last estimate follows from Young's inequality $2ab \leq a^2+b^2.$
	Now suppose  $w,\hat w\in K\mathbb{B}$ and $x,\hat x\in K\mathbb{B}$. 
	We then successively compute:
	\begin{align*}
	|g(w,x)-g(\hat w,\hat x)|
	\leq \frac{1}{m}\|\cA(wx^\top-\hat w\hat x^\top)\|_1
	&\leq c_2\|wx^\top-\hat w\hat x^\top\|_F\\
	&=c_2\|(w-\hat w)x^\top+\hat w(x-\hat x)^\top\|_F\\
	&\leq c_2\|x\|_2\|w-\hat w\|_2+c_2\|\hat w\|_2\|x-\hat x\|_2&\\
	&\leq \sqrt{2}c_2K\cdot \|(w,x)-(\hat w,\hat x)\|_2.
	\end{align*}
	The proof is complete.
\end{proof}

We next move on to estimates of the sharpness constant $\mu$. To this end, consider two vectors $\bar w\in \R^d_1$ and $\bar x\in\R^{d_2}$, and set $M:=\|\bar x\bar w^T\|_F=\|\bar x\|_2\cdot\|\bar w^T\|_2$. Without loss of generality, henceforth, we suppose $\|\bar w\|_2=\|\bar x\|_2$. Our estimates on the sharpness constant will be valid only on bounded sets. Consequently, define the two sets:
\begin{align*}%\label{eq:solution_main_problem}
\cS_{\nu} :=  \nu \sqrt{M}\cdot (\mathbb{B}^{d_1}\times\mathbb{B}^{d_2}),\qquad
\cS^\ast_{\nu} := \{ (\alpha \bar w, (1/\alpha) \bar x) \colon 1/\nu \leq |\alpha| \leq \nu\}.
\end{align*}
The set $\cS_{\nu}$ simply encodes a bounded region, while $\cS^\ast_{\nu}$
encodes all rank-1 factorizations of the matrix $\bar w\bar x^\top$ with  bounded factors.
We begin with the following proposition, which analyzes the sharpness properties of the idealized function
$$(x,w)\mapsto \|wx^\top-\bar w{\bar x}^\top\|_F.$$
The proof is quite long, and therefore we have placed it in Appendix~\ref{sec:appendix_sharp_baby_function}.

\begin{proposition}\label{prop:l2sharpness}
	For any $\nu \geq 1$, we have the following  bound
	\begin{align*}
	\|w x^\top - \wbar \xbar^\top\|_F  \geq \frac{\sqrt{M}}{2\sqrt{2}(\nu+1)} \dist \big((w,x), \cS^\ast_\nu\big) \qquad \text{for all $(w, x) \in \cS_\nu$}.
	\end{align*}
\end{proposition}

Thus the function $(x,w)\mapsto \|wx^\top-\bar w{\bar x}^\top\|_F$ is sharp on the set $\cS_\nu$ with coefficient $\frac{\sqrt{M}}{2\sqrt{2}(\nu+1)}$. We note in passing that the analogue of Proposition~\ref{prop:l2sharpness} for symmetric matrices was proved in \cite[Lemma 5.4]{proc_flow}.

The sharpness of the loss $g(\cdot,\cdot)$ in the noiseless regime (i.e. when $y=\cA(\bar w\bar x^\top)$) is now immediate.  
\begin{proposition}[Sharpness in the Noiseless Regime]\label{prop:noiseless_sharp} 
	Suppose  that Assumption~\ref{assump:RIP} holds and that equality, $y =\cA(\bar w\bar x^\top)$, holds. Then for any $\nu\geq1$, we have the following bound:
	\begin{align*}
	g(w,x) - g(\bar w, \bar x)  \geq  \frac{c_1\sqrt{M}}{2\sqrt{2}(\nu+1)} \dist \big((w,x), \cS^\ast_\nu\big) \qquad \text{for all $(w, x) \in \cS_\nu$}.
	\end{align*}
\end{proposition}
\begin{proof}
	Using Assumption~\ref{assump:RIP} and Proposition~\ref{prop:l2sharpness}, we deduce
	$$g(w,x) - g(\bar w, \bar x)=\frac{1}{m}\|\cA(wx^\top-\bar w\bar x^\top)\|_1\geq c_1\|wx^\top-\bar w\bar x^\top\|_F\geq \frac{c_1\sqrt{M}}{2\sqrt{2}(\nu+1)} \dist \big((w,x), \cS^\ast_\nu\big),$$
	as claimed.
\end{proof}

Sharpness in the noisy case requires an additional assumption. We record it below. Henceforth, for any set $\mathcal{I}$, we define the restricted linear map $\cA_{\mathcal{I}}\colon\R^{d_1\times d_2}\to\R^{|\mathcal{I}|}$ by setting $\cA_{\mathcal{I}}(X):=(\cA(X))_{i\in \mathcal{I}}$.

\begin{assumption}[$\cI$-outliner bounds]\label{assump:outlier}
	There exists a set $\cI \subset\{1, \ldots, m\}$, vectors $\bar w\in\R^{d_1}$, $\bar x\in\R^{d_2}$, and a constant $c_3 > 0$ such that the following hold.
	\begin{enumerate}[label = $\mathrm{(C\arabic*)}$]
		\item Equality $y_i=\cA(\bar w\bar x^\top)_i$ holds for all $i\notin\cI$.
		\item \label{item:assump:rip_outliers}  For all matrices $X \in \RR^{d_1 \times d_2}$ of rank at most two, we have
		\begin{equation}\label{eq:Upper_Bound}
		c_3\|X\|_F \leq  \frac{1}{m}\|\cA_{\mathcal{I}^c}(X)\|_1 -  \frac{1}{m} \|\cA_{\cI}(X)\|_1.
		\end{equation}
	\end{enumerate}
\end{assumption}

Combining Assumption~\ref{assump:outlier}  with Proposition~\ref{prop:l2sharpness} quickly yields sharpness of the objective even in the noisy setting.

\begin{proposition}[Sharpness in the noisy regime]\label{prop:noisy_sharp} 
	Suppose that Assumption~\ref{assump:outlier} holds.  Then 
	\begin{align*}
	g(w,x) - g(\bar w, \bar x)  \geq  \frac{c_3\sqrt{M}}{2\sqrt{2}(\nu+1)} \dist \big((w,x), \cS^\ast_{\nu}\big) \qquad \text{for all $(w, x) \in \cS_\nu$}.
	\end{align*}
\end{proposition}
\begin{proof}
	Defining $\eta = \cA(\bar w \bar x^T) - y$, we have the following bound: 
	\begin{align*}
	&g(w,x) - g(\bar w, \bar x)\\  & \hspace*{1cm}= 
	\frac{1}{m}\left(\|\cA\left(w x^\top - \bar w \bar x^\top \right)  + \eta \|_1 
	- 
	\| \eta \|_1 \right) \\
	&\hspace*{1cm} = \frac{1}{m} \left( \|\cA( wx^\top - \bar w \bar x^\top)\|_1 + 
	\sum_{i \in \mathcal{I}} \left(\left|\left(\cA(w x^\top - \bar w \bar x^\top)\right)_i 
	+\eta_i\right|  -\left|\left(\cA(w x^\top - \bar w \bar x^\top)\right)_i\right| 
	- |\eta_i|\right)\right)\\
	& \hspace*{1cm}\geq \frac{1}{m} \left( \|\cA( wx^\top - \bar w \bar x^\top)\|_1 
	-2\sum_{i \in \mathcal{I}} \left|\left(\cA(w x^\top - \bar w \bar 
	x^\top)\right)_i\right| \right)\\
	&\hspace*{1cm} = \frac{1}{m}\sum_{i \in \mathcal{I}^c} \left|\left(\cA(w x^\top - \bar w 
	\bar x^\top)\right)_i\right| -\frac{1}{m}\sum_{i \in \mathcal{I}} \left|\left(\cA(w x^\top - \bar 
	w \bar x^\top)\right)_i\right|\\
	&\hspace*{1cm}\geq c_3\|w x^\top - \bar w \bar x^\top\|_F \geq\frac{c_3\sqrt{M}}{2\sqrt{2}(\nu+1)}\dist \big((w,x), \cS^\ast_{\nu}\big),
	\end{align*}
	where the first inequality follows by the reverse triangle inequality, the second inequality follows by Assumption~\ref{item:assump:rip_outliers}, and the final inequality follows from Proposition~\ref{prop:l2sharpness}. The proof is complete.
\end{proof}

To summarize, suppose Assumptions~\ref{assump:RIP} and \ref{assump:outlier} are valid. Then in the notation of Section~\ref{sec:all_algos} we may set:
$$\boxed{\rho=c_2,\quad L=c_2\nu\sqrt{2M},\quad \mu=\frac{c_3\sqrt{M}}{2\sqrt{2}(\nu+1)}.}$$
Consequently, the tube radius of $\mathcal{T}_{1}$ is $\frac{2\mu}{\rho}=\frac{c_3}{c_2}\cdot\frac{\sqrt{M}}{\sqrt{2}(\nu+1)}$ and the the linear convergence rate of the subgradient method is governed by $\tau=\frac{\mu}{L}=\frac{c_3}{c_2}\cdot\frac{1}{4(\nu+1)^2}$.
In particular,  the local search algorithms must be initialized at a point $(x,w)$, whose relative distance to the solution set $\frac{\dist((x,w),\mathcal{S}^*_{\nu})}{\sqrt{\|\bar x\bar w^\top\|_F}}$ is upper bounded by a constant. We record this conclusion below.

\begin{corollary}[Convergence guarantees]\label{cor:generic_conv_iso}
	Suppose Assumptions~\ref{assump:RIP} and \ref{assump:outlier} are valid, and consider the optimization problem 
	$$\min_{(x,w)\in \cS_\nu}~ g(w,x)=\frac{1}{m}\|\cA(wx^\top) - y\|_1.$$
	Choose any pair $(x_0,y_0)$  satisfying
	$$\frac{\dist((w_0,x_0),\cS^\ast_{\nu})}{\sqrt{\|\bar w\bar x^\top\|_F}}\leq \frac{c_3}{4\sqrt{2}c_2(\nu+1)}.$$
	Then the following are true.
	\begin{enumerate}
		\item {\bf (Polyak subgradient)} Algorithm~\ref{alg:polyak} initialized  $(x_0,y_0)$  produces iterates that 
		converge linearly to $\cS^\ast_{\nu}$, that is 
		\begin{equation*}
		\frac{\dist^2((w_{k},x_{k}),\cS^\ast_{\nu})}{\|\bar w\bar x^\top\|_F}\leq \left(1-\frac{c_3^2}{32c_2^2(\nu+1)^4}\right)^{k}\cdot \frac{c_3^2}{32c_2^2(\nu+1)^2}\qquad \forall k\geq 0.%\left(1  - \frac{\rho\mu\|x_k + \bar x \|}{\|\zeta_k\|^2}\left( \frac{\kappa}{\rho}\|x_k + \bar x\| - \|x_k - \bar x\| \right)\right)
		\end{equation*}
		\item {\bf (geometric subgradient)}
		Set 
		$\lambda:=\frac{c_3^2\sqrt{\|\bar w\bar x^\top\|_F}}{16\sqrt{2}c_2^2\nu(\nu+1)^2} \textrm{ and } q:=\sqrt{1-\frac{c_3^2}{32c_2^2(\nu+1)^4}}.$
		Then the iterates $x_k$ generated by Algorithm~\ref{alg:geometrically_step}, initialized at $(w_0,x_0)$
		converge linearly: 		
		\begin{equation*} 
		\frac{\dist^2((w_{k},x_{k}),\cS^\ast_{\nu})}{\|\bar w\bar x^\top\|_F} \leq 
		\left(1-\frac{c_3^2}{32c_2^2(\nu+1)^4}\right)^{k}\cdot\frac{c_3^2}{32c_2^2(\nu+1)^2}\qquad \forall k\geq 0.
		\end{equation*}
		\item {\bf (prox-linear)}  Algorithm~\ref{alg:prox_lin} with $\beta = \rho$ and initialized at $(w_0,x_0)$ converges quadratically:
		$$\frac{\dist((w_k,x_{k}),\cS^\ast_{\nu}))}{\sqrt{\|\bar w\bar x^\top\|_F}}\leq 2^{-2^{k}}\cdot \frac{c_3}{2\sqrt{2}c_2(\nu+1)}\qquad \forall k\geq 0.$$
	\end{enumerate}
\end{corollary}

\subsection{Assumptions under generative models}\label{sec:generative_model}
In this section, we present natural generative models under which  Assumptions~\ref{assump:RIP} and~\ref{assump:outlier} are guaranteed to hold.
Recall that at the high level, we aim to recover the pair of signals $(\bar w, \bar x)$ based on given corrupted bilinear measurements $y$. Formally, let us fix two disjoint sets $\cIi\subseteq [m]$ and  $\cIo \subseteq [m]$, called the {\em inlier} and {\em outlier} sets. Intuitively, the index set $\cIi$ encodes exact measurements while $\cIo$ encodes measurements that have been replaced by gross outliers. Define the corruption frequency $\pfail:=\frac{|\cIo|}{m}$; henceforth, we will suppose $\pfail \in [0, 1/2).$
Then for an arbitrary, potentially random sequence $\{\xi_i\}_{i =1}^m$, we consider the measurement model:
\begin{equation}\label{eqn:meas_model}
\begin{aligned}
y_i := \begin{cases}
\dotp{\ell_i, \bar w}\dotp{r_i, \bar x}  & \text{if $i \in \cIi$,}\\
\xi_i &\text{if $i \in \cIo$}.
\end{cases}
\end{aligned}
\end{equation}
In accordance with the previous section, we define the linear map $\cA\colon\R^{d_1\times d_2}\to\R^m$ by $\cA(X)=(\ell_i^\top X r_i )_{i=1}^m$. To simplify notation, we let $L\in \R^{m\times d_1}$ denote the matrix whose rows, in column form, are $\ell_i$ and we let $R\in \R^{m\times d_2}$ denote the matrix whose rows are $r_i$. 
Note that we make no assumptions about the nature of $\xi_i$. In particular, $\xi_i$ can even encode exact measurements for a different signal.

We focus on two measurement matrix models. The first model requires both matrices $L$ and $R$ to be random. For simplicity, the reader may assume both are Gaussian with i.i.d.\ entries, though the results of this paper extend beyond this case. The second model allows semi-deterministic matrices, namely deterministic $L$ and Gaussian $R$ with i.i.d. entries. In the later parts of the paper, we will put further incoherence assumptions on the deterministic matrix $L$.

\paragraph{Random matrix models.} 
\begin{enumerate}[label = \textbf{M\arabic*}]
	\item \label{MModel:1} 	The vectors $\ell_i$ and $r_i$ are i.i.d. realizations of  $\eta$-sub-gaussian random vectors $\ell\in\R^{d_1}$ and $r\in \R^{d_2}$, respectively. Suppose moreover that $\ell$ and $r$ are independent  and satisfy the nondegeneracy condition,
	\begin{equation}\label{eqn:nondeg}
	\inf_{\substack{X:~ \rank\, X \leq 2\\ \|X\|_F = 1}} \PP(|\ell^\top X r| \geq \mu_0) \geq p_0,
	\end{equation}
	for some real $\mu_0, p_0>0$.  
	\item \label{MModel:2} The matrix $L$ is arbitrary and the matrix $R$ is standard Gaussian.
\end{enumerate}

Some comments are in order. The model \ref{MModel:1} is fully stochastic, in the sense that $\ell_i$ and $r_i$ are generated by independent sub-gaussian random vectors. The nondegeneracy condition \eqref{eqn:nondeg} essentially asserts that with positive probability, the products $\ell^\top X r$ are non-negligible, uniformly over all unit norm rank two matrices $X$. 
In particular, the following example shows that  Gaussian matrices with i.i.d.\ entries are admissible under Model \ref{MModel:1}.  In contrast, the model \ref{MModel:2} is semi-stochastic: it allows $L$ to be deterministic, while making the stronger assumption that $R$ is Gaussian.

\begin{example}[Gaussian Matrices Satisfy  Model \ref{MModel:1}]
	{\rm 
		Assume that $\ell$ and $r$ are standard Gaussian random vectors in $\RR^{d_1}$ and $\RR^{d_2}$, respectively. We claim this setting is admissible under \ref{MModel:1}. To see this, fix a rank 2 matrix $X$ having unit Frobenius norm. Consider now a singular value decomposition $X = \sigma_1 u_1 v_1^\top + \sigma_2 u_2 v_2^\top$, and note the equality, $\sigma_1^2+\sigma_2^2=1$. 
		For each index $i=1,2$ define $a_i:=\dotp{\ell,u_i}$ and $b_i:= \dotp{v_i,r}$.
		Then clearly  $a_1,a_2,b_1,b_2$ are i.i.d. standard Gaussian; see e.g. \cite[Exercise 3.3.6]{vershynin2016high}. Thus, for any $c \geq 0$, we compute
		\begin{align*}
		\PP(|\ell^\top X r| \geq c ) &=   \PP(|\sigma_1 a_1 b_1 + \sigma_2 a_2 b_2| \geq c )=   \EE \left(\PP(|\sigma_1 a_1 b_1 + \sigma_2 a_2 b_2| \geq c \mid a_1, a_2 ) \right).
		\end{align*}
		Notice that conditioned on $a_1,a_2$, we have $\sigma_1 a_1 b_1 + \sigma_2 a_2 b_2\sim \normal(0,(\sigma_1 a_1)^2  + (\sigma_2 a_2)^2)$. Thus letting $z$ be a standard normal, we have 
		\begin{align*}
		\PP(|\ell^\top X r| \geq c ) 
		&=  \EE \left(\PP(\sqrt{(\sigma_1 a_1)^2  + (\sigma_2 a_2)^2} |z| \geq c \mid a_1, a_2 ) \right)\\
		&=  \PP(\sqrt{(\sigma_1 a_1)^2  + (\sigma_2 a_2)^2} |z| \geq c )\\
		&\geq  \PP(\sigma_1   |a_1z| \geq c )
		\geq  \PP(|a_1 z| \geq \sqrt{2} c ).
		\end{align*}
		Therefore, we may simply set $\mu_0=\textrm{median}(|a_1z|)/\sqrt{2}$  and  $p_0=\tfrac{1}{2}$.
	}
\end{example}

\subsubsection{Assumptions \ref{assump:RIP} and \ref{assump:outlier} under Model~\ref{MModel:1}}
In this section, we aim to prove the following theorem, which shows validity of Assumptions~\ref{assump:RIP} and \ref{assump:outlier} under \ref{MModel:1}, with high probability.  %\damek{go back and make dependence on $\eta, \mu, p_0$ clear, or just say the dependence in the proof.}

\begin{thm}[Measurement Model~\ref{MModel:1}]\label{theo:RIP}
	Consider a set  $\cI \subseteq \{1, \ldots, m\}$ satisfying $|\cI | < m/2$. Then there exist constants $c_1,c_2,c_3,c_4,c_5,c_6>0$ depending only on $\mu_0, p_0, \eta$ such that the following holds. As long as 
	$m\geq  \frac{c_1(d_1+d_2+1)}{\left(1-2|\cI|/m\right)^2}\ln\left(c_2+\frac{c_2}{1-2|\cI|/m}\right)$, then with probability at least $1-4\exp\left(-c_3\left(1-2|\cI|/m\right)^2m\right)$, every matrix $X\in\R^{d_1\times d_2}$ of rank at most two satisfies 
	\begin{equation}\label{eqn:RIPM1}
	c_4\|X\|_F\leq  \frac{1}{m}\|\cA(X)\|_1\leq c_5\|X\|_F,
	\end{equation}
	and
	\begin{equation}\label{eqn:oulierM1}
	\frac{1}{m}\left[\|\cA_{ \cI^c }(X)\|_1 - \|\cA_{ \cI }(X)\|_1\right]\geq c_6\left(1- \frac{2|\cI|}{m}\right)\|X\|_F.
	\end{equation}
	
\end{thm}

Due to scale invariance, in the proof we only concern ourselves with matrices $X$ of rank at most two satisfying $\|X\|_F=1$. Let us fix such a matrix $X$ and an arbitrary index set $\cI\subseteq\{1,\ldots,m\}$ with $|\cI|<m/2$. We begin with the following lemma.

\begin{lemma}[Pointwise concentration]\label{lem:exp_rip}
	The random variable $|\ell^\top X r|$ is sub-exponential with parameter $\sqrt{2}\eta^2$. Consequently, the estimate holds:
	\begin{equation}\label{eqn:expec_bound}
	\mu_0 p_0 \leq \EE |\ell^\top X r|  \lesssim \eta^2.
	\end{equation}
	Moreover, there exists a numerical constant $c>0$ such that for any $t\in (0,\sqrt{2}\eta^2]$, we have with probability at least $1- 2\exp(-\frac{c t^2}{\eta^4}m)$ the estimate:
	\begin{equation}\label{ineq:concentration_ind}
	\frac{1}{m}\Big|\|\cA_{ \cI^c }(X)\|_1 - \|\cA_{ \cI }(X)\|_1 - \EE\left[\|\cA_{ \cI^c }(X)\|_1 - \|\cA_{ \cI }(X)\|_1\right]\Big| \leq t.
	\end{equation} 
\end{lemma}
\begin{proof}
	Markov's inequality along with \eqref{eqn:nondeg}  implies 
	\[\EE |\ell^\top X r| \geq \mu_0\cdot  \PP(|\ell^\top X r| \geq \mu_0) \geq \mu_0 p_0, \]
	which is the lower bound in \eqref{eqn:expec_bound}.
	Now we address the upper bound. To that end, suppose that $X$ has a singular value decomposition $X=\sigma_1 U_1 V_1^\top + \sigma_2 U_2 V_2^\top$. We then deduce
	\begin{align*}
	\| |\ell^\top X r| \|_{\psi_1}  = \| \ell^\top \left( \sigma_1 U_1 V_1^\top + \sigma_2 U_2 V_2^\top\right) r \|_{\psi_1} &= \|\sigma_1 \dotp{\ell, U_1} \dotp{V_1,r} + \sigma_2 \dotp{\ell, U_2} \dotp{V_2,r}\|_{\psi_1}\\
	& \leq \sigma_1 \|\dotp{\ell, U_1} \dotp{V_1,r}\|_{\psi_1} + \sigma_2 \|\dotp{\ell, U_2} \dotp{V_2,r} \|_{\psi_1}\\ 
	& \leq \sigma_1 \|\dotp{\ell, U_1}\|_{\psi_2}\|\dotp{V_1,r}\|_{\psi_2} + \sigma_2 \|\dotp{\ell, U_2}\|_{\psi_2} \| \dotp{V_2,r}\|_{\psi_2}\\ 
	& \leq (\sigma_1 + \sigma_2)\eta^2 \leq \sqrt{2}\eta^2,
	\end{align*}
	where the second inequality follows since $\|\cdot\|_{\psi_1}$ is a norm and $\|XY\|_{\psi_1} \leq \|X\|_{\psi_2}\|Y\|_{\psi_2}$ \cite[Lemma 2.7.7]{vershynin2016high}. This bound has two consequences: first $|\ell^\top X r|$ is a sub-exponential random variable with parameter $\sqrt{2}\eta^2$ and second $\EE |\ell^\top X r| \leq \sqrt{2}\eta^2,$ see \cite[Exercise 2.7.2]{vershynin2016high}. The first bound will be useful momentarily, while the second completes the proof of \eqref{eqn:expec_bound}. 
	
	Next define the sub-exponential random variable 
	$$
	Y_i = 
	\begin{cases}
	|\ell_i^\top X r_i| - \EE |\ell_i^\top X r_i| & \text{if } i \notin  \cI \\
	-(|\ell_i^\top X r_i| - \EE |\ell_i^\top X r_i|) & \text{if } i \in  \cI.
	\end{cases}
	$$ 
	Standard results (e.g. \cite[Exercise 2.7.10]{vershynin2016high}) imply $\|Y_i\|_{\psi_1} \lesssim \sqrt{2}\eta^2$ for all $i$. Using Bernstein inequality for sub-exponential random variables, \cref{theo:bern}, to upper bound $\PP \left( \frac{1}{m}\abs{\sum_{i=1}^m Y_i} \geq t \right)$ completes the proof.
\end{proof}

\begin{proof}[Proof of Theorem~\ref{theo:RIP}]
	Choose $\epsilon 
	\in (0,\sqrt{2})$ and let $\cN$ be the ($\epsilon/\sqrt{2}$)-net guaranteed 
	by Lemma~\ref{lemma:eps_net}. Let $\cE$ denote the event that the following two estimates hold for all matrices in $X\in \cN$:
	\begin{align}
	\frac{1}{m}\Big|\|\cA_{ \cI^c }(X)\|_1 - \|\cA_{ \cI }(X)\|_1 - \EE\left[\|\cA_{ \cI^c }(X)\|_1 - \|\cA_{ \cI }(X)\|_1\right]\Big| &\leq t,\label{ineq:concentration_ind1_first}\\
	\frac{1}{m}\Big|\|\cA(X)\|_1  - \EE\left[\|\cA(X)\|_1 \right]\Big| &\leq t. \label{ineq:concentration_ind_first}
	\end{align} 
	Throughout the proof, we will assume that the event $\cE$ holds.
	We will estimate the probability of $\cE$ at the end of the proof. Meanwhile, seeking to establish RIP,  define the quantity
	$$c_2 := \sup_{X \in S_2} \frac{1}{m}\|\cA(X)\|_1.$$
	We aim first to provide a high probability bound on  $c_2$.

	Let $X \in S_2$ be arbitrary and let $X_\star$ be the closest point to $X$ in $\cN$. Then we have 
	\begin{align}
	\frac{1}{m}\|\cA(X)\|_1 & \leq \frac{1}{m}\|\cA(X_\star)\|_1 + \frac{1}{m}\|\cA(X- X_\star)\|_1\notag\\ 
	&\leq \frac{1}{m}\EE\|\cA(X_\star)\|_1 + t+ \frac{1}{m}\|\cA(X- X_\star)\|_1  \label{eqn:est_firstusesec}\\
	&\leq \frac{1}{m}\EE\|\cA(X)\|_1 + t+ \frac{1}{m}\left(\EE\|\cA(X - X_\star)\|_1+ \|\cA(X- X_\star)\|_1 \right),\label{ineq:wtf_triangle}
	\end{align}
	where \eqref{eqn:est_firstusesec} follows from \eqref{ineq:concentration_ind} and \eqref{ineq:wtf_triangle} follows from the triangle inequality. To simplify the third term in \eqref{ineq:wtf_triangle},  using SVD, we deduce that there exist two orthogonal matrices $X_1, X_2$ of rank at most two satisfying $X - X_\star = X_1+X_2.$ With this decomposition in hand, we compute
	\begin{align}
	\frac{1}{m}\|\cA (X - X_\star)\|_1 &\leq \frac{1}{m}\|\cA(X_1)\|_1 + \frac{1}{m}\|\cA(X_2)\|_1\notag
	\\& \leq c_2 (\|X_1\|_F+\|X_2\|_F) \leq \sqrt{2}c_2 \|X-X_\star \|_F \leq c_2 \epsilon,\label{eqn:c2bound}
	\end{align}
	where the second inequality follows from the definition of $c_2$ and the estimate $\|X_1\|_F + \|X_2\|_F \leq \sqrt{2} \|(X_1, X_2)\|_F = \sqrt{2} \|X_1 + X_2\|_F.$ 
	Thus, we arrive at the bound
	\begin{equation}\label{ineq:RIP_upper}\frac{1}{m}\|\cA(X)\|_1 \leq \frac{1}{m}\EE\|\cA(X)\|_1 + t+ 2c_2 \epsilon.\end{equation}
	As $X$ was arbitrary, we may take the supremum of both sides of the inequality, yielding
	$c_2\leq \frac{1}{m}\sup_{X \in S_2}\EE\|\cA(X)\|_1 + t+ 2c_2 \epsilon$. Rearranging yields the bound
	$$
	c_2 \leq \dfrac{\frac{1}{m}\sup_{X \in S_2}\EE\|\cA(X)\|_1 + t}{1-2\epsilon}.
	$$ Assuming that $\epsilon \leq 1/4$, we further deduce that 
	\begin{equation} \label{sigma_bounded}
	c_2 \leq \bar \sigma := \frac{2}{m}\sup_{X \in S_2}\EE\|\cA(X)\|_1 + 2t,
	\end{equation}
	establishing that the random variable $c_2$ is bounded by $\bar \sigma$ in the event $\cE$.

	Now let $\hat \cI$ denote either $\hat \cI=\emptyset$ or $\hat \cI=\cI$. We now provide a uniform lower bound on $\frac{1}{m}\|\cA_{\hat \cI^c }(X)\|_1 - \frac{1}{m}\|\cA_{\hat \cI }(X)\|_1$. Indeed, 
	\begin{align}
	&\frac{1}{m}\|\cA_{\hat \cI^c }(X)\|_1 - \frac{1}{m}\|\cA_{\hat \cI }(X)\|_1\notag \\
	&=\frac{1}{m}\|\cA_{\hat \cI^c }(X_{\star})+\cA_{\hat \cI^c }(X-X_{\star})\|_1 - \frac{1}{m}\|\cA_{\hat \cI }(X_{\star})+\cA_{\hat \cI }(X-X_\star)\|_1\notag \\
	&\geq \frac{1}{m}\|\cA_{\hat \cI^c }(X_\star)\|_1 - \frac{1}{m}\|\cA_{\hat \cI }(X_\star)\|_1 -  \frac{1}{m}\|\cA(X- X_\star)\|_1 \label{eqn:doubletriangle}\\ 
	& \geq  \frac{1}{m}\EE\left[\|\cA_{\hat \cI^c }(X_\star)\|_1 - \|\cA_{\hat \cI }(X_\star)\|_1\right] - t -  \frac{1}{m}\|\cA(X- X_\star)\|_1 \label{eqn:allple45}\\ 
	& \geq \frac{1}{m}\EE\left[\|\cA_{\hat \cI^c }(X)\|_1 - \|\cA_{\hat \cI }(X)\|_1\right] - t -\frac{1}{m} \left(\EE \|\cA(X- X_\star)\|_1 + \|\cA(X- X_\star)\|_1\right) \label{eqn:moredoubletriangle}\\
	& \geq \frac{1}{m}\EE\left[|\|\cA_{\hat \cI^c }(X)\|_1 - \|\cA_{\hat \cI }(X)\|_1\right] - t  - 2\bar \sigma \epsilon,\label{eqn:finalestthankgod}
	\end{align}
	where \eqref{eqn:doubletriangle} uses the forward and reverse triangle inequalities,
	\eqref{eqn:allple45} follows from~\eqref{ineq:concentration_ind1_first}, the estimate \eqref{eqn:moredoubletriangle} follows from the forward and reverse triangle inequalities, and \eqref{eqn:finalestthankgod} follows from \eqref{eqn:c2bound} and \eqref{sigma_bounded}. Switching the roles of $\cI$ and $\cI^c$ in the above sequence of inequalities, and choosing $\epsilon = t/4\bar \sigma$, we deduce 
	% Therefore, if $\epsilon \leq t/4\bar \sigma$, this inequality together with %\eqref{ineq:RIP_upper} implies that in event $\cE,$ we have
	\[\frac{1}{m}\sup_{X \in S_2}\Big|\|\cA_{\hat \cI^c }(X)\|_1 - \|\cA_{\hat \cI }(X)\|_1 - \EE\left[|\|\cA_{\hat \cI^c }(X)\|_1 - \|\cA_{\hat \cI }(X)\|_1\right]\Big| \leq  \frac{3t}{2}.\]
	In particular, setting $\hat \cI=\emptyset$, we deduce 
	\[\frac{1}{m}\sup_{X \in S_2}\Big|\|\cA(X)\|_1 - \EE\left[\|\cA(X)\|_1 \right]\Big| \leq  \frac{3t}{2}\]
	and therefore using \eqref{eqn:expec_bound}, we conclude the RIP property
	\begin{equation}\label{eqn:target_eqn1proof}
	\mu_0 p_0-\frac{3t}{2}\leq  \frac{1}{m}\|\cA(X)\|_1\lesssim \eta^2+\frac{3t}{2},\qquad \forall X\in S_2.
	\end{equation}
	Next, let $\hat \cI = \cI$ and note that 
	$$
	\frac{1}{m}\EE\left[\|\cA_{\hat \cI^c }(X)\|_1 - \|\cA_{\hat \cI }(X)\|_1\right] = \frac{|\cI^c| - |\cI|}{m}\cdot\EE |\ell^\top X r| \geq \mu_0 p_0\left(1- \frac{2|\cI|}{m}\right).
	$$
	Therefore every $X\in S_2$ satisfies 
	\begin{equation}\label{eqn:target_eqn2proof}
	\frac{1}{m}\left[\|\cA_{\hat \cI^c }(X)\|_1 - \|\cA_{\hat \cI}(X)\|_1\right]\geq \mu_0 p_0\left(1- \frac{2|\cI|}{m}\right)-\frac{3t}{2}.
	\end{equation}
	Setting $t=\frac{2}{3}\min\{\mu_0 p_0/2, \mu_0 p_0(1-2|\cI|/m)/2\} =  \frac{1}{3}\mu_0 p_0(1-2|\cI|/m)$ in \eqref{eqn:target_eqn1proof} and \eqref{eqn:target_eqn2proof}, we deduce the claimed estimates \eqref{eqn:RIPM1} and \eqref{eqn:oulierM1}.
	Finally, let us estimate the probability of $\cE$. Using Lemma~\ref{lem:exp_rip} and the union bound yields
	\begin{align*}
	\PP(\cE^c) & \leq \sum_{X \in \cN} \PP \big\{ \text{\eqref{ineq:concentration_ind1_first} or  \eqref{ineq:concentration_ind_first} fails at }X\big\} \\
	&\leq 4|\cN|\exp\left(-\frac{c t^2}{\eta^4}m\right)\\
	& \leq 4\left(\frac{9}{\epsilon}\right)^{2(d_1+d_2+1)} \exp\left(-\frac{c t^2}{\eta^4}m\right) \\
	&=4  \exp\left(2(d_1+d_2+1)\ln(9/\epsilon)-\frac{c t^2}{\eta^4}m\right)\end{align*}
	where the second inequality follows from \cref{lemma:eps_net} and $c$ is a constant. 
	
	Then we deduce since $1/\epsilon = 4\bar \sigma/t \lesssim 2 + \eta^2/(1 - 2|\cI|/m)$.
	$$\PP(\cE^c)\leq 4  \exp\left(c_1(d_1+d_2+1)\ln\left(c_2+\frac{c_2}{1-2|\cI|/m}\right)-\frac{4c\mu^2_0 p_0^2(1-\frac{2|\cI|}{m})^2}{9\eta^4}m\right).$$
	Hence as long as $m\geq \frac{18\eta^4c_1(d_1+d_2+1)\ln\left(c_2+\frac{c_2}{1-2|\cI|/m}\right)}{4c\mu^2_0 p_0^2(1-\frac{2|\cI|}{m})^2}$, we can be sure $\PP(\cE^c)\leq 4  \exp\left(-\frac{4c\mu^2_0p_0^2(1-\frac{2|\cI|}{m})^2}{18\eta^4}m\right)$. The result follows immediately.
	
\end{proof}

Combining Theorem~\ref{theo:RIP} with Corollary~\ref{cor:generic_conv_iso} we obtain the following guarantee.

\begin{corollary}[Convergence guarantees]\label{cor:coverg_guarant_distrib}
	Consider the measurement model \eqref{eqn:meas_model} and suppose that model~\ref{MModel:1} is valid. Consider the optimization problem 
	$$\min_{(x,w)\in \cS_\nu}~ f(w,x)=\frac{1}{m}\sum_{i=1}^m|\langle \ell_i,w\rangle\langle r_i,x\rangle - y_i|.$$
	Then there exist constants $c_1,c_2,c_3,c_4,c_5,c_6>0$ depending only on $\mu_0, p_0, \eta$ such that as long as 
	$m\geq  \frac{c_1(d_1+d_2+1)}{(1-2\pfail)^2}\ln\left(c_2+\frac{c_2}{1-2\pfail}\right)$ and you choose any pair $(x_0,y_0)$  with relative error
	\begin{equation}\label{eqn:rel_error_cond}
	\frac{\dist((w_0,x_0),\cS^\ast_{\nu})}{\sqrt{\|\bar w\bar x^\top\|_F}}\leq \frac{c_6\left(1- 2\pfail\right)}{4\sqrt{2}c_5(\nu+1)},
	\end{equation}
	then with probability at least $1-4\exp\left(-c_3(1-2\pfail)^2m\right)$
	the following are true.
	\begin{enumerate}
		\item {\bf (Polyak subgradient)} Algorithm~\ref{alg:polyak} initialized  $(x_0,y_0)$  produces iterates that 
		converge linearly to $\cS^\ast_{\nu}$, that is 
		\begin{equation*}
		\frac{\dist^2((w_{k},x_{k}),\cS^\ast_{\nu})}{\|\bar w\bar x^\top\|_F}\leq \left(1-\frac{c_6^2\left(1- 2\pfail\right)^2}{32c_5^2(\nu+1)^4}\right)^{k}\cdot \frac{c_6^2\left(1- 2\pfail\right)^2}{32c_5^2(\nu+1)^2}\qquad \forall k\geq 0.%\left(1  - \frac{\rho\mu\|x_k + \bar x \|}{\|\zeta_k\|^2}\left( \frac{\kappa}{\rho}\|x_k + \bar x\| - \|x_k - \bar x\| \right)\right)
		\end{equation*}
		\item {\bf (geometric subgradient)}
		Set 
		$\lambda:=\frac{c_6^2\left(1- 2\pfail\right)^2\sqrt{\|\bar w\bar x^\top\|_F}}{16\sqrt{2}c_5^2\nu(\nu+1)^2} \textrm{ and } q:=\sqrt{1-\frac{c_6^2\left(1- 2\pfail\right)^2}{32c_5^2(\nu+1)^4}}.$
		Then the iterates $x_k$ generated by Algorithm~\ref{alg:geometrically_step}, initialized at $(w_0,x_0)$	
		converge linearly: 		
		\begin{equation*} 
		\frac{\dist^2((w_{k},x_{k}),\cS^\ast_{\nu})}{\|\bar w\bar x^\top\|_F} \leq 
		\left(1-\frac{c_6^2\left(1- 2\pfail\right)^2}{32c_5^2(\nu+1)^4}\right)^{k}\cdot\frac{c_6^2\left(1- 2\pfail\right)^2}{32c_5^2(\nu+1)^2}\qquad \forall k\geq 0.
		\end{equation*}
		\item {\bf (prox-linear)}  Algorithm~\ref{alg:prox_lin} with $\beta = \rho$ and initialized at $(w_0,x_0)$ converges quadratically:
		$$\frac{\dist((w_k,x_{k}),\cX^*)}{\sqrt{\|\bar w\bar x^\top\|_F}}\leq 2^{-2^{k}}\cdot \frac{c_6\left(1- 2\pfail\right)}{2\sqrt{2}c_5(\nu+1)}\qquad \forall k\geq 0.$$
	\end{enumerate}
\end{corollary}

Thus with high probability, if one initializes the subgradient and prox-linear methods at a pair $(w_0,x_0)$ satisfying $\frac{\dist((w_0,x_0),\cS^\ast_{\nu})}{\sqrt{\|\bar w\bar x^\top\|_F}}\leq \frac{c_6\left(1- 2\pfail\right)}{4\sqrt{2}c_5(\nu+1)}$, then the methods will converge to the optimal solution set at a dimension independent rate.

\subsubsection{Assumptions \ref{assump:RIP} and \ref{assump:outlier} under Model~\ref{MModel:2}}
In this section, we verify Assumptions \ref{assump:RIP} and \ref{assump:outlier} under Model~\ref{MModel:2} and an extra incoherence condition.
Namely, we  impose further conditions on $\ell_p/\ell_2$ singular values of $L$ ($p \geq1$)
\begin{align*}
\sigma_{p, \min}(L) = \inf_{w\in \SS^{d-1}}\|Lw\|_p \qquad \text{and} \qquad \sigma_{p,\max}(L) : = \sup_{w\in \SS^{d-1}}\|Lw\|_p,
\end{align*}
which intuitively guarantee that the entries of any vector in $\{\cA(X) \mid \rank(X) \leq 2\}$ are ``well-spread.''

\begin{proposition}[Measurement Model~\ref{MModel:2}]\label{theo:RIPM2}
	Assume Model~\ref{MModel:2} and fix an arbitrary index $\cI\subseteq\{1,\ldots, m\}$. Define the parameter 
	$$\Delta:=\frac{ \sigma_{1,\min}(L)}{2\sqrt{\pi}m} - 2\sigma_{\infty,\max}(L)\sqrt{\frac{2}{\pi}} \frac{| \cI|}{m},$$
	and suppose $\Delta>0$. Then there exist numerical constants $c_1, c_2, c_3 > 0$ such that with probability $$1-4  \exp\left(c_1(d_1+d_2+1)\ln\left(c_2\left(1+\tfrac{\sigma_{1, \max}(L)}{\sigma_{1, \min}(L)} \right)\right)-c_3 \cdot\frac{\sigma^2_{1, \min}(L)}{\sigma^2_{2, \max}(L)}\right),$$
	every matrix $X\in\R^{d_1\times d_2}$ of rank at most two satisfies 
	\begin{equation}\label{eqn:ripM2}
	\frac{  \sigma_{1, \min}(L)}{2\sqrt{\pi}m}\|X\|_F \leq \frac{1}{m}\|\cA(X)\|_1\leq \frac{  (2^{5/2}+1) \sigma_{1, \max}(L)}{2\sqrt{\pi}m}\cdot \|X\|_F,
	\end{equation}
	and
	\begin{equation}\label{eqn:outlierM2}
	\frac{1}{m}\|\cA_{\cI^c }(X)\|_1 - \frac{1}{m}\|\cA_{\cI }(X)\|_1\geq \Delta\|X\|_F.
	\end{equation}
	
\end{proposition}

\begin{proof}
	The argument mirrors the proof of Proposition~\ref{theo:RIP} and therefore we only provide a sketch. Fix a unit Frobenius norm matrix $X$ of rank at most two.  We aim to show that for any fixed $\hat \cI \subseteq \{1, \ldots, m\}$, the following random variable is highly concentrated around its mean:
	$$
	Z_{\hat\cI} = \frac{1}{m}\|\cA_{\hat \cI^c }(X)\|_1 - \frac{1}{m}\|\cA_{\hat \cI }(X)\|_1.
	$$
	To that end, fix a singular value decomposition $X=s_1u_1v_1^\top+s_2u_2v_2^\top$. We then compute 
	\[\left(\cA(X)\right)_i =  \ell_i^\top( s_1u_1v_1^\top+s_2u_2v_2^\top) r_i =  s_1\dotp{\ell_i, u_1}
	\widehat r_i^{(1)}+  s_2\dotp{\ell_i, u_2} \widehat r_i^{(2)},  \]
	where $u_1$ and $u_2$ are orthogonal, $s_1^2 + s_2^2 = 1$, and $\widehat
	r_i^{(1)}, \widehat r_i^{(2)}$ are i.i.d.\ standard normal random variables. This decomposition,  together with the rotation invariance of the normal distribution, furnishes us with the following distributional equivalence:
	\[\left(\cA(X)\right)_i \stackrel{(d)}{=} \sqrt{(s_1\dotp{\ell_i, u_1})^2 +
		(s_2\dotp{\ell_i, u_2})^2} \;\widehat r_i,\]
	where $\widehat r_i$ is a standard normal random variable. Consequently, we have the following expression for the expectation:
	\begin{align*}
	\EE\left[Z_{\hat \cI}\right] = \sqrt{\frac{2}{\pi}} \frac{1}{m}
	\sum_{i\in \hat \cI^c} \sqrt{(s_1\dotp{\ell_i, u_1})^2 + (s_2\dotp{\ell_i, u_2})^2} - \sqrt{\frac{2}{\pi}} \frac{1}{m}
	\sum_{i\in \hat \cI} \sqrt{(s_1\dotp{\ell_i, u_1})^2 + (s_2\dotp{\ell_i, u_2})^2}.
	\end{align*}
	We now upper/lower bound this expectation. 
	The upper bound follows from the estimate
	\begin{align*}
	\EE\left[Z_{\hat \cI}\right] \leq  \EE\left[\frac{1}{m} \|\cA(X)\|_1\right] \leq  \sqrt{\frac{2}{\pi}}\frac{1}{m}\left(\|Lu_1\|_1 + \|Lu_2\|_1\right) =
	\frac{  2^{3/2} \sigma_{1, \max}(L)}{\sqrt{\pi}m}.
	\end{align*}
	The lower bound uses the following two dimensional inequality $\frac{\|z\|_1}{\sqrt{2}} \leq \|z\|_2 \leq \|z\|_1$, which holds for all $z \in \RR^2$:
	\begin{align*}
	\EE\left[ Z_{\hat \cI}\right] &=  \sqrt{\frac{2}{\pi}} \frac{1}{m}
	\sum_{i\in \hat \cI^c} \sqrt{(s_1\dotp{\ell_i, u_1})^2 + (s_2\dotp{\ell_i, u_2})^2} -  \sqrt{\frac{2}{\pi}} \frac{1}{m}
	\sum_{i\in \hat \cI} \sqrt{(s_1\dotp{\ell_i, u_1})^2 + (s_2\dotp{\ell_i, u_2})^2} \\
	&=  \sqrt{\frac{2}{\pi}} \frac{1}{m}
	\sum_{i=1}^m \sqrt{(s_1\dotp{\ell_i, u_1})^2 + (s_2\dotp{\ell_i, u_2})^2} - 2 \sqrt{\frac{2}{\pi}} \frac{1}{m}
	\sum_{i\in \hat \cI} \sqrt{(s_1\dotp{\ell_i, u_1})^2 + (s_2\dotp{\ell_i, u_2})^2} \\
	&\geq \frac{1 }{\sqrt{\pi}m}\left(|s_1|\|Lu_1\|_1 +
	|s_2|\|Lu_2\|_1\right) - 2 \sqrt{\frac{2}{\pi}} \frac{|\hat \cI|}{m} \max_{i = 1, \ldots, m} \|\ell_i\|_2\\
	&\geq \frac{  \sigma_{1, \min}(L)}{\sqrt{\pi}m} - 2 \sigma_{\infty,\max}(L)\sqrt{\frac{2}{\pi}} \frac{|\hat \cI|}{m}.
	\end{align*}
	In particular, setting $\hat\cI=\emptyset$, we deduce
	\begin{equation}\label{eqn:exp_bound}
	\frac{  \sigma_{1, \min}(L)}{\sqrt{\pi}m} \leq \frac{1}{m}\EE\|\cA(X)\|_1\leq \frac{  2^{3/2} \sigma_{1, \max}(L)}{\sqrt{\pi}m}.
	\end{equation}
	
	To establish concentration of the random variable $Z_{\hat \cI}$, we apply a standard result (Theorem~\ref{theo:subg_conc}) on the concentration of weighted sums of mean zero independent sub-gaussian random variables. In particular, to apply Theorem~\ref{theo:subg_conc}, we write $Y_i = |\widehat r_i| - \EE|\widehat r_i|$, and define weights
	$$
	a_i = \frac{1}{m}\cdot\begin{cases}
	\sqrt{(s_1\dotp{\ell_i, u_1})^2 + (s_2\dotp{\ell_i, u_2})^2}& \text{if }i \notin \hat \cI,\\
	-
	\sqrt{(s_1\dotp{\ell_i, u_1})^2 + (s_2\dotp{\ell_i, u_2})^2} & \text{if } i \in \hat \cI.
	\end{cases}
	$$
	Noticing that $\||\widehat r_i| - \EE |\widehat r_i|\|_{\psi_2} \leq K$, where $K > 0$ is an absolute constant, and 
	$$
	\|a\|^2_2= \frac{1}{m^2}
	\sum_{i=1}^m \Big((s_1\dotp{\ell_i, u_1})^2 + (s_2\dotp{\ell_i, u_2})^2\Big) \leq \frac{2\sigma^2_{2, \max}(L)}{m^2},
	$$
	it follows that for any fixed unit Frobenius norm matrix $X$ of rank at most two, with probability at least $1 - 2 \exp\left(-\frac{c t^2 
		m^2}{2K^2\sigma^2_{2, \max}(L)}\right)$, we have 
	\begin{align}
	\frac{1}{m}\left|\|\cA_{\hat \cI^c }(X)\|_1 - \|\cA_{\hat \cI }(X)\|_1 - \EE\left[|\|\cA_{\hat \cI^c }(X)\|_1 - \|\cA_{\hat \cI }(X)\|_1\right]\right| \leq t.
	\label{eq:fourier-rip-single}
	\end{align}
	We have thus established concentration for any fixed $X$.  
	We now proceed with a covering argument in the same way as in the proof of Theorem~\ref{theo:RIP}. To this end, choose $\epsilon 
	\in (0,\sqrt{2})$ and let $\cN$ be the ($\epsilon/\sqrt{2}$)-net guaranteed 
	by~\cref{lemma:eps_net}. Let $\cE$ denote the event that the following two estimates hold for all matrices $X\in \cN$:
	\begin{align*}
	\frac{1}{m}\Big|\|\cA_{ \cI^c }(X)\|_1 - \|\cA_{ \cI }(X)\|_1 - \EE\left[\|\cA_{ \cI^c }(X)\|_1 - \|\cA_{ \cI }(X)\|_1\right]\Big| &\leq t,\\
	\frac{1}{m}\Big|\|\cA(X)\|_1  - \EE\left[\|\cA(X)\|_1 \right]\Big| &\leq t.
	\end{align*} 
	Throughout the proof, we will assume that the event $\cE$ holds. By exactly the same covering argument as in Theorem~\ref{theo:RIP}, setting $\epsilon = t/4\bar \sigma$ with $\bar \sigma=\frac{2}{m}\sup_{X \in S_2}\EE\|\cA(X)\|_1 + 2t$, we deduce
	\[\frac{1}{m}\sup_{X \in S_2}\Big|\|\cA_{\hat \cI^c }(X)\|_1 - \|\cA_{\hat \cI }(X)\|_1 - \EE\left[\|\cA_{\hat \cI^c }(X)\|_1 - \|\cA_{\hat \cI }(X)\|_1\right]\Big| \leq  \frac{3t}{2},\]
	where either $\hat\cI=\emptyset$ or $\hat\cI=\cI$.

	In particular, setting $\hat \cI=\emptyset$ and using the bound \eqref{eqn:exp_bound}, we deduce 
	$$\frac{  \sigma_{1, \min}(L)}{\sqrt{\pi}m}-\frac{3t}{2} \leq \frac{1}{m}\|\cA(X)\|_1\leq \frac{  2^{3/2} \sigma_{1, \max}(L)}{\sqrt{\pi}m}+\frac{3t}{2}$$
	for all $X\in S_2$. In turn, setting $\hat \cI=\cI$ we deduce 
	\begin{align*}
	\frac{1}{m}\|\cA_{\cI^c }(X)\|_1 - \frac{1}{m}\|\cA_{\cI }(X)\|_1&\geq \EE\left[\|\cA_{\cI^c }(X)\|_1 - \|\cA_{\cI }(X)\|_1\right]-\frac{3t}{2}\\
	&\geq \frac{  \sigma_{1, \min}(L)}{\sqrt{\pi}m} - 2 \sigma_{\infty,\max}(L)\sqrt{\frac{2}{\pi}} \frac{| \cI|}{m}-\frac{3t}{2}.
	\end{align*}
	Setting $t:=\frac{  \sigma_{1, \min}(L)}{3\sqrt{\pi}m}$, the estimates  \eqref{eqn:ripM2} and \eqref{eqn:outlierM2} follow immediately. Finally, estimating the probability of $\cE$ using the union bound quickly yields:
	\begin{align*}
	\PP(\cE^c)
	&\leq 4  \exp\left(c_1(d_1+d_2+1)\ln\left(c_2\left(1+\frac{\sigma_{1, \max}(L)}{\sigma_{1, \min}(L)} \right)\right)-c_3 \cdot\frac{\sigma^2_{1, \min}(L)}{\sigma^2_{2, \max}(L)}\right).
	\end{align*}
	The result follows.
\end{proof}

\section{Initialization} \label{sec:damek_init}

Previous sections have focused on local convergence guarantees under various statistical assumptions. In particular, under Assumptions~\ref{assump:RIP} and \ref{assump:outlier}, one must initialize the local search procedures at a point $(w,x)$, whose relative distance to the solution set $\frac{\dist((x,w),\mathcal{S}^*_{\nu})}{\sqrt{\|\bar x\bar w^\top\|_F}}$ is upper bounded by a constant. In this section, we present a new spectral initialization routine (Algorithm~\ref{alg:spectral_init}) that is able to efficiently find such point $(w,x)$. The algorithm is inspired by~\cite[Section 4]{duchi_ruan_PR} and~\cite{eldar_init}. 

Before describing the intuition behind the procedure, let us formally introduce our assumptions. Throughout this section, we make the following assumption on the data generating mechanism, which is stronger than Model~\ref{MModel:1}:
\begin{enumerate}[label = $\mathbf{\overline{M\arabic*}}$]
	\item \label{hatMModel:1} The entries of matrices $L$ and $R$  are i.i.d.\ Gaussian.  
\end{enumerate}
Our arguments rely heavily on properties of the Gaussian distribution. We note, however, that our experimental results suggest that Algorithm~\ref{alg:spectral_init} provides high-quality initializations under weaker distributional assumptions.

Recall that in the previous sections, the noise $\xi$ was arbitrary. In this section, however, we must assume more about the nature of the noise. We will  consider two different settings.
\begin{enumerate}[label = \textbf{N\arabic*}]
	\item\label{NModel:1} The measurement vectors $\{(\ell_i, r_i)\}_{i=1}^m$ and the noise sequence $\{\xi_i\}_{i = 1}^m$ are independent.
	\item\label{NModel:2} The inlying measurement vectors  $\{(\ell_i, r_i)\}_{i\in \cIi}$ and the corrupted observations $\{\xi_i\}_{i \in \cIo}$ are independent.
\end{enumerate}

The noise models \ref{NModel:1}  and \ref{NModel:2} differ in how an adversary may choose to corrupt the measurements.  Model~\ref{NModel:1} allows an adversary to corrupt the signal, but does not allow observation of the measurement vectors $\{(\ell_i,r_i)\}_{i=1}^m$. On the other hand, Model~\ref{NModel:2} allows an adversary to observe the outlying measurement vectors $\{(\ell_i, r_i)\}_{i \in \cIo}$ and arbitrarily corrupt those measurements. For example, the adversary may replace the outlying measurements with those taken from a completely different signal:  $y_i = \left( \cA(\widetilde w \widetilde x^\top )\right)_i$ for $i \in \cIo.$

\begin{algorithm}[h]
	\KwData{$y \in \RR^m, L \in \RR^{m \times d_1}, R \in \RR^{m \times d_2}$}
	$\cIs \gets \{i \mid |y_i| \leq \text{\texttt{med}}(|y|) \}$ \\
	Form directional estimates:
	$$\arraycolsep=1.4pt\def\arraystretch{2.2}
	\begin{array}{c}
	\Linit \gets \frac{1}{m} \sum_{i \in \cIs} \ell_i \ell_i^\top, \quad
	\Rinit \gets \frac{1}{m} \sum_{i \in \cIs} r_i r_i^\top \\
	\what \gets \argmin_{p \in \SS^{d_1-1}} p^\top \Linit p, \quad \text{and} 
	\qquad
	\xhat \gets \argmin_{q \in \SS^{d_2-1}} q^\top \Rinit q.
	\end{array}$$
	Estimate the norm of the signal:
	\begin{gather*}
	\estrad \gets \argmin_{\beta \in \RR} G(\beta) := \frac{1}{m}
	\sum_{i=1}^{m} \abs{y_i - \beta \langle \ell_i, \what\rangle \langle r_i,\xhat\rangle}, \\
	w_0 \gets \sign(\estrad)\abs{\estrad}^{1/2} \widehat{w},\qquad \text{and}  \qquad
	x_0 \gets \abs{\estrad}^{1/2} \widehat{x}.
	\end{gather*}
	\Return $(w_0, x_0)$
	\caption{Initialization.}
	\label{alg:spectral_init}
\end{algorithm}

We can now describe the intuition underlying Algorithm~\ref{alg:spectral_init}.
Throughout we denote unit vectors parallel to $\wbar$ and  $\xbar$ by
$\dirw$ and $\dirx $, respectively. Algorithm~\ref{alg:spectral_init} exploits the expected near orthogonality of the random vectors $\ell_i$ and $r_i$ to the directions $\dirw$ and $\dirx$, respectively, in order to select a ``good" set of measurement vectors. Namely, since $\EE\left[\dotp{\ell_i, \dirw}\right] = \EE\left[\dotp{r_i, \dirx}\right] = 0$, we expect minimal eigenvectors of $\Linit$ and $\Rinit$ to be near $\dirw$ and $\dirx$, respectively. Since our measurements are bilinear, we cannot necessarily select vectors for which $|\dotp{\ell_i,\dirw}|$ and $|\dotp{r_i, \dirx}|$ are both small, rather, we may only select vectors for which the product \( \abs{\ip{\ell_i, \dirw} \ip{r_i, \dirx}} \) is small, leading to subtle ambiguities not present in \cite[Section 4]{duchi_ruan_PR} and~\cite{eldar_init}; see Figure~\ref{fig:spectral_init}. Corruptions add further ambiguities since the noise model~\ref{NModel:2} allows a constant fraction of measurements to be adversarially modified.

\begin{figure}[h!]
	\centering
	\begin{tikzpicture}[scale=0.8]
	\draw[line width=0.8pt, <->] (-3.5,0)--(3.5,0) node[right]{$x$};
	\draw[line width=0.8pt, <->] (0,-3.5)--(0,3.5) node[above]{$y$};
	\draw[line width=1.25pt, hblue,-stealth]
	(0,0)--(1.61,2.52) node[anchor=south west]{$\ell_1$};
	\draw[line width=1.25pt, hblue,-stealth]
	(0,0)--(-2.34,-1.87) node[anchor=south east]{$r_1$};
	\draw[line width=1.25pt, black,-stealth]
	(0,0)--(-2.61, 1.47) node[anchor=south west]{$\wstar$};
	\draw[line width=1.25pt, black,-stealth]
	(0,0)--(1.66, -2.5) node[anchor=south west]{$\xstar$};
	\draw[line width=1.25pt, hred,-stealth](0,0)--(-2.94,-0.5)
	node[anchor=north east]{$\ell_2$};
	\draw[line width=1.25pt,hred,-stealth](0,0)--(2.74,-1.224)
	node[anchor=south]{$r_2$};
	\end{tikzpicture}
	\caption{Intuition behind spectral initialization. The pair \( \ell_1, r_1 \) 
		will be included since both vectors are almost orthogonal to the true 
		directions. \( \ell_2, r_2 \) is unlikely to be included since $r_2$ is almost 
		aligned with \( \xstar \).}
	\label{fig:spectral_init}
\end{figure}
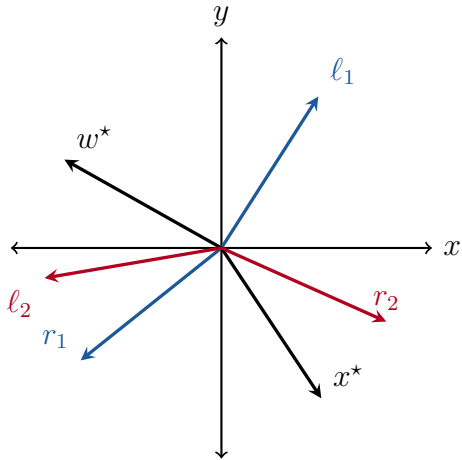

Formally, Algorithm~\ref{alg:spectral_init} estimates an initial signal $(w_0, x_0)$ in two stages: first it constructs a pair of directions $(\hat w, \hat x)$ which estimate the true directions $$ 
\dirw :=  \frac{1}{\|\bar w\|_2} \bar w \qquad \text{ and } \qquad \dirx := \frac{1}{\|\bar x\|_2} \bar x
$$ (up to sign); then it constructs an estimate $\estrad$ of the signed signal norm $\pm M$, which corrects for sign errors in the first stage. We now discuss both stages in more detail, starting with the direction estimate. Most proofs will be deferred to Appendix~\ref{appendix:initialization}. The general proof strategy we follow is analogous to~\cite[Section 4]{duchi_ruan_PR} for phase retrieval, with some subtle modifications due to asymmetry.

\paragraph{Direction Estimate.} In the first stage of the algorithm, we estimate the directions $\dirw$ and $\dirx$, up to sign. Key to our argument is the following decomposition for model \ref{NModel:1} (which will be proved in Appendix~\ref{appendix:init_direction}):
\begin{align*}
\Linit = \tfrac{|\sel|}{m}\cdot I_{d_1} - \gamma_1 \dirw \dirw^\top + \Delta_L, \quad
\Rinit = \tfrac{|\sel|}{m}\cdot I_{d_2} - \gamma_2 \dirx \dirx^\top + \Delta_R, \quad
\end{align*}
where $\gamma_1,\gamma_2\gtrsim 1$ and the matrices $\Delta_L, \Delta_R$ have small operator norm (decreasing with $(d_1+d_2)/m$), with high probability. Using the Davis-Kahan $\sin \theta$ theorem~\cite{DavKah70}, we can then show that the minimal eigenvectors of $\Linit$ and $\Rinit$ are sufficiently close to $\{\pm \dirw\}$ and $\{\pm\dirx\}$, respectively. 

\begin{proposition}[Directional estimates]
	\label{prop:directional_init_ok}There exist numerical constants $c_1, c_2,C > 0$, so that for any $\pfail\in [0,1/10]$ and $t\in [0,1]$, with probability at least 
	$1 - c_1\exp\left( -c_2mt\right)$,
	the following hold:
	\begin{align*}
	\min_{s \in\{\pm 1\}} \norm{\what \xhat^\top -s \wstar \xstar^\top}_F &\leq \begin{cases}
	C \cdot \left(
	\sqrt{\frac{\max\{d_1, d_2\}}{m} + t}\right)& \text{ under Model~\ref{NModel:1}, and} \\
	C \cdot \left(\pfail + 
	\sqrt{\frac{\max\{d_1, d_2\}}{m} + t}\right)& \text{ under Model~\ref{NModel:2}.}
	\end{cases}
	\end{align*}
\end{proposition}

\paragraph{Norm estimate.} In the second stage of the algorithm, we estimate $M$ as well as correct the sign of the direction estimates from the previous stage. In particular, for any $(\what, \xhat)\in \SS^{d_1-1} \times \SS^{d_2-1}$ define the quantity
\begin{equation}
\delta := \left(1 + \frac{c_5}{c_6(1-2\pfail)}\right)\min_{s\in\{\pm 1\}} \left\|\widehat w \widehat
x^\top  -s  \dirw \dirx^\top\right\|_F
\label{eq:radius_delta},
\end{equation}
where $c_5$ and $c_6$ are as in Theorem~\ref{theo:RIP}.
Then we prove the following estimate (see Appendix~\ref{appendix:init_radius}).
\begin{proposition}[Norm Estimate]\label{prop:radius_estimate_ok}
	Under either noise model, \ref{NModel:1} and \ref{NModel:2}, there exist numerical constants $c_1,\ldots, c_6 > 0$ so that if 
	$m\geq  \frac{c_1(d_1+d_2+1)}{(1-2\pfail)^2}\ln\left(c_2+\frac{c_2}{1-2\pfail}\right)$, then with probability at least $1-4\exp\left(-c_3(1-2\pfail)^2m\right)$, we have that any minimizer $\estrad$ of the function
	$$
	G(\beta) := 
	\frac{1}{m} \sum_{i=1}^m|y_i - \beta \langle \ell_i, \widehat w\rangle\langle
	\widehat  x, r_i\rangle|
	$$
	satisfies $||\estrad | - M| \leq \delta M$. Moreover, if in this event $\delta<1$, then we have $\sign (\widehat M) =  \argmin_{s \in \{\pm 1\}} \norm{\what \xhat^\top - s \dirw \dirx^\top}_F$.
	
\end{proposition}

Thus, the preceding proposition shows that tighter estimates on the norm $M$ result from better directional estimates in the first stage of Algorithm~\ref{alg:spectral_init}. In light of Proposition~\ref{prop:radius_estimate_ok}, we next estimate the probability of the event $\delta\leq 1/2$, which in particular implies with high probability $\sign (\widehat M) =  \argmin_{s \in \{\pm 1\}} \norm{\what \xhat^\top - s \dirw \dirx^\top}_F$.

\begin{proposition}[Sign estimate]\label{prop:small_delta}
	Under either Model~\ref{NModel:1} and \ref{NModel:2}, there exist numerical constants $c_0,c_1, c_2,c_3 > 0$ such that if $\pfail<c_0$ and $m \geq c_3 (d_1 + d_2)$, then the estimate holds:\footnote{In the case of model \ref{NModel:1}, one can set $c_0=1/10$.}
	
	$$\Prob{\delta> 1/2}\leq c_1\exp\left( -c_2m\right).$$
\end{proposition}
\begin{proof}
	Using  Theorem~\ref{theo:RIP} and Propositions~\ref{prop:directional_init_ok}, we deduce that for any $t\in [0,1]$, with probability $1 - c_1\exp\left( -c_2mt\right)$ we have
	$$\delta\leq \begin{cases}
	C \cdot \left(
	\sqrt{\frac{\max\{d_1, d_2\}}{m} + t}\right)& \text{ under Model~\ref{NModel:1}, and} \\
	C \cdot \left(\pfail + 
	\sqrt{\frac{\max\{d_1, d_2\}}{m} + t}\right)& \text{ under Model~\ref{NModel:2}.}
	\end{cases}
	$$
	Thus under model \ref{NModel:1} it suffices to set $t=(2C)^{-2}-\frac{\max\{d_1,d_2\}}{m}$. Then the probability of the event $\delta\leq 1/2$ is at least $1 - c_1\exp\left( -c_2((2C)^{-2} m-\max\{d_1,d_2\})\right)$. On the other hand, under model \ref{NModel:2}, it suffices to  assume $2C\pfail<1$ and then we can set $t=(((2C)^{-1}-\pfail)^{2}-\frac{\max\{d_1,d_2\}}{m})$. The probability of the event $\delta\leq 1/2$ is then at least $1 - c_1(\exp\left( -c_2(m((2C)^{-1}-\pfail)^2-\max\{d_1,d_2\}))\right)$. Finally using the bound $\max\{d_1,d_2\}\leq d_1+d_2\leq \frac{m}{c_3}$ yields the result.
\end{proof}

\paragraph{Step 3: Final estimate.}
Putting the directional and norm estimates together, we arrive at the following theorem.

\begin{thm}\label{thm:final_est_init}
	There exist numerical constants $c_0,c_1, c_2,c_3, C > 0$ such that if $\pfail\leq c_0$ and 
	$m \geq c_4 (d_1 + d_2)$, then for all $t\in [0,1]$, with probability at least 
	$1 -c_1\exp\left( -c_3mt\right),$
	we have 
	
	$$
	\frac{\norm{w_0 x_0^\top - \wbar \xbar^\top}_F}{\|\bar w\bar x^\top\|_F} \leq \begin{cases}
	C \cdot \left(
	\sqrt{\frac{\max\{d_1, d_2\}}{m} + t}\right)& \text{ under Model~\ref{NModel:1}, and} \\
	C \cdot \left(\pfail + 
	\sqrt{\frac{\max\{d_1, d_2\}}{m} + t}\right)& \text{ under Model~\ref{NModel:2}.}
	\end{cases}
	$$
	
\end{thm}
\begin{proof}
	Suppose that we are in the events guaranteed by Propositions~\ref{prop:directional_init_ok},\ref{prop:radius_estimate_ok}, and \ref{prop:small_delta}. Then noting that 
	\[
	w_0 = \sign(\estrad) |\estrad|^{1/2} \widehat{w}, \; x_0 = |\estrad|^{1/2} \widehat{x},
	\]
	we find that 
	\begin{align*}
	\norm{w_0 x_0^\top - \wbar \xbar^\top}_F    &=
	\norm{\sign(\estrad) |\estrad| \widehat{w} \widehat{x}^\top - 
		M 
		\dirw \dirx^\top}_F \\
	&= M\norm{ \widehat{w} \widehat{x}^\top - 
		\sign(\estrad)\dirw \dirx^\top + \frac{|\estrad| - M}{M}\widehat{w} \widehat{x}^\top }_F \\
	&\leq M \norm{\what \xhat^\top - \sign(\estrad)\wstar \xstar^\top}_F + M\delta \\
	&= M\cdot\left(2 + \frac{c_5}{c_6(1-2\pfail)}\right)\min_{s\in\{\pm 1\}} \left\|\widehat w \widehat
	x^\top  -s  \dirw \dirx^\top\right\|_F,
	\end{align*}
	where $c_5$ and $c_6$ are defined  in Theorem~\ref{theo:RIP}. Appealing to Proposition~\ref{prop:directional_init_ok}, the result follows.
\end{proof}

Combining Corollary~\ref{cor:generic_conv_iso} and Theorem~\ref{thm:final_est_init}, we arrive at the following guarantee for the stage procedure.

\begin{corollary}[Efficiency estimates]\label{cor:final_generic_conv_iso}
	Suppose either of the models \ref{NModel:1} and \ref{NModel:2}. Let $(w_0,x_0)$
	be the output of the initialization Algorithm~\ref{alg:spectral_init}. Set $\widehat{M}=\|w_0x_0^\top\|_F$ and consider the optimization problem 
	\begin{equation}\label{eqn:feas_region_weapply}
	\min_{\|x\|_2, \|w\|_2\leq  \sqrt{2\widehat{M}}}~ g(w,x)=\frac{1}{m}\|\cA(wx^\top) - y\|_1.
	\end{equation}
	Set $\nu:=\sqrt{\frac{2\widehat M}{M}}$ and notice that the feasible region of \eqref{eqn:feas_region_weapply} coincides with $\cS_{\nu}$. Then there exist constants $c_0,c_1,c_2,c_3,c_5>0$ and $c_4\in (0,1)$ such that as long as $m\geq c_3(d_1+d_2)$ and $\pfail\leq c_0$, the following properties hold with probability $1-c_1\exp(-c_2m)$.\footnote{In the case of model \ref{NModel:1}, one can set $c_0=1/10$.}
	\begin{enumerate}
		\item {\bf (subgradient)} Both Algorithms~\ref{alg:polyak} and \ref{alg:geometrically_step} (with appropriate $\lambda, q$) initialized  $(x_0,y_0)$  produce iterates that converge linearly to $\cS^\ast_{\nu}$, that is 
		\begin{equation*}
		\frac{\dist^2((w_{k},x_{k}),\cS^\ast_{\nu})}{\|\bar w\bar x^\top\|_F}\leq c_4\left(1-c_4\right)^k\qquad \forall k\geq 0.
		\end{equation*}
		\item {\bf (prox-linear)}  Algorithm~\ref{alg:prox_lin} initialized at $(w_0,x_0)$ (with appropriate $\beta>0$) converges quadratically:
		$$\frac{\dist((w_k,x_{k}),\cS^\ast_{\nu}))}{\sqrt{\|\bar w\bar x^\top\|_F}}\leq c_5\cdot 2^{-2^{k}}\qquad \forall k\geq 0.$$
	\end{enumerate}
\end{corollary}
\begin{proof}
	We provide the proof under model \ref{NModel:1}. The proof under model \ref{NModel:2} is completely analogous.
	Combining Proposition~\ref{prop:radius_estimate_ok}, Proposition~\ref{prop:small_delta}, and Theorem~\ref{thm:final_est_init}, we deduce that there exist constants $c_0,c_1,c_2,c_3,C$ such that as long as $m\geq c_3 (d_1+d_2)$ and $\pfail<c_0$, then for any $t\in [0,1 ]$, with probability $1 - c_1\exp\left( -c_2mt\right)$, we have
	
	\begin{equation}\label{eqn:mbound}
	\left|\frac{\widehat{M}}{M}-1\right|\leq \delta \leq \frac{1}{2},
	\end{equation}
	and
	$$\frac{\norm{w_0 x_0^\top - \wbar \xbar^\top}_F}{M}\leq C\sqrt{\frac{\max\{d_1, d_2\}}{m}+t}.$$
	In particular, notice from \eqref{eqn:mbound} that $1\leq \nu\leq \sqrt{3}$ and therefore the feasible region $\cS_{\nu}$ contains an optimal solution of the original problem \eqref{formulation}. Using Proposition~\ref{prop:l2sharpness}, we have
	\begin{align*}
	\|w_0 x^\top_0 - \wbar \xbar^\top\|_F  \geq \frac{\sqrt{M}}{2\sqrt{2}(\nu+1)} \dist \big((w_0,x_0), \cS^\ast_\nu\big).
	\end{align*}
	Combining the estimates, we conclude
	$$\frac{\dist((w_0,x_0),\cS^\ast_{\nu})}{\sqrt{M}}\leq 2\sqrt{2}(\nu+1) \frac{\|w_0 x^\top_0 - \wbar \xbar^\top\|_F}{M}\leq 2\sqrt{2}(\nu+1)C   \sqrt{\frac{\max\{d_1, d_2\}}{m}+t}.$$	
	Thus to ensure the relative error assumption \eqref{eqn:rel_error_cond}, it suffices to ensure the inequality
	$$2\sqrt{2}(\nu+1)C   \sqrt{\frac{\max\{d_1, d_2\}}{m}+t}\leq \frac{c_6\left(1- 2\pfail\right)}{4\sqrt{2}c_5(\nu+1)},$$
	where $c_5, c_6$ are the constants from Corollary~\ref{cor:coverg_guarant_distrib}. Using the bound $\nu\leq \sqrt{3}$, it suffices to set 
	
	$$t= \left(\frac{c_6(1-2p)}{16\sqrt{3}c_5 C}\right)^2-\frac{\max\{d_1,d_2\}}{m}.$$
	Thus the probability of the desired event becomes
	$1 - c_2(\exp\left( -c_3(c_4m-\max\{d_1,d_2\})\right)$
	for some constant $c_4$. Finally, using the bound $\max\{d_1,d_2\}\leq d_1+d_2\leq \frac{m}{c_3}$ and applying Corollary~\ref{cor:coverg_guarant_distrib} completes the proof.
\end{proof}

\section{Numerical Experiments} \label{sec:experiments}
In this section we demonstrate the performance and stability of the prox-linear and subgradient methods, and the initialization procedure,  
when applied to real and artificial instances of Problem~\eqref{formulation}. All 
experiments were performed using the \texttt{Julia}~\cite{bezanson2017julia} programming language.

\paragraph{Subgradient method implementation.}
Implementation of the subgradient method for Problem~\eqref{formulation} is simple, and has low per-iteration cost. Indeed, one may simply choose the subgradient
\[\frac{1}{m}\sum_{i=1}^m \sign(\dotp{\ell_i, w} \dotp{x, r_i} - y) \left(\dotp{x,r_i} \begin{bmatrix}
\ell_i \\
0
\end{bmatrix} + \dotp{\ell_i, w} \begin{bmatrix}
0 \\ r_i
\end{bmatrix}\right) \in \partial f(w,x), \]
where $\sign(t)$ denotes the sign of $t$, with the convention $\sign(0) =0.$ The cost of computing this subgradient is on the order of four matrix multiplications. 
When applying Algorithm~\ref{alg:geometrically_step}, choosing the correct parameters is important, since its convergence is especially sensitive to the value of the step-size decay $q$; the experiment described in
Section~\ref{sec:step-size-decay}, which aided us empirically in choosing
$q$ for the rest of the experiments, demonstrates this phenomenon. Setting
$\lambda = 1.0$ seemed to suffice for all the experiments depicted hereafter.

\paragraph{Prox-linear method implementation.}
Recall that the convex models used by the prox-linear method take the form:

\begin{equation}\label{eq:BS_linearization}
f_{(w_k, x_k)} (w,x)=\frac{1}{m}\|\cA(w_kx^\top_k+w_k(x-x_k)^\top+(w-w_k)x^\top_k)-y\|_1
\end{equation}
Equivalently, one may rewrite this expression as a Least Absolute Deviation (LAD) objective: 
\begin{align*}
f_{(w_k, x_k)}(w, x) &=
\frac{1}{m} \sum_{i=1}^m \Big|
\underbrace{\left( \begin{array}{c | c}
	\langle x_k,r_i\rangle \ell_i^\top & \langle \ell_i,w_k\rangle r_i^\top
	\end{array} \right)}_{ =: A_i}
\underbrace{\begin{pmatrix} w - w_k \\ x - x_k \end{pmatrix}}_{=:z}
- \underbrace{(y_i - \langle \ell_i,w_k\rangle\langle x_k,r_i\rangle )}_{ =: \tilde y_i} \Big| \;  \\
&= \frac{1}{m} \norm{A z - \tilde{y}}_1.
\end{align*}
Thus, each iteration of Algorithm~\ref{alg:prox_lin} requires solving a strongly convex optimization problem:
\begin{align*}
z_{k+1} &= \argmin_{z \in \cS_\nu} \left\{\frac{1}{m} \norm{Az - \tilde{y}}_1 
+ \frac{1}{2 \alpha} \norm{z}_2^2 \right\}.
\end{align*}
Motivated by the work of \cite{duchi_ruan_PR} on robust phase retrieval, we solve this subproblem with the \emph{graph splitting} variant of the Alternating Direction Method of Multipliers, as described in~\cite{PariBoyd14}. This iterative method applies to problems of the form
\begin{align*}
&\min_{ z\in \cX}~  \; \frac{1}{m} \norm{t - \tilde{y}}_1
+ \frac{1}{2 \alpha} \norm{z}_2^2 \\
&\mbox{s.t. }~~  t = Az.
\end{align*}
Yielding the following subproblems, which are repeatedly executed:
\begin{align*}
z' &\gets \argmin_{z \in \cS_\nu} \set{
	\frac{1}{2 \alpha} \norm{z}_2^2 + \frac{\rho}{2}
	\norm{z - (z_k - \lambda_k)}_2^2
} \\
t' &\gets \argmin_{t} \set{
	\frac{1}{m} \norm{t - \tilde{y}}_1 + \frac{\rho}{2}
	\norm{t - (t_k - \nu_k)}_2^2} \\
\begin{pmatrix} z_{+} \\ t_{+} \end{pmatrix} &\gets
\begin{bmatrix}
I_{d_1 + d_2} & A^\top \\ A & -I_{m}
\end{bmatrix}^{-1} \begin{bmatrix}
I_{d_1 + d_2} & A^\top \\ \mathbf{0} & \mathbf{0}
\end{bmatrix} \begin{pmatrix} z' + \lambda \\ t' + \nu
\end{pmatrix} \\
\lambda_{+} &\gets \lambda + (z' - z_{+}), \nu_{+} \gets
\nu + (t' - t_{+}),
\end{align*}
where $\lambda \in \RR^{d_1 + d_2}$ and $\nu \in \RR^m$ are dual multipliers and 
$\rho > 0$ is a control parameter. 
Each above step may be computed analytically. We found in our experiments that choosing $\alpha = 1$ and $\rho \sim 
\frac{1}{m}$ yielded fast convergence. Our stopping criteria for this subproblem is considered met when the primal residual satisfies $\|(z_+, t_+) - (z, t)\| \leq \epsilon_k \cdot 
\left(\sqrt{d_1 + d_2} + \max \set{\norm{z}_2, \norm{t}_2} \right)$ and the dual residual satisfies $\|(\lambda_+, \nu_+) - (\lambda, \nu)\| \leq \epsilon_k \cdot \left( \sqrt{d_1 + d_2}
+ \max \set{\norm{\lambda}_2, \norm{\nu}_2} \right)$ with $\epsilon_k = 
2^{-k}$.

\subsection{Artificial Data}
We first illustrate the performance of the prox-linear and subgradient methods under noise model~\ref{NModel:1} with i.i.d. standard Gaussian noise $\xi_i$.
Both methods are initialized with Algorithm~\ref{alg:spectral_init}.
We experimented with Gaussian noise of varying variances, and observed that
higher levels did not adversely affect the performance of our algorithm. This
is not surprising, since the theory suggests that both the objective and the initialization procedure
are robust to gross outliers.
We analyze the performance with problem dimensions $d_1 \in \set{400, 1000}$ and 
$d_2 = 500$ and with number of measurements $m = c \cdot (d_1 + d_2)$ with $c$ varying from $1$ to $8$. In~\cref{fig:synthetic_errs_nonoise} and \ref{fig:proximal_errs_45}, we have depicted how the
quantity $$ \frac{\norm{w_k x_k^\top - \wbar \xbar^\top}_F}{
	\norm{\wbar \xbar^\top}_F} $$ changes per iteration for the prox-linear and subgradient methods. We conducted tests in both the moderate corruption ($\pfail = .25$) and high corruption ($\pfail = .45$) regimes. For both methods, under moderate corruption ($\pfail = .25$) we see that exact recovery is possible as long as $c \geq 5$. Likewise, even in high corruption regime ($\pfail = .45$) exact recovery is still possible as long as $c \geq 8$.
We also illustrate the performance of Algorithm~\ref{alg:polyak} when there is
no corruption at all in~\cref{fig:synthetic_errs_nonoise}, which converges
an order of magnitude faster than Algorithm~\ref{alg:geometrically_step}.

In terms of algorithm performance, we see that the prox-linear method takes  few outer iterations,  approximately 15, to achieve very high accuracy, while the subgradient method requires a few hundred iterations. This behavior is expected as the prox-linear method converges quadratically and the subgradient method converges linearly. Although the number of iterations of the prox-linear method is small, we demonstrate in the sequel that its total run-time, including the cost of solving subproblems, can be higher than the subgradient method.

\begin{figure}[!h]
	\centering
	\includegraphics[width=\linewidth]{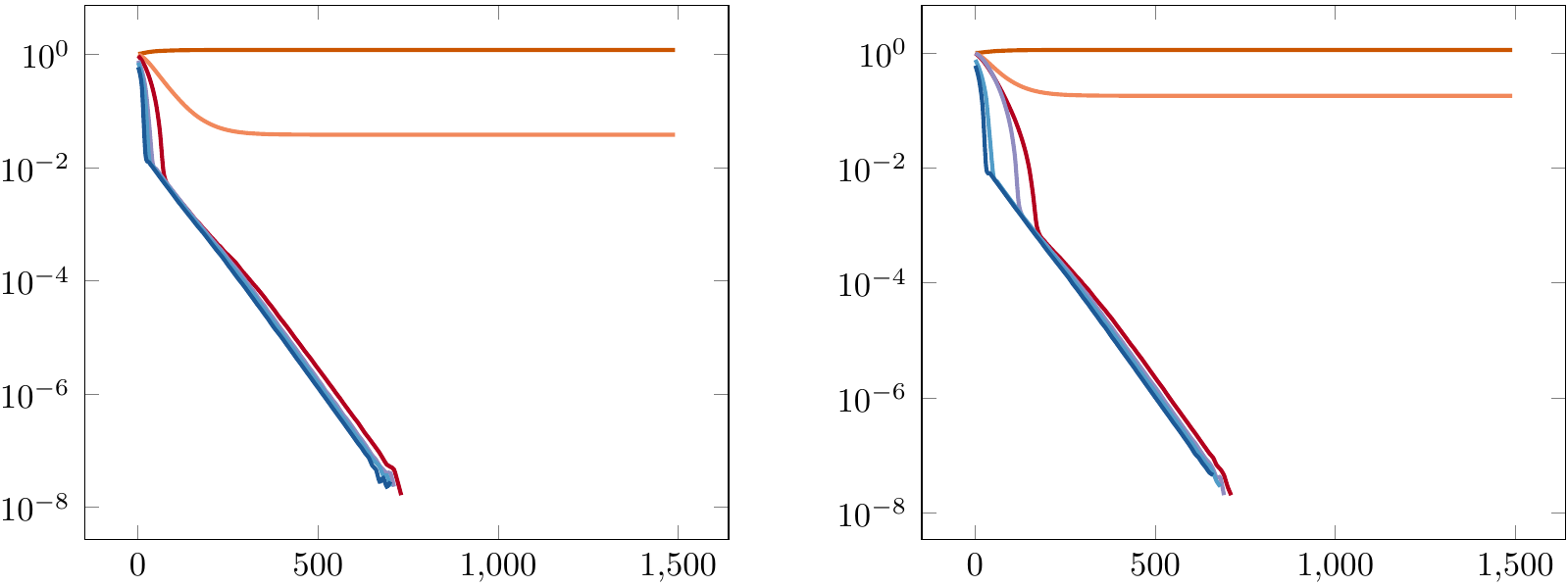}
	\includegraphics[width=\linewidth]{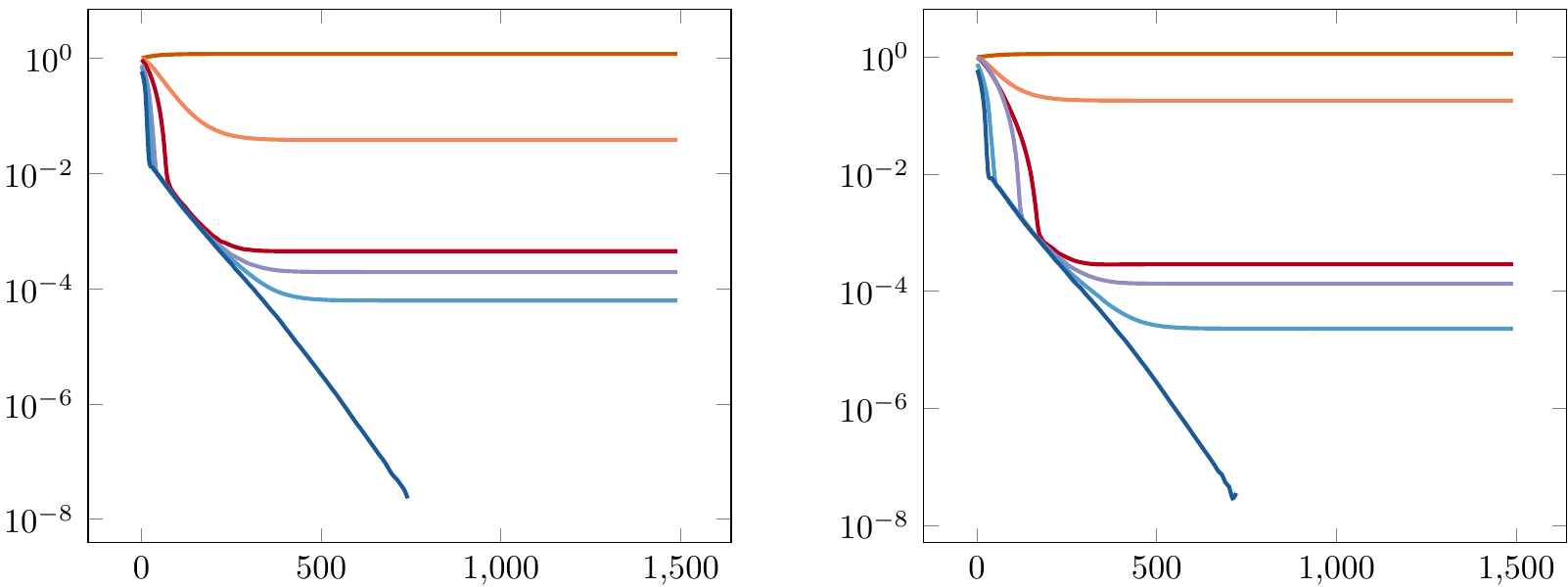}
	\includegraphics[width=\linewidth]{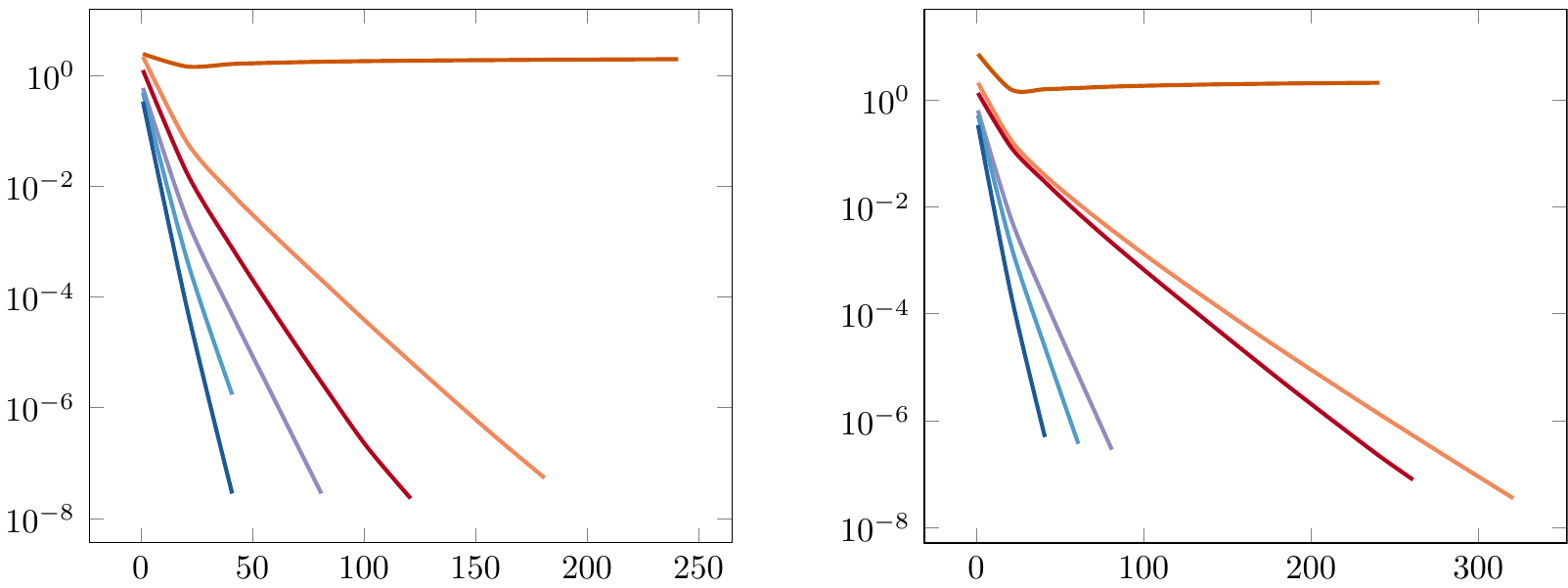}
	\includegraphics[width=0.7 \linewidth]{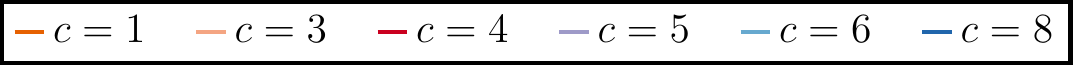}
	\caption{Dimensions are $(d_1,d_2)=(400,500)$ in the first column and $(d_1, d_2) = (1000, 500)$ in the second column. We plot the error $\norm{w_k x_k^\top - \wbar \xbar^\top}_F
		/ \norm{\wbar \xbar^\top}_F$ vs iteration count. Top row is using Algorithm~\ref{alg:geometrically_step} with  $\pfail = 0.25$. Second row is using Algorithm~\ref{alg:geometrically_step} with  $\pfail = 0.45$. Third row is using Algorithm~\ref{alg:polyak} with $\pfail=0$.}
	\label{fig:synthetic_errs_nonoise}
\end{figure}

\begin{figure}[!h]
	\centering
	\includegraphics[width=\linewidth]{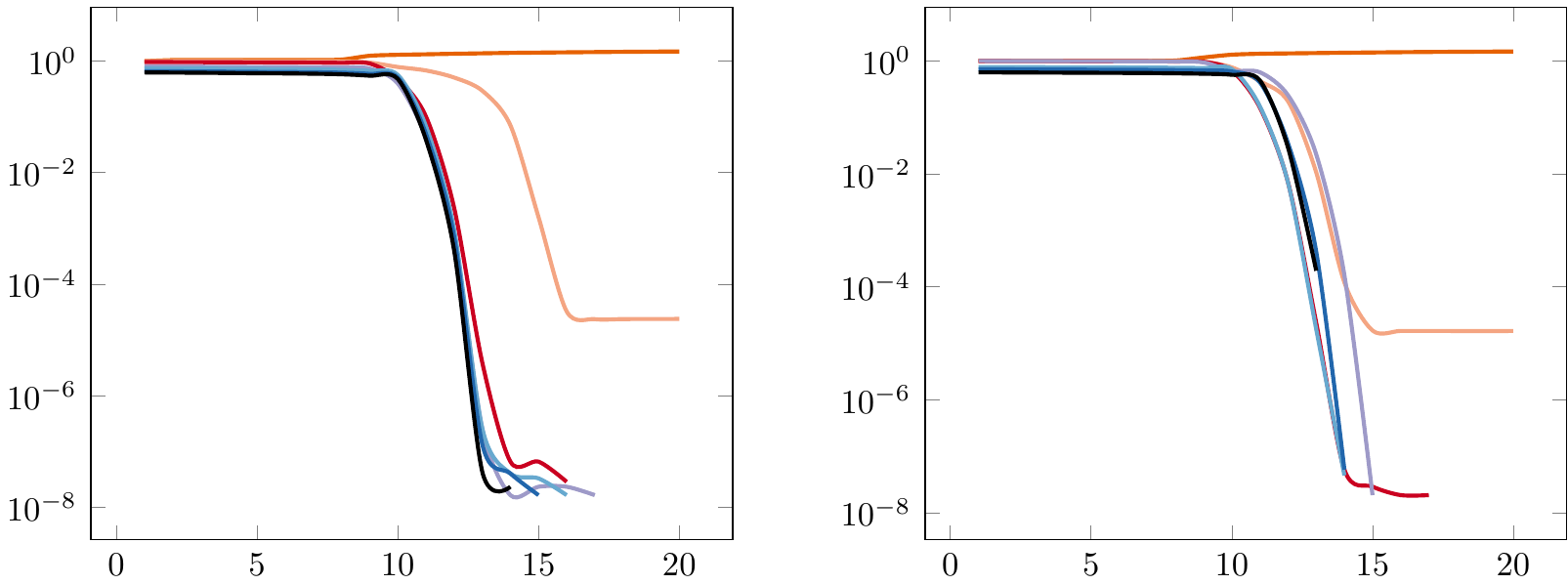}
	\includegraphics[width=\linewidth]{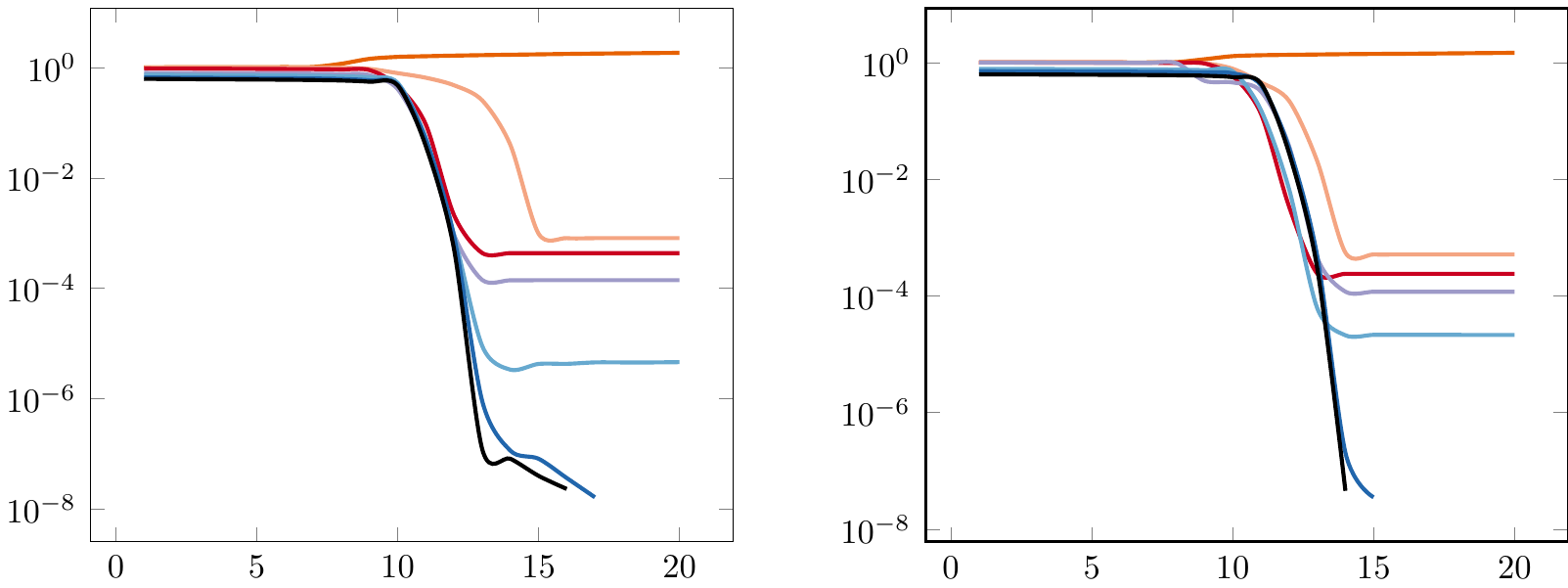}
	\includegraphics[width=0.7 \linewidth]{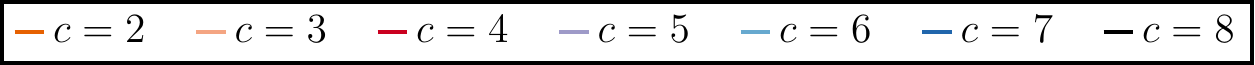}
	\caption{Dimensions are $(d_1,d_2)=(400,500)$ in the first column and
		$(d_1, d_2) = (1000, 500)$ in the second column. We plot the error $\norm{w_k x_k^\top - \wbar \xbar^\top}_F
		/ \norm{\wbar \xbar^\top}_F$ vs iteration count  for an application of Algorithm~\ref{alg:prox_lin} in the two settings: $\pfail = 0.25$ (top row) and $\pfail = 0.45$ (bottom row).}
	\label{fig:proximal_errs_45}
\end{figure}

\subsubsection{Number of matrix-vector multiplications}
Each iteration of the prox-linear method requires the numerical resolution of a convex optimization problem. We solve this subproblem using the  \emph{graph splitting} ADMM algorithm, as described in~\cite{PariBoyd14}, the cost of which is dominated by the number of matrix vector products required to reach the target accuracy. The number of ``inner iterations" of the prox-linear method and thus the number of matrix vector products is not determined a priori. The cost of each iteration of the subgradient method, on the other hand, is on the order of 4 matrix vector products.  
In the subsequent plots, we solve a sequence of synthetic problems for $d_1 
= d_2 = 100$ and keep track of the total number of matrix-vector multiplications 
performed.
We run both methods until we obtain $\frac{\norm{wx^\top - 
		\bar{w} \bar{x}^\top}_F}{\norm{\wbar \xbar^\top}_F} \leq 10^{-5}$.
Additionally, we keep track of the same statistics for the subgradient method.
We present the results in~\cref{fig:matvec-prox}. We
observe that the number of matrix-vector multiplications required by the prox-linear method can be much greater than those required by the subgradient method. Additionally, they seem
to be much more sensitive to the ratio $\frac{m}{d_1 + d_2}$.

\begin{figure}[h]
	\centering
	\includegraphics[width=0.5\linewidth]{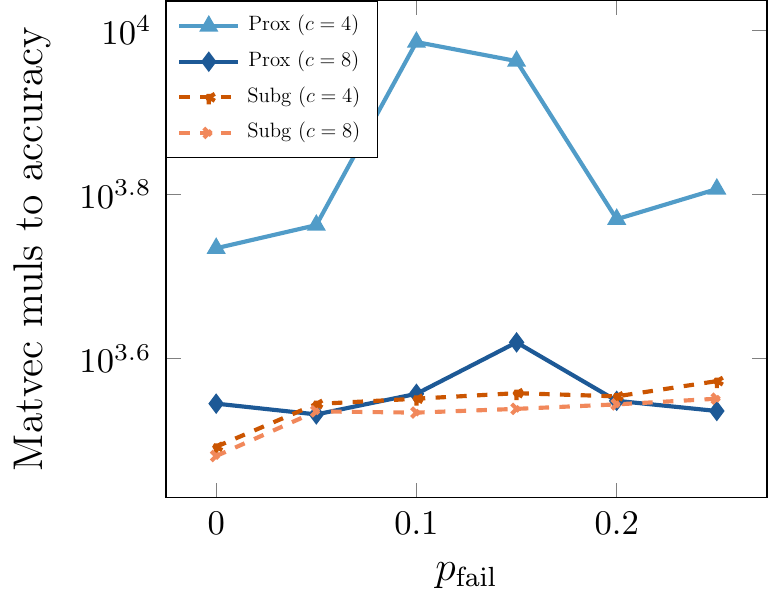}
	\caption{Matrix-vector multiplications to reach
		rel. accuracy of $10^{-5}$.}
	\label{fig:matvec-prox}
\end{figure}

\subsubsection{Choice of step size decay} \label{sec:step-size-decay}
Due to the sensitivity of Algorithm~\ref{alg:geometrically_step} to the step size decay $q$,
we experiment with different choices of $q$ in order to find an empirical range of values
which yield acceptable performance. To that end, we generate synthetic problems of dimension
$100 \times 100$ and choose $q \in \{0.90, 0.905, \dots, 0.995\}$, and record the average error
of the final iterate after $1000$ iterations of the subgradient method for different choices
of $m = c \cdot (d_1 + d_2)$. The average is taken over $50$ test runs with $\lambda = 1.0$.
We test both noisy and noiseless instances to see if corruption of entries significantly
changes the effective range of $q$. Results are shown
in~\cref{fig:final-dist-q-100x100-0.25}.

\begin{figure}[!h]
	\includegraphics[width=0.5\linewidth]{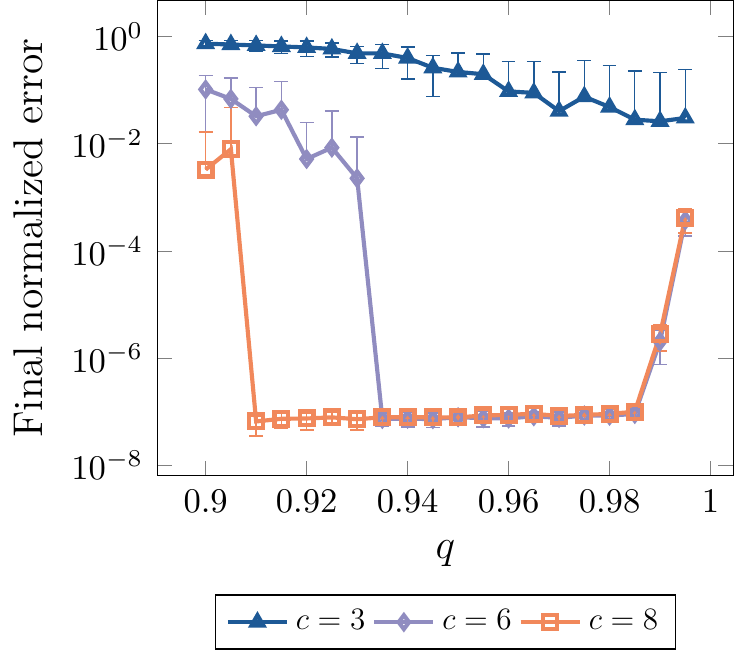}
	\includegraphics[width=0.5\linewidth]{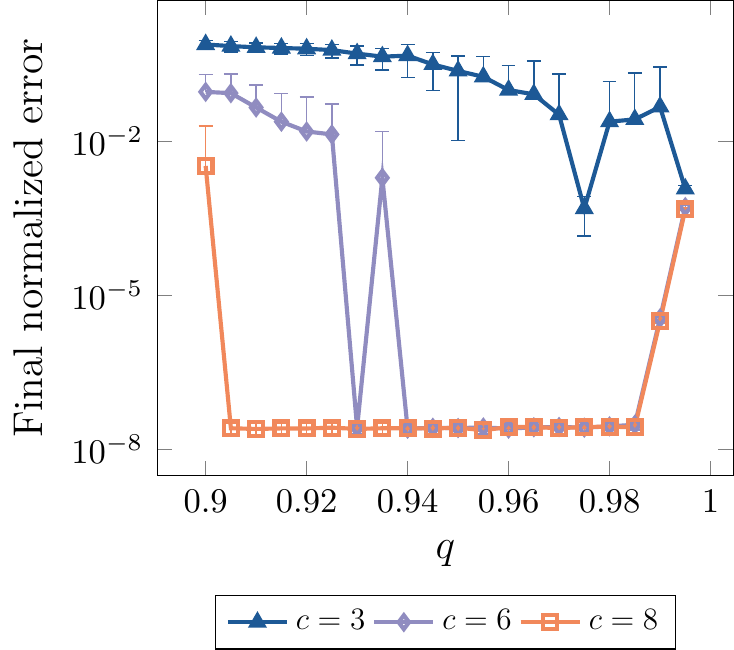}
	\caption{Final normalized error $\norm{w_k x_k^\top - \wbar \xbar^\top}_F
		/ \norm{\wbar \xbar^\top}_F$ for Algorithm~\ref{alg:geometrically_step} with different choices of $q$, in the settings $\pfail=0$ (left) and $\pfail = 0.25$ (right).}
	\label{fig:final-dist-q-100x100-0.25}
\end{figure}

\subsubsection{Robustness to noise}
We now empirically validate the robustness of the prox-linear and subgradients algorithms 
to noise. In a setup familiar from other recent 
works~\cite{duchi_ruan_PR,ahmed2014blind}, we generate \textit{phase transition 
	plots}, where the $x$-axis varies with the level of corruption $\pfail$, the
$y$-axis varies as the ratio $\frac{m}{d_1 + d_2}$ changes, and the shade of each pixel represents the percentage of problem instances solved successfully. For 
every configuration $(\pfail, m / (d_1 + d_2))$, we run $100$ experiments.

\paragraph{Noise model~\ref{NModel:1} - independent noise}

Initially, we experiment with Gaussian random matrices and $(d_1, d_2) \in 
\set{(100, 100), (200, 200)}$, the results for which can be found
in~\cref{fig:phase_tr_gaussian}.

The phase transition plots are similar for both dimensionality 
choices, revealing that in the moderate independent noise regime ($\pfail 
\leq 25\%$), setting $m \geq 4(d_1 + d_2)$ suffices. On the other hand, for exact recovery
in high noise regimes ($\pfail \simeq 45\%$), one may need to choose $m$ as
large as $8 \cdot (d_1 + d_2)$.

We repeat the same experiment in the setting where the matrix $L$ is
deterministic and has orthogonal columns of Euclidean norm $\sqrt{m}$, and $R$ is a gaussian random matrix. 
Specifically, we take $L$ to be a partial Hadamard matrix, from the first $d_1$ columns of an $m \times m$ Hadamard matrix. In that
case, the operator $v \mapsto L v$ can be computed efficiently in 
$\mathcal{O}(m \log m)$ time by $0$-padding $v$ to length $m$ and computing
its Fast Walsh-Hadamard Transform (FWHT). Additionally, the products $w \mapsto L^\top w$ 
can also be computed in $\mathcal{O}(m \log m)$ time 
by taking the FWHT of $w$ and keeping the first $d_1$ coordinates of the result.

The phase transition plots can be found 
in~\cref{fig:phase_tr_dethadm}. A comparison with the phase transition plot 
in~\cref{fig:phase_tr_gaussian} shows a different trend. In this case, exact
recovery does not occur when the noise is above $\pfail \simeq 20\%$ and $m$ is in the range $\set{1, \dots, 8}$.

\paragraph{Noise model~\ref{NModel:2} - arbitrary noise}
We now repeat the previous experiments, but switch to noise model~\ref{NModel:2}. In particular, we now adversarially hide a different signal in a subset of
measurements, i.e., we set
\[
y_i = \begin{cases}
\ip{\ell_i, \wbar} \ip{\xbar, r_i}, & i \notin \cIi, \\
\ip{\ell_i, \wbar_{\mathrm{imp}}} \ip{\xbar_{\mathrm{imp}}, r_i} &
i \in \cIo,
\end{cases}
\]
where in the above $(\wbar_{\mathrm{imp}}, \xbar_{\mathrm{imp}}) \in \RR^{d_1}
\times \RR^{d_2}$ is an arbitrary pair of signals. Intuitively, this is a
more challenging noise model than~\ref{NModel:1}, since it allows an adversary  try to trick the algorithm into recovering an entirely different signal.
Our experiments confirm that this regime is indeed more difficult for the proposed algorithms, which is why
we only depict the range $\pfail \in [0, 0.38]$ 
in~\cref{fig:phase_tr_gaussian_adv,fig:phase_tr_dethadm_adv} below.

\begin{figure}[h!]
	\centering
	\fontsize{8}{10}\selectfont
	\begin{tikzpicture}
	\begin{axis}[
	hide axis, scale only axis, height=0pt, width=0pt, colormap/blackwhite,
	colorbar horizontal, point meta min=0, point meta max=1,
	colorbar style={
		width=10cm,
		xtick={0,0.1,0.2,0.4,...,1}
	}]
	\addplot [draw=none] coordinates {(0,0)};
	\end{axis}
	\end{tikzpicture}
	
	\vspace{1em}
	
	\begin{minipage}{0.49 \textwidth}
		\def\svgwidth{\linewidth}
		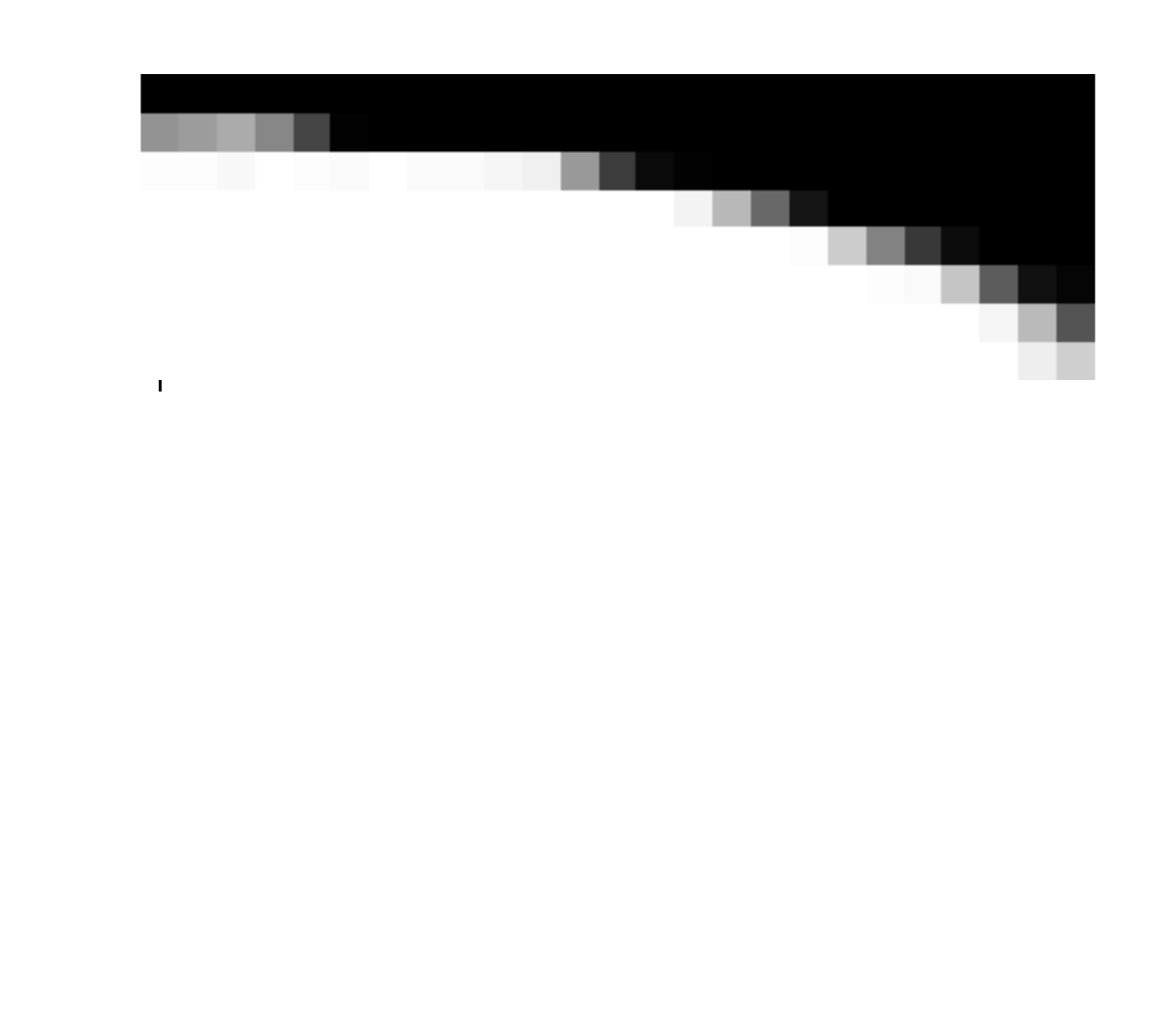
		\caption{Phase transition for~\ref{MModel:1},~\ref{NModel:1}.}
		\label{fig:phase_tr_gaussian}
	\end{minipage}
	\begin{minipage}{0.49 \textwidth}
		\def\svgwidth{\linewidth}
		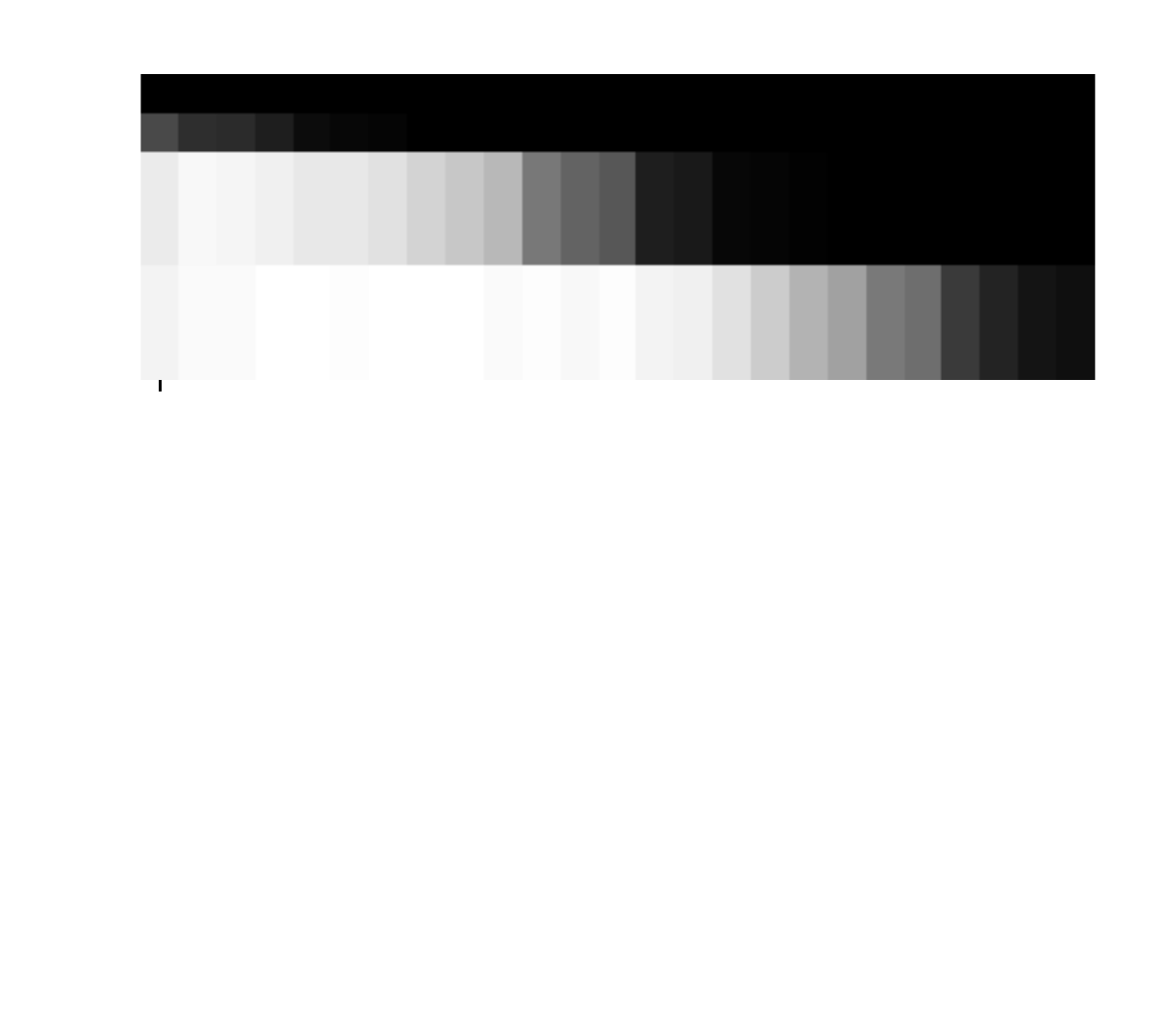    
		\caption{Phase transition for~\ref{MModel:2},~\ref{NModel:1}.}
		\label{fig:phase_tr_dethadm}
	\end{minipage}
	\begin{minipage}{0.49 \textwidth}
		\def\svgwidth{\linewidth}
		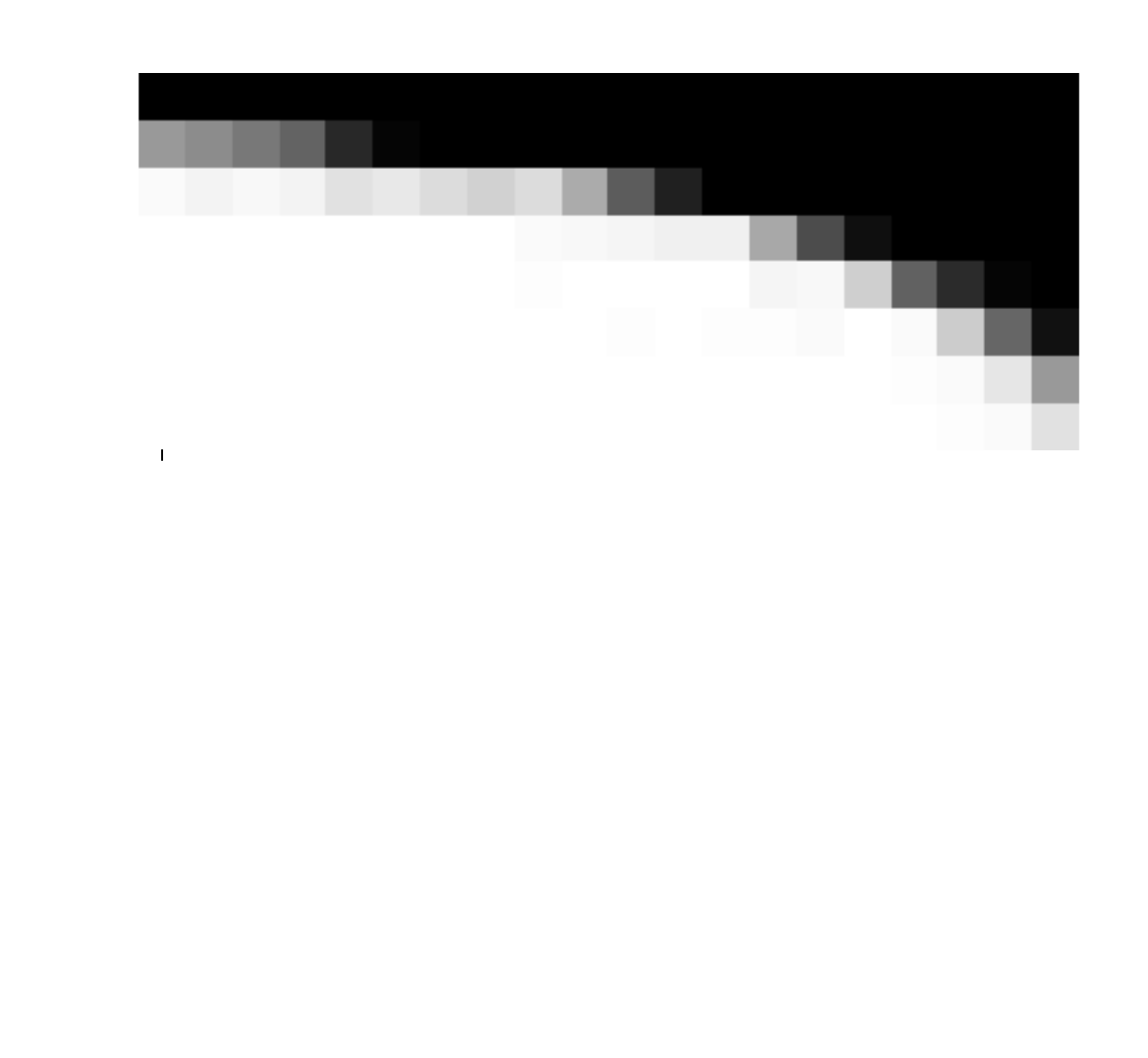
		\caption{Phase transition for~\ref{MModel:1},~\ref{NModel:2}.}
		\label{fig:phase_tr_gaussian_adv}
	\end{minipage}
	\begin{minipage}{0.49 \textwidth}
		\def\svgwidth{\linewidth}
		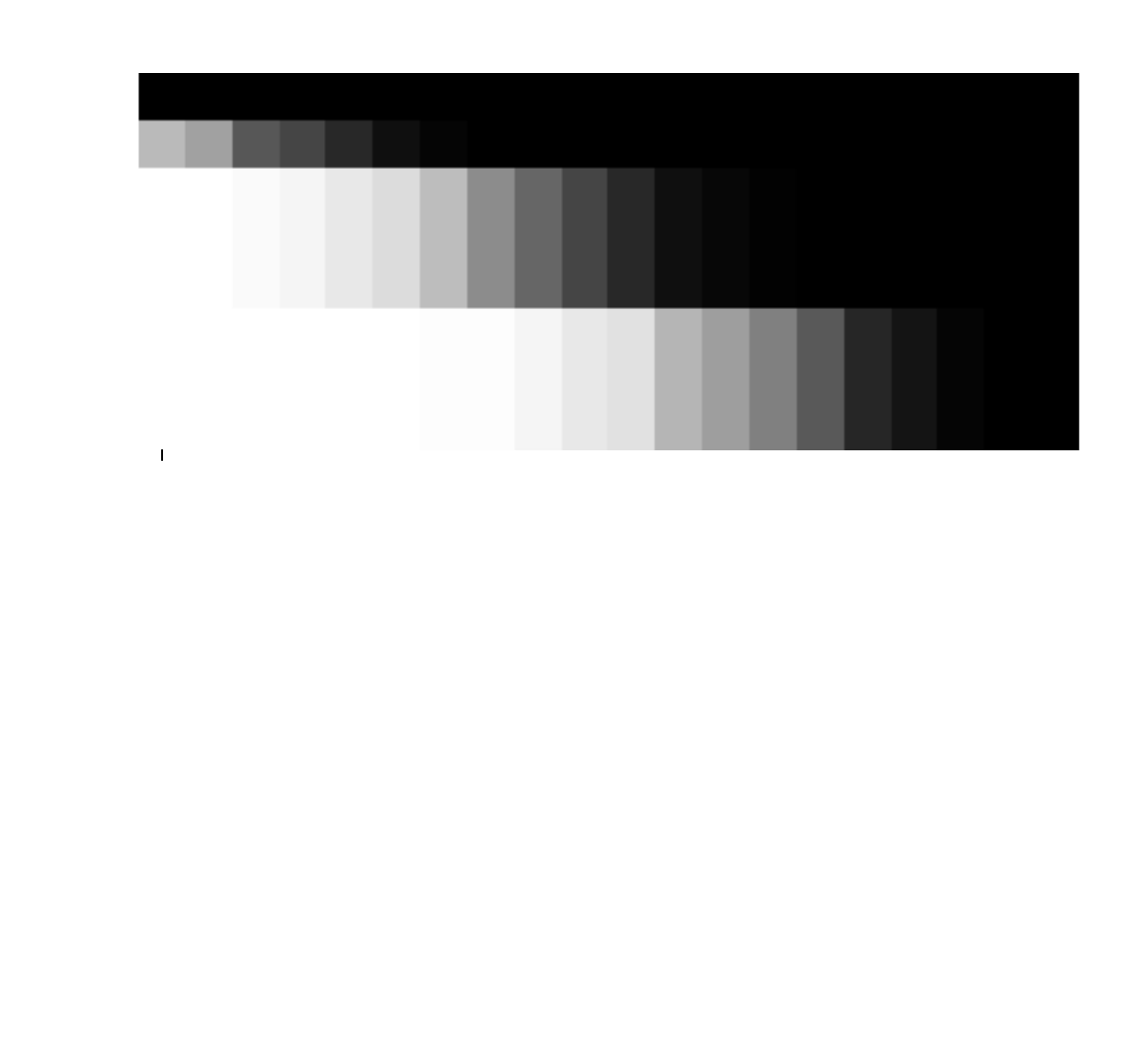
		\caption{Phase transition for~\ref{MModel:2},~\ref{NModel:2}.}
		\label{fig:phase_tr_dethadm_adv}
	\end{minipage}
\end{figure}

\subsection{Performance of initialization on real data}
We now demonstrate the proposed initialization strategy on real world images.
Specifically, we set $\wbar$ and $\xbar$ to be two random digits from
the training subset of the MNIST dataset~\cite{LecBotBen+98}. In this experiment, 
the measurement matrices $L, R \in \RR^{(16 
	\cdot 784) \times 784}$ have i.i.d.\ Gaussian entries, and the noise follows
Model~\ref{NModel:1} with $\pfail = 0.45$.
We apply the initialization method and plot the resulting 
images (initial estimates) in~\cref{fig:mnist_example_2}.
Evidently, the initial estimates of the images are visually similar to the true digits, up to sign; in other
examples, the foreground appears to be switched with the background, which corresponds to the natural
sign ambiguity. Finally, we plot the normalized error for the two recovery methods (subgradient and prox-linear)
in~\cref{fig:mnist-error-prox}.

\begin{figure}[h!]
	\centering
	\begin{minipage}{0.24 \textwidth}
		\includegraphics[width=\linewidth]{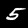}
	\end{minipage}
	\begin{minipage}{0.24 \textwidth}
		\includegraphics[width=\linewidth]{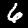}
	\end{minipage}
	\begin{minipage}{0.24 \textwidth}
		\includegraphics[width=\linewidth]{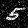}
	\end{minipage}
	\begin{minipage}{0.24 \textwidth}
		\includegraphics[width=\linewidth]{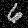}
	\end{minipage}
	\begin{minipage}{0.24 \textwidth}
		\includegraphics[width=\linewidth]{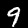}
	\end{minipage}
	\begin{minipage}{0.24 \textwidth}
		\includegraphics[width=\linewidth]{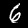}
	\end{minipage}
	\begin{minipage}{0.24 \textwidth}
		\includegraphics[width=\linewidth]{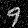}
	\end{minipage}
	\begin{minipage}{0.24 \textwidth}
		\includegraphics[width=\linewidth]{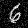}
	\end{minipage}
	\caption{Digits $5, 6$ (top) and $9, 6$ (bottom). Original images are shown
		on the left, estimates on the right. Parameters: $\pfail = 0.45, m = 16 
		\cdot 784.$}
	\label{fig:mnist_example_2}
\end{figure}

\begin{figure}[!h]
	\centering
	\begin{minipage}{0.45 \textwidth}
		\begin{tikzpicture}% table
		\begin{axis}[xlabel=$k$,ylabel=Error, ymode=log, width=\linewidth]
		\addplot[mark=square*, color=hred, thick, mark repeat=50] table[
		y=err, x=iter, col sep=comma]
		{mnist_error_900_450.csv};
		\addplot[mark=triangle*, color=hblue, thick, mark repeat=50]
		table[y=err, x=iter, col sep=comma]
		{mnist_error_4096_8192.csv};
		\legend{\scalebox{0.75}{\cref{fig:mnist_example_2}(top)},
			\scalebox{0.75}{\cref{fig:mnist_example_2}(bottom)}};		\end{axis}
		\end{tikzpicture}
	\end{minipage}
	\begin{minipage}{0.45 \textwidth}
		\begin{tikzpicture}% table
		\begin{axis}[xlabel=$k$,ylabel=Error, ymode=log, width=\linewidth]
		\addplot[mark=asterisk, color=hred, thick] table[y=err, x=iter, col sep=comma]
		{mnist_error_proximal_900_450.csv};
		\addplot[mark=x, color=hblue, thick] table[y=err, x=iter, col sep=comma]
		{mnist_error_proximal_4096_8192.csv};
		\legend{\scalebox{0.75}{\cref{fig:mnist_example_2}(top)},
			\scalebox{0.75}{\cref{fig:mnist_example_2}(bottom)}};
		\end{axis}
		\end{tikzpicture}
	\end{minipage}
	\caption{Relative error vs iteration count on \texttt{mnist} digits for subgradient method (left) and prox-linear method (right).}
	\label{fig:mnist-error-prox}
\end{figure}
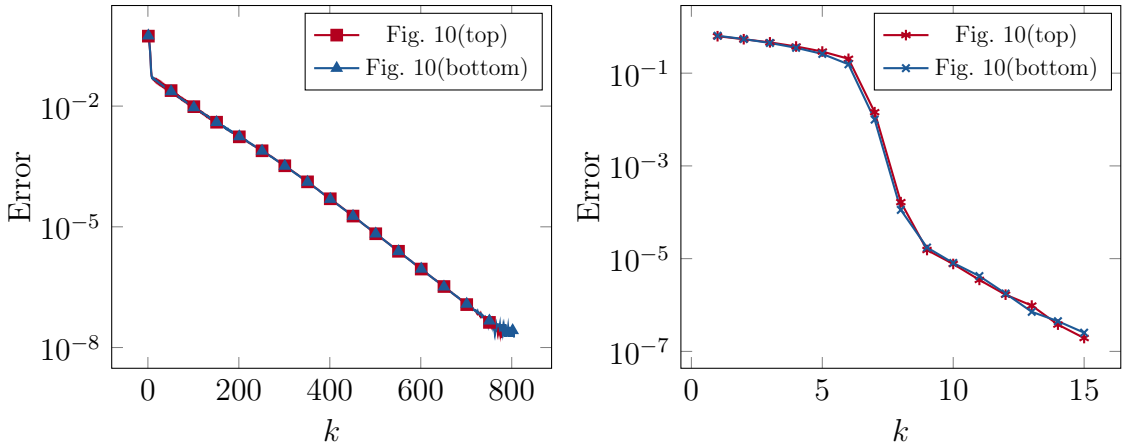

\subsection{Experiments on Big Data}
We apply the subgradient method for recovering large-scale real color images
$W, X \in \RR^{n \times n \times 3}$. In this setting, $\pfail = 0.0$ so using
Algorithm~\ref{alg:polyak} is applicable with $\min_{\mathcal{X}} f = 0$.
We ``flatten'' the matrices $W, X$ into $3n^2$ dimensional vectors $w, x$. In contrast to the previous experiments, our sensing matrices are of the following form:
\[
L = \begin{bmatrix}
H S_1 \\
\vdots \\
H S_k
\end{bmatrix}, \; R = \begin{bmatrix}
H S'_1 \\
\vdots \\
H S'_k
\end{bmatrix},
\]
where $H \in \set{-1, 1}^{d \times d} / \sqrt{d}$ is the $d \times d$ symmetric
normalized Hadamard matrix and $S_i = \mathrm{diag}(\xi_1, \dots, \xi_d), \;
\xi \sim_{\mathrm{i.i.d}} \set{-1, 1}$ is a diagonal random sign matrix. The
same holds for $S'_i$. Notice that we can perform the operations $w \mapsto L 
w, \; x \mapsto R x$ in $\mathcal{O}(k d \log d)$ time: we first form the 
elementwise product between the signal and the random signs, and then take its 
Hadamard transform, which can be performed in $\mathcal{O}(d \log d)$ flops. We can 
efficiently compute $p \mapsto L^\top p, \; q \mapsto R^\top q$, required for
the subgradient method, in a similar fashion.
We recover each channel separately, which means we essentially have to solve
three similar minimization problems. Notice  that this results in dimensionality 
$d_1 = d_2 = n^2, \; m = kn^2$ for each channel. 

We observed that our initialization procedure (Algorithm~\ref{alg:spectral_init}) is extremely accurate in this setting. Therefore to better illustrate the performance of the local search algorithms, we perform the following heuristic initialization.
For each channel, we first sample $\widehat{w}, \widehat{x} \sim \mathbb{S}^{d - 1}$, rescale by the true magnitude of the signal, and run Algorithm~\ref{alg:polyak} for one step to obtain our initial estimates $w_0, x_0$.

An example where  we recover a  pair of $512 \times 512$ color images using the Polyak subgradient method (Algorithm~\ref{alg:polyak}) is shown below;~\cref{fig:w-iterates} shows the progression of the estimates $w_k$, up until the $90$-th iteration, while \cref{fig:img-rec-error} depicts the normalized error at each iteration for the different channels of the images.

\begin{figure}[!h]
	\centering
	\begin{minipage}[t]{0.33 \textwidth}
		\centering
		\includegraphics[width=0.95\linewidth]{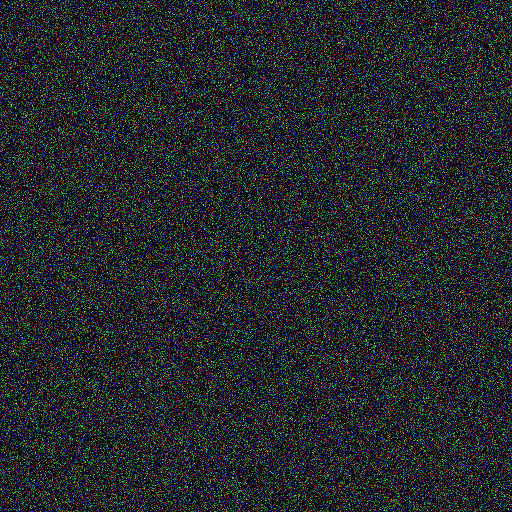}
	\end{minipage}~
	\begin{minipage}[t]{0.33 \textwidth}
		\centering
		\includegraphics[width=0.95\linewidth]{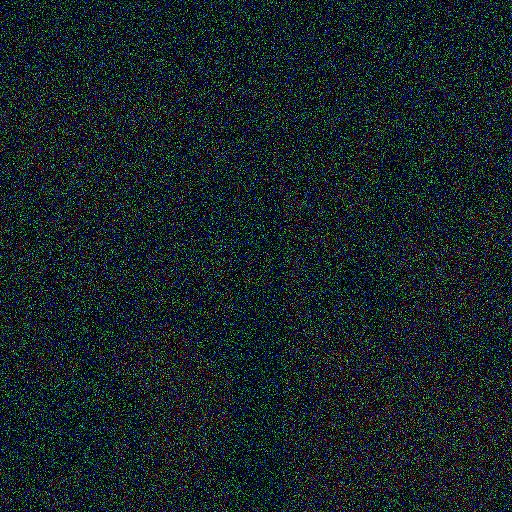}
	\end{minipage}~
	\begin{minipage}[t]{0.33 \textwidth}
		\centering
		\includegraphics[width=0.95\linewidth]{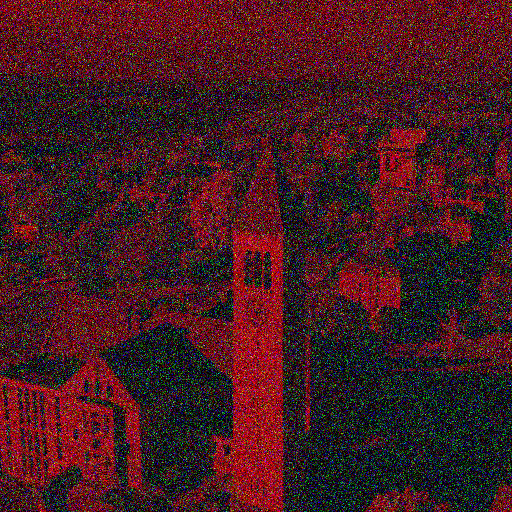}
	\end{minipage}
	
	\hfill\vline\hfill
	
	\begin{minipage}[t]{0.33 \textwidth}
		\centering
		\includegraphics[width=0.95\linewidth]{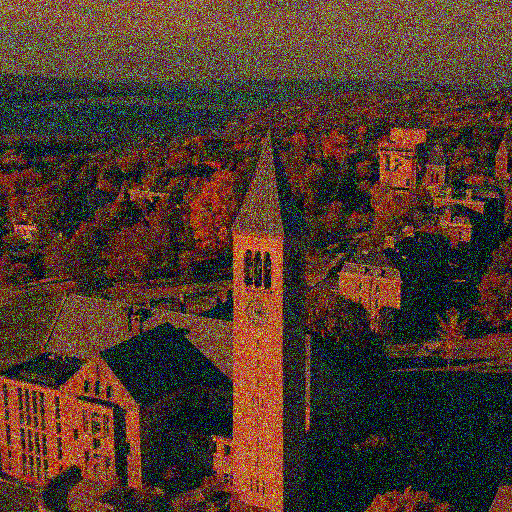}
	\end{minipage}~
	\begin{minipage}[t]{0.33 \textwidth}
		\centering
		\includegraphics[width=0.95\linewidth]{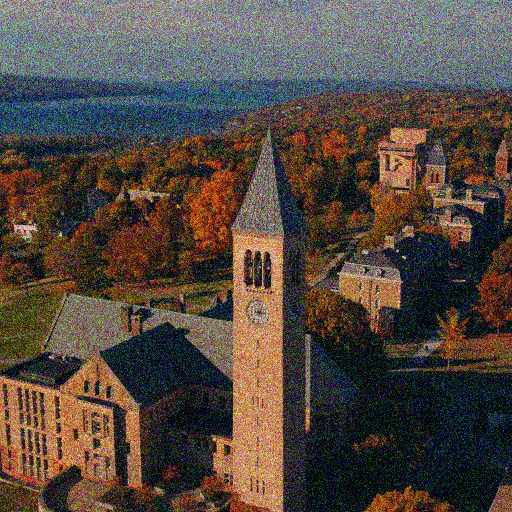}
	\end{minipage}~
	\begin{minipage}[t]{0.33 \textwidth}
		\centering
		\includegraphics[width=0.95\linewidth]{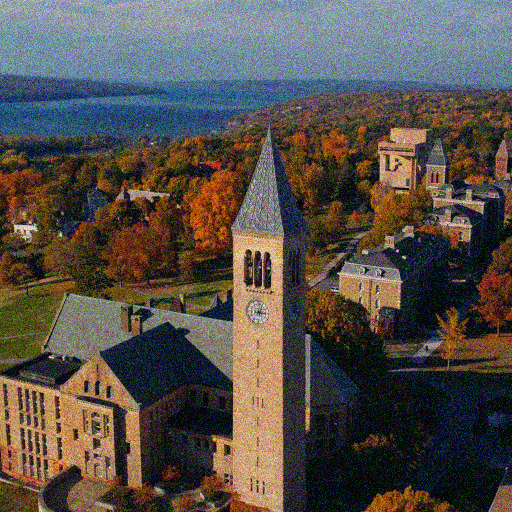}
	\end{minipage}
	
	\hfill\vline\hfill
	
	\begin{minipage}[t]{0.33 \textwidth}
		\centering
		\includegraphics[width=0.95\linewidth]{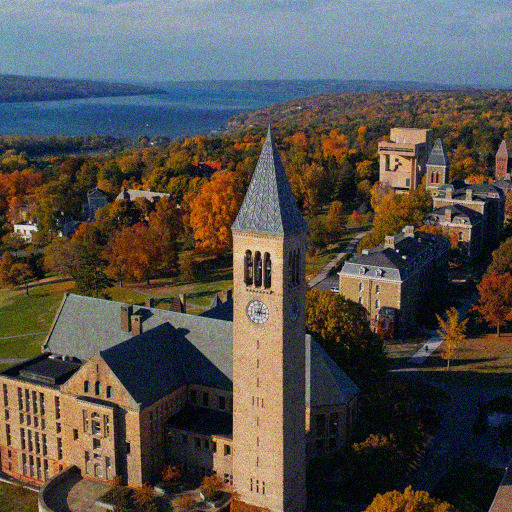}
	\end{minipage}~
	\begin{minipage}[t]{0.33 \textwidth}
		\centering
		\includegraphics[width=0.95\linewidth]{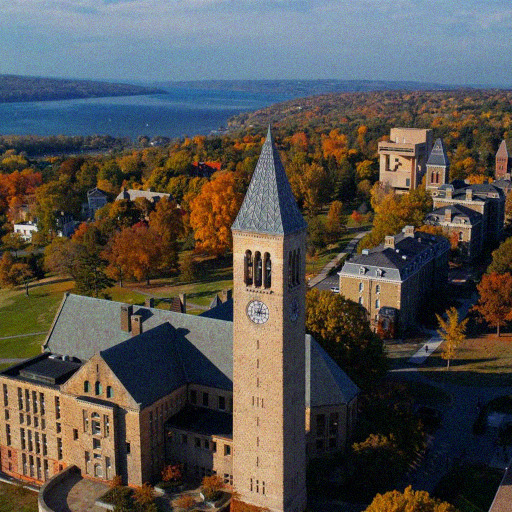}
	\end{minipage}~
	\begin{minipage}[t]{0.33 \textwidth}
		\centering
		\includegraphics[width=0.95\linewidth]{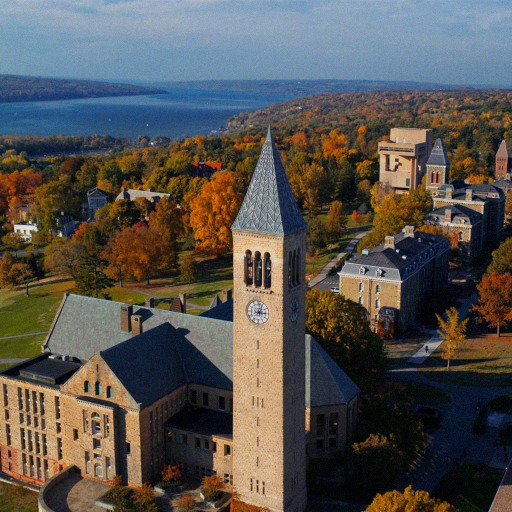}
	\end{minipage}
	\caption{Iterates $w_{10i}, \; i = 1, \dots, 9$.
		$(m, k, d, n) = (2^{22}, 16, 2^{18}, 512)$.}
	\label{fig:w-iterates}
\end{figure}

% Style to select only points from #1 to #2 (inclusive)
\pgfplotsset{select coords between index/.style 2 args={
		x filter/.code={
			\ifnum\coordindex<#1\def\pgfmathresult{}\fi
			\ifnum\coordindex>#2\def\pgfmathresult{}\fi
		}
}}

\begin{figure}[!h]
	\centering
	\begin{tikzpicture}[scale=0.75]% table
	\begin{axis}[xlabel=$k$,ylabel=Error, ymode=log, width=0.6\linewidth]
	\addplot[color=hred, thick, mark=x, mark repeat=50] table[
	select coords between index={0}{378}, y=err_red, col sep=comma,
	x=iter]
	{img_errors_random_512x512.csv};
	\addplot[color=mblue, thick, mark=triangle*, mark repeat=50] table[
	select coords between index={0}{395}, y=err_green, col sep=comma,
	x=iter]
	{img_errors_random_512x512.csv};
	\addplot[color=lred, thick, mark=*, mark repeat=50] table[
	select coords between index={0}{404}, y=err_blue, col sep=comma,
	x=iter]
	{img_errors_random_512x512.csv};
	\legend{Red, Blue, Green};
	\end{axis}
	\end{tikzpicture}
	\caption{Normalized error for different channels in image recovery.}
	\label{fig:img-rec-error}
\end{figure}
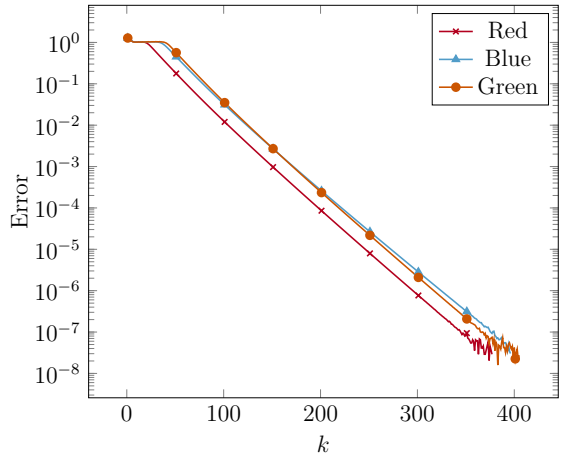

\bibliographystyle{plain}
\bibliography{bibliography}

\begin{appendices}
\section{Sharpness}\label{sec:appendix_sharp}

\subsection{Proof of Proposition~\ref{prop:l2sharpness}}\label{sec:appendix_sharp_baby_function}

Without loss of generality, we assume that $M = 1$ (by rescaling) and that $\bar w = e_1 \in \RR^{d_1}$ and $\bar x = e_1 \in \RR^{d_2}$ (by rotation invariance). Recall that  the distance to $\cS^\ast_{\nu}$ may be written succinctly as 
$$
\dist((w,x), \cS^\ast_{\nu})  = \sqrt{\inf_{ (1/\nu) \leq |\alpha| \leq \nu} \left\{\|w - \alpha \bar w\|^2_2 +  \|x - 
	({1}/{\alpha})\bar x\ \|^2_2\right\}}.
$$

Before we establish the general result, we first consider the simpler case, $d_1 = d_2 = 1$.
\begin{claim}\label{claim:sharp_claim}
	The following bound holds:
	\begin{align*}
	|wx - 1| \geq \frac{1}{\sqrt{2}} \cdot  \sqrt{\inf_{(1/\nu) \leq |\alpha| \leq \nu } \left\{|w - \alpha |^2 +  |x - 
		({1}/{\alpha}) |^2\right\}},
	\end{align*}
	for all $w, x\in [-\nu, \nu]$.
\end{claim}
\begin{proof}[Proof of Claim]
	Consider a pair $(w,x)\in\R^2$ with $|w|,|x|\leq \nu$.
	It is easy to see that without loss of generality, we may assume $ w\geq |x|$. We then separate the proof into two cases, which are graphically depicted in Figure~\ref{fig:partition}.
	\begin{figure}[t]
		\centering
		\includegraphics[width=4cm]{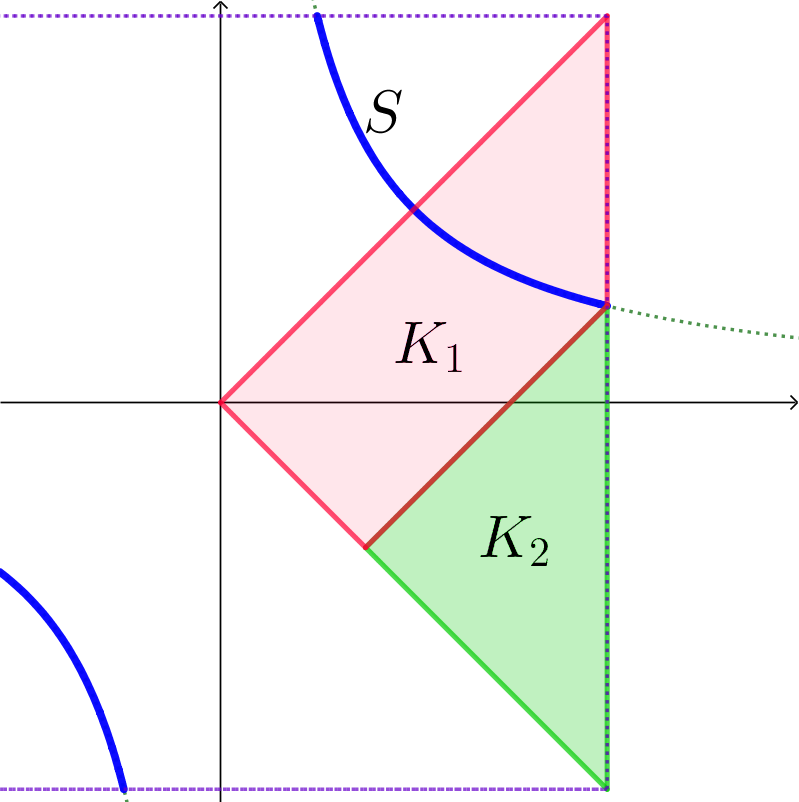}
		\caption{The regions $K_1$, $K_2$ correspond to cases 1 and 2 of the proof of Claim~\ref{claim:sharp_claim}, respectively.}\label{fig:partition}
	\end{figure}
	\paragraph{Case 1: $w-x \leq \frac{\nu^2 - 1}{\nu}$.} 
	In this case, we will traverse from $(w,x)$ to the $\cS_{\nu}^*$ in the direction $(1,1)$. See Figure~\ref{fig:partition}.
	First, consider the equation
	$$
	wx - \sqrt{2}(w + x)t + t^2/2 = 1, 
	$$
	in the variable $t$ and note the equality 
	$$wx - \sqrt{2}(w + x)t + t^2/2=(w - t/\sqrt{2})(x - t/\sqrt{2}).$$ Using the quadratic formula to solve for $t$, we get
	$$
	t = \sqrt{2}(w+x) - \sqrt{ 2(w+x)^2 - 2(wx - 1)}.
	$$
	Note that the discriminant is nonnegative since $
	(w+x)^2 - (wx - 1) = w^2 + x^2 + xw + 1 \geq 1.
	$ 
	
	Set $\alpha = (w - t/\sqrt{2})$ and note the identity $1/\alpha  = (x - t/\sqrt{2})$. Therefore,
	\begin{align*}
	|wx - 1| &= | (1/\alpha)(w - \alpha) + \alpha(x-1/\alpha) + (w-\alpha)(x-1/\alpha)| \\
	&= | (x - t/\sqrt{2}) (t/\sqrt{2}) + (w-t/\sqrt{2})(t/\sqrt{2}) + t^2/2|\\
	&= \frac{|t|}{\sqrt{2}}| (w+x) - t/\sqrt{2}| = \frac{|t|}{2} \sqrt{ 2(w+x)^2 - 2(wx - 1)} \geq \frac{|t|}{\sqrt{2}}.
	\end{align*}
	Observe now the equality
	$$
	\frac{|t|}{\sqrt{2}} = \frac{1}{\sqrt{2}} \cdot (|w - \alpha|^2 + |x - 1/\alpha|^2)^{1/2}.
	$$
	Hence it remains to bound $\alpha$. First we note that $\alpha \geq 0, 1/\alpha \geq 0$, since 
	\begin{align*}
	\alpha + 1/\alpha &= (w - t/\sqrt{2}) + (x - t/\sqrt{2}) \\
	%&= -x + \sqrt{ (w+x)^2 - (wx - 1)}  - w + \sqrt{ (w+x)^2 - (wx - 1)} \\
	&= -(w + x) +  2\sqrt{ (w+x)^2 - (wx - 1)} \geq 0.
	\end{align*}
	In addition, since $ w \geq x$, we have $\alpha = w - t/\sqrt{2} \geq x - t/\sqrt{2} = 1/\alpha$. Since $\alpha$ and $1/\alpha$ are positive, we must therefore have $\alpha \geq 1 \geq 1/\nu$. Thus, it remains to verify the bound $\alpha \leq \nu$.
	To that end, notice that 
	$$
	1/\alpha = x - t/\sqrt{2} \geq w - t/\sqrt{2} - \frac{\nu^2 - 1}{\nu} = \alpha - \frac{\nu^2 - 1}{\nu}.
	$$
	Therefore, $ \frac{\nu^2 - 1}{\nu} \geq \frac{\alpha^2 - 1}{\alpha}$. Since the function $t\mapsto \frac{t^2 - 1}{t}$ is increasing, we deduce $\alpha \leq \nu$.
	
	\paragraph{Case 2:  $w-x \geq \frac{\nu^2 - 1}{\nu}$.}  In this case, we will simply set $\alpha = \nu$. Define 
	$$
	t = \left((w- \nu)^2 + (x-1/\nu)^2\right)^{1/2}, \qquad a = \frac{w-\nu}{t}, \qquad \text{and} \qquad b = \frac{x - 1/\nu}{t}.
	$$
	Notice that proving the desired bound amounts to showing $|wx-1| \geq \frac{t}{\sqrt{2}}.$
	Observe the following estimates
	$$
	a, b \leq 0, \qquad b\leq a , \qquad a^2+b^2 =1, \qquad \text{and} \qquad t \leq -\frac{1}{(a+b)} \left(\frac{\nu^2+\nu}{\nu}\right),
	$$
	where the the first inequality follows from the bounds $w \leq \nu$ and $\nu \geq w \geq x + \nu - 1/\nu$, second inequality follows from the bound $w - x \geq (\nu^2 - 1)/\nu$, the equality follows from algebraic manipulations, and the third inequality follows from the estimate $ w+ x \geq 0$. Observe
	$$
	|wx -1| = | (\nu + ta )(1/\nu + tb) - 1|  = |t^2ab+ t\nu b + ta/\nu |.
	$$
	Thus, by dividing through by $t$, we need only show that 
	\begin{equation}
	\label{eq:2Dcase2}
	|tab+ \nu b + a/\nu | \geq \frac{1}{\sqrt{2}}.
	\end{equation}
	To prove this bound, note that since $2b^2 \geq a^2+b^2 =1$, we have the  $-\nu b - a/\nu \geq -\nu b \geq 1/\sqrt{2}$. Therefore, in the particular case when $ab = 0$ the estimate~\ref{eq:2Dcase2} follows immediately. 
	Define the linear function $p(s):=-(ab)s- \nu b - a/\nu$.
	Hence, assume $ab \neq 0.$ Notice $p(0)\geq 1/\sqrt{2}$. Thus it suffices to show that the solution $s^*$ of the equation $p(s)=1/\sqrt{2}$ satisfies $s^*\geq t$. To see this, we compute:
	\begin{align*}
	s^* & = -\frac{1}{ab}\left( \nu b + a/\nu  + \frac{1}{\sqrt{2}}\right) \\
	& = -\frac{1}{(a+b)} (a+b)\left(\frac{\nu}{a}  + \frac{1}{b\nu}  + \frac{1}{\sqrt{2}ab}\right) \\
	& = -\frac{1}{(a+b)}  \left({\nu} \left(1+ \frac{b}{a} \right)  + \frac{1}{\nu}\left(1+\frac{a}{b}\right)  +   \frac{1}{\sqrt{2}}\left(\frac{1}{a} + \frac{1}{b}\right) \right) \\
	& \geq -\frac{1}{(a+b)}  \left({\nu}   + \frac{1}{\nu}\left(1+\frac{a}{b}+ \frac{b}{a}  + \frac{1}{\sqrt{2} b}    +\frac{1}{\sqrt{2}a} \right)	 \right) \\
	& = -\frac{1}{(a+b)}  \left({\nu}   + \frac{1}{\nu}\left(1+\frac{\sqrt{2}\left(a^2+b^2\right) - (|a|+|b|)}{\sqrt{2}ab} \right)	 \right) \\ 
	& = -\frac{1}{(a+b)}  \left({\nu}   + \frac{1}{\nu}\left(1+\frac{\sqrt{2} - (|a|+|b|)}{\sqrt{2}ab} \right)	 \right) \\
	& \geq -\frac{1}{(a+b)}  \left({\nu}   + \frac{1}{\nu}\right) \geq t, 
	\end{align*}
	where the first inequality follows since $\nu \geq 1$ and the second inequality follows since $a^2+b^2=1$ and  $\sqrt{2}\|(a,b)\|_2 \geq \|(a,b)\|_1$, as desired.  
\end{proof}

Now we prove the general case. First suppose that $\|wx^\top - \bar w \bar x^\top\|_F \geq 1/2$. Since $\| w - \bar w \|_2 \leq (\nu + 1)$ and $\| x - \bar x\|_2 \leq (\nu+1)$, we have 
$$
\dist((w, x), \cS^\ast_{\nu}) \leq \sqrt{2}(\nu+1) \leq 2\sqrt{2}(\nu+1) \|wx^\top - \bar w \bar x^\top\|_F,
$$
which proves the desired bound. 

On the other hand, suppose that $\|wx^\top - \bar w \bar x^\top\|_F < 1/2$. Define the two vectors:
$$
\tilde w 
= (w_1, 0, \dots, 0)^\top \in \RR^{d_1} \qquad \text{and} \qquad \tilde x  = (x_1, 0, \dots, 0)^\top \in\ \RR^{d_2}.
$$
With this notation, we find that by Claim~\ref{claim:sharp_claim}, there exists an $\alpha$ satisfying $(1/\nu) \leq |\alpha| \leq \nu$, such that the following holds:
\begin{align*}
\|wx^\top - \bar w \bar x^\top \|^2_F &= \|wx^\top - \tilde w \tilde x^\top+ \tilde w \tilde x^\top 
-\bar w \bar x^\top \|^2_F\\
&= \|wx^\top - \tilde w \tilde x^\top\|_F^2 + \|\tilde w \tilde x^\top 
-\bar w \bar x^\top \|^2_F \\
&\geq \|wx^\top -  \tilde w \tilde x^\top\|_F^2 + 
\frac{1}{2}\left(\|\tilde w - \alpha \bar w \|^2_F + \|\tilde x - (1/\alpha)\bar x 
\|^2_F\right).
\end{align*}
We now turn our attention to lower bounding the first term. Observe since  
$|w_1x_1-\bar w_1 \bar x_1|\leq \|wx^T - \bar w \bar x^T\|_F < 1/2$, we have
\[
| w_1x_1|\geq |\bar w_1 \bar x_1| - |w_1x_1 - \bar w_1 \bar 
x_1| \geq (1/2)|\bar w_1 \bar x_1| = 1/2, 
\]
Moreover, note the estimates, $\nu |w_1| \geq |x_1||w_1| \geq 1/2$ and  $\nu |x_1| \geq |x_1||w_1| \geq 1/2$, which imply that $|w_1| \geq 1/2\nu$ and $|x_1| \geq 1/2\nu$. Thus, we obtain the lower bound 
\begin{align*}
\|wx^\top - \tilde w \tilde x^\top\|_F^2 &= \| (w-\tilde w)\tilde x^\top +\tilde w(x-\tilde x)^\top+(w-\tilde w)(x-\tilde x)^\top\|^2_F\\
&= |x_1|^2\|w - \tilde w\|^2 + |w_1|^2\|x - \tilde x\|^2_2 + 
\|(w - \tilde w) (x - \tilde x)^\top \|^2_F  \\
& \geq |x_1|^2\|w - \tilde w\|^2_2 + |w_1|^2\|x - \tilde x\|^2_2 \\
& \geq \left(\frac{1}{2\nu}\right)^2\left(\|w - \tilde w\|^2_2 + \|x - 
\tilde x\|^2_2\right).
\end{align*}
Finally, we obtain the bound
\begin{align*}
\|wx^\top - \bar w \bar x^\top \|^2_F & \geq \|wx^\top - \tilde w \tilde x^\top\|_F^2 + 
\frac{1}{2}\left(\|\tilde w - \alpha \bar w \|^2_F + \|\tilde x - (1/\alpha)\bar x 
\|^2_F\right) \\
& \geq \left(\frac{1}{2\nu}\right)^2\left(\|w - \tilde w\|^2_2 + \|x - 
\tilde x\|^2_2\right) + \frac{1}{{2}}\left(\|\tilde w - \alpha \bar w \|^2_2 + \|\tilde x - 
(1/\alpha)\bar x \|^2_2\right) \\
& \geq 
\min\left\{\frac{1}{{2}},\left(\frac{1}{2\nu}\right)^2\right\}\left(\|w - 
\tilde w\|^2_2 + \|x - \tilde x\|^2_2+\|\tilde w - \alpha \bar w \|^2_2 + \|\tilde x - (1/\alpha)\bar x 
\|^2_2\right) \\
& = 
\left(\frac{1}{2\nu}\right)^2\cdot \dist^2((w, x), \cS^\ast_{\nu}). 
\end{align*}
By recalling that $1/2\nu \geq 1/2\sqrt{2}(\nu+1)$, the proof is complete.

\section{Initialization}
\label{appendix:initialization}

\subsection{Proof of Proposition~\ref{prop:directional_init_ok}}\label{appendix:init_direction}

As stated in Section~\ref{sec:damek_init}, we first verify that $\Linit$ and $\Rinit$ are nearby matrices with minimal eigenvectors equal to $\dirw$ and $\dirx$.  Then we apply the Davis-Kahan $\sin \theta$ theorem~\cite{DavKah70} to prove that the minimal eigenvectors of $\Linit$ and $\Rinit$ must also be close to the optimal directions. 

Throughout the rest of the proof, we define the sets of ``selected" inliers and outliers:
$$
\selinliers = \inliers \cap \sel \qquad \text{and} \qquad \seloutliers = \outliers \cap
\sel. 
$$
We record the relative size of these parameters as well, since they appear in the bounds that follow: 
$$
\Sin := \frac{1}{m} \abs{\selinliers}\qquad \text{and } \qquad  \Sout = \frac{1}{m} \abs{\seloutliers}.
$$
\begin{thm}\label{claim:decomp_init_matrix}
	There exist numerical constants $c_1, c_2, c_3,c_4, c_5 > 0$, so that for any $\pfail\in [0,1/10]$ and $t\in [0,1]$,  with probability at least 
	$1 - c_1(\exp\left( -c_2mt\right)$
	the following hold:
	\begin{enumerate}
		\item Under noise model~\ref{NModel:1}
		\begin{align*}
		\Linit = (\Sin + \Sout) I_{d_1} - \gamma_1 \dirw \dirw^\top + \Delta_1, \quad
		\Rinit = (\Sin + \Sout) I_{d_2} - \gamma_2 \dirx \dirx^\top + \Delta_2, \quad
		\end{align*}
		where $\gamma_1\geq c_3$ and $\gamma_2\geq c_4$ and 
		\begin{align*}
		\max\{\|\Delta_1\|_{\op},  \|\Delta_2\|_{\op}\}  &\leq c_5 \left(\sqrt{\frac{\max\{d_1,d_2\}}{m}} +t\right). 
		\end{align*}
		\item Under noise model~\ref{NModel:2}
		\begin{align*}
		\Linit = \Sin I_{d_1} - \gamma_1 \dirw \dirw^\top + \Delta_1, \quad
		\Rinit = \Sin I_{d_2} - \gamma_2 \dirx \dirx^\top + \Delta_2, \quad
		\end{align*}
		where $\gamma_1\geq c_3$ and $\gamma_2\geq c_4$ and 
		\begin{align*}
		\max\{\|\Delta_1\|_{\op},  \|\Delta_2\|_{\op}\}  &\leq \pfail + c_5 \left(\sqrt{\frac{\max\{d_1,d_2\}}{m}} +t\right). 
		\end{align*}
	\end{enumerate}
\end{thm}
\begin{proof}
	Without loss of generality, we only prove the result for $\Linit$; the result for $\Rinit$ follows by a symmetric argument. 
	
	Define the projection operators $P_{\dirw} := \dirw \dirw^\top$ and let $ P_{\dirw}^\perp := I - \dirw \dirw^\top.$  Then decompose $\Linit$ into the sums of four matrices $Y_0, Y_1,Y_2, Y_3$, as follows:
	\begin{align}
	L^\text{init} &= \dfrac{1}{m} \bigg(
	\underbrace{\sum_{i \in \selinliers} P_\dirw \ell_i
		\ell_i^\top P_\dirw}_{m \cdot Y_0} + \underbrace{\sum_{i \in \selinliers} \left(P_\dirw
		\ell_i \ell_i^\top P_\dirw^\perp + P_\dirw^\perp \ell_i \ell_i^\top P_\dirw
		\right)}_{m \cdot Y_1}
	+ \underbrace{\sum_{i \in \cIis} P_\dirw^\perp \ell_i \ell_i^\top
		P_\dirw^\perp}_{m \cdot Y_2} + \underbrace{\sum_{i \in \seloutliers} \ell_i
		\ell_i^\top}_{m \cdot Y_3} \bigg).
	\label{eq:LR_decomp}.
	\end{align}
	We will now study the properties of these four matrices under both noise models.
	
	First, note that in either case we may write $Y_0 = y_0 \dirw \dirw^\top$, where 
	$$
	y_0 := \frac{1}{m}\sum_{i \in \cIis}(\ell_i^\top \dirw)^2.
	$$
	In addition, we will present a series of Lemmas showing the following high probability deviation bounds: 
	\[
	\gamma_1 := \Sin - y_0 \gtrsim 1, \quad \|Y_1\|_{\op} \lesssim \sqrt{\frac{d_1}{m}}, \quad \text{and}\quad \|Y_2 - \Sin(I_{d_1}-\dirw {\dirw}^\top)\|_\op \lesssim \sqrt{\frac{d_1}{m}}.\]
	Finally, our bounds on the term $Y_3$ as well as the definition of $\Delta_1$ depend on the noise model under consideration. Thus, we separate this bound into two cases:%Then the next four lemmas will show that under that adversarial noise model~\ref{NModel:2} the following bounds hold with high probability
	\paragraph{Noise model~\ref{NModel:1}.} 
	Under this noise model, we have 
	$$ 
	\|Y_3 - \Sout I_{d_1}\|_{\op} \lesssim \sqrt{\frac{d_1}{m}}.
	$$
	Thus, we set 
	$
	\Delta_1 = Y_1 + \left(Y_2 - \Sin(I_{d_1}-\dirw {\dirw}^\top) \right) + \left(Y_3 - \Sout I_{d_1}\right).
	$
	\paragraph{Noise model~\ref{NModel:2}.} 
	Under this noise model, we have 
	$$ 
	\|Y_3 \|_{\op} \lesssim \pfail + \sqrt{\frac{d_1}{m}}.
	$$
	Thus, we set
	$
	\Delta_1 = Y_1 + (Y_2 - \Sin(I_{d_1}-\dirw {\dirw}^\top)) + Y_3 .
	$
	
	Therefore, under either noise model, the result will follow immediately from the following four Lemmas.  We defer the proofs for the moment. 
	\begin{lem} \label{lemma:y0_z0}
		There exist constants $c,c_1,c_2 > 0$ such that for any $\pfail \in [0,1/10]$ the following holds:
		\begin{align*}
		\Prob{\Sin - y_0 \geq c }\geq 1 - c_1\exp\left( -c_2m\right).
		\end{align*} 
	\end{lem}
	
	\begin{lem} \label{lemma:Z1_concentration}
		For $t \geq 0$, we have
		\begin{align*}
		\Prob{\opnorm{Y_1} \geq 2 \sqrt{\frac{d_1 - 1}{m}} + t} &\leq
		\exp\left(-\frac{mt^2}{8}\right) + \exp\left(-\frac{m}{2}\right).
		\end{align*}
	\end{lem}
	
	\begin{lem}
		\label{lemma:Z2_concentration}
		There exist numerical constants
		$C, c > 0$ such that for any $t>0$ we have
		\begin{align*}
		\Prob{\opnorm{Y_2 - \Sin(I_{d_1} - \dirw {\dirw}^\top)} \geq C \sqrt{\frac{d_1}{m} + t}}\leq 2 \exp(-c mt). 
		\end{align*}

	\end{lem}
	
	\begin{lem} \label{lemma:Z3_concentration}
		There exist
		constants $C_1,C_2, c_1,c_2 > 0$ such that for any $t >0$ the following hold. Under the noise Model~\ref{NModel:1}, we have the estimate
		\begin{align*}
		\Prob{\opnorm{Y_3 - \Sout I_{d_1}} \geq c_3 \sqrt{\frac{d_1}{m} + t}} \leq 2 \exp(-c_4 mt),
		\end{align*}
		while under the noise model~\ref{NModel:2} we have
		\begin{align*}
		\Prob{\opnorm{Y_3} \geq \pfail + c_1 \sqrt{\frac{d_1}{m} + t}} \leq 2 \exp(-c_2 mt).
		\end{align*}
	\end{lem}
	The proof of the the theorem is complete. 
\end{proof}

We now apply the Davis-Kahan $\sin \theta$ theorem~\cite{DavKah70} as stated in Lemma~\ref{lemma:davis_kahan_variant}. Throughout we assume that we are in the event described in~\cref{claim:decomp_init_matrix}. %We address each noise model separately.

\begin{proof}[Proof of Proposition~\ref{prop:directional_init_ok}]
	We will use the notation from Theorem~\ref{claim:decomp_init_matrix}.
	We only prove the result under~\ref{NModel:1}, since the proof under~\ref{NModel:2}  is completely analogous. Define matrices $V_1 =  \gamma_1 \dirw \dirw^\top - (\Sin + \Sout) I_{d_1} $ and $ V_2 = \gamma_2 \dirx \dirx^\top - (\Sin + \Sout) I_{d_2} $. Matrix $V_1$ has spectral gap $\gamma_1$ and top eigenvector $\dirw$, while matrix $V_2$ has spectral gap $\gamma_2$ and top eigenvector $\dirx$.  Therefore, since $-\Linit = V_1 - \Delta_1$ and $-\Rinit = V_2- \Delta_2$, Lemma~\ref{lemma:davis_kahan_variant} implies that 
	\begin{align*}
	\min_{s \in \{\pm 1\} }\|\widehat w - s\dirw\|_2 \leq \frac{ \sqrt{2} \opnorm{\Delta_1}}{\gamma_1} && \text{and} && \min_{s \in \{\pm 1\} }\|\widehat x - s\dirx\|_2 \leq \frac{ \sqrt{2} \opnorm{\Delta_2}}{\gamma_2 }.
	\end{align*}
	We will use these two inequalities to bound $\min_{s \in \{\pm 1\}}\|\widehat w \widehat x^\top -s \dirw \dirx^\top\|_F$. To do so, we need to analyze $s_1 = \argmin_{s \in \{\pm 1\} }\|\widehat w - s\dirw\|$ and $s_2 = \argmin_{s \in \{\pm 1\} }\|\widehat x - s\dirx\|$. We split the argument into two cases. 
	
	Suppose first $s_1 = s_2$. Then 
	\begin{align*}
	\|\widehat w \widehat x^\top - \dirw \dirx^\top\|_F = \|\widehat w (\widehat x - s_2\dirx)^\top - (\dirw - s_1\widehat w) \dirx^\top\|_F &\leq \|\widehat x - s_2\dirx\|_2+\|\dirw - s_1\widehat w\|_2 \\
	&\leq \frac{2 \sqrt{2}\max\{\|\Delta_1\|_{\op},  \|\Delta_2\|_{\op}\}}{\min\{\gamma_1, \gamma_2\}},
	\end{align*}
	as desired.
	
	Suppose instead $s_1 = -s_2$. Then 
	\begin{align*}
	\|\widehat w \widehat x^\top + \dirw \dirx^\top\|_F = \|\widehat w (\widehat x - s_2\dirx)^\top + (\dirw + s_2\widehat w) \dirx^\top\|_F &\leq\|\widehat x - s_2\dirx\|_2+ \|\dirw - s_1\widehat w\|_2\\
	&\leq \frac{2 \sqrt{2} \max\{\|\Delta_1\|_{\op},  \|\Delta_2\|_{\op}\}}{\min\{\gamma_1, \gamma_2\}},
	\end{align*}
	as desired. Bounding $\max\{\|\Delta_1\|_{\op},  \|\Delta_2\|_{\op}\}$ using Theorem~\ref{claim:decomp_init_matrix} completes the proof.
	
\end{proof}

The next sections present the proof of Lemmas~\ref{lemma:y0_z0}-\ref{lemma:Z3_concentration}. We next set up the notation. For any sequence of vectors $\{w_i\}_{i=1}^m$ in $\R^d$, we will use the symbol $w_{i, 2:d}$ to denote the vector in $\R^{d-1}$ consisting of the last $d-1$ coordinates of $w_i$.

We will use the following two observations throughout. First, by rotation invariance we will assume, without loss of generality, that $\dirw = e_1$ and $\dirx = e_1$. Second, and crucially, this assumption implies that $\selinliers$ depends on $\{\ell_i\}_{i=1}^m$ only through the first component. In particular, we have that $\{\ell_{i, 2:d_1}\}_{i=1}^m$ and $\selinliers$ are independent. Similarly, $\{r_{i, 2:d_2}\}_{i=1}^m$ and $\selinliers$ are independent as well.

\subsubsection{Proof of Lemma~\ref{lemma:y0_z0}}
Our goal is to lower bound the quantity
$$
\Sin - y_0 = \dfrac{1}{m} \sum_{i \in \selinliers} (1-\ell_{i,1}^2).
$$
To prove a lower bound, we need to control the random variables $\ell_{i, 1}^2$ on the set $\selinliers$. 

Before proving the key claim, we first introduce some notation. First, define 
$$
q_\mathrm{fail} := \frac{5-2\pfail}{8(1-\pfail)}, 
$$ 
which is strictly less than one since $\pfail < 1/2$.
Let $a, b \sim \normal(0, 1)$ and define $\Qfail$ to be 
the $q_\mathrm{fail}$-quantile of the random variable $|ab|$. In particular, the following relationship holds
\[q_\mathrm{fail} = \Prob{\abs{ab} \leq 
	\Qfail} .\]
Additionally, define the conditional expected value
\[ \omegafail = \EE\left[a^2 \mid \abs{ab} \leq \Qfail\right]. \]
Rather than analyzing $\selinliers$ directly, we introduce the following set $\inliers^Q$, which is simpler to analyze: 
\[
\inliers^Q :=  \left\{ i \in \inliers \mid \left|\ell_i^\top \dirw {\dirx}^\top
r_i \right| \leq \Qfail\right\}.
\]
Then we prove the following claim.
\begin{claim}\label{claim:init_dir_4}
	There exist numerical constants $c, K >0$ such that for all $t \geq 0$ the following inequalities hold true:
	\begin{enumerate}
		\item\label{claim:init_dir_4:item0} $\frac{|\selinliers|}{m} \geq \frac{1 - 2\pfail}{2}$.
		\item \label{claim:init_dir_4:item1}$\Prob{\inliers^Q \supseteq \selinliers}   \geq 1 - \exp\left(-\tfrac{3(1-2\pfail)}{160} m \right),$
		\item \label{claim:init_dir_4:item2}$\Prob{|\inliers^Q| \geq \frac{6251m}{10000}} \leq \exp\left(-\frac{m}{2\cdot10^8}\right),$
		\item \label{claim:init_dir_4:item3} $\Prob{\frac{1}{|\inliers^Q|} \sum_{i \in \inliers^Q}  \ell_{i,1}^2
			\geq \omegafail + t} \leq \exp\left(- c \min\left\{\frac{t^2}{K^2},
		\frac{t}{K}\right\} \frac{m(1-2\pfail)}{2} \right) +\exp\left(-\tfrac{3(1-2\pfail)}{160} m \right).$
	\end{enumerate}
\end{claim}
Before we prove the claim, we show it leads to the conclusion of the lemma. Assuming we are in the event 
\begin{align*}
\mathcal{E} = &\left\{\inliers^Q \supseteq \selinliers,
\;\; | \inliers^Q| < \frac{6251m}{10000}, \;\; \frac{1}{|\inliers^Q|} \sum_{i \in \inliers^Q} \ell_{i,1}^2 \leq \frac{101}{100}\omegafail \right\},
\end{align*}
it follows that 
\begin{align*}
S-y_0 = \frac{1}{m} \sum_{i \in \selinliers} (1 - \ell_{i,1}^2)  \geq \frac{|\selinliers|}{m} - \frac{1}{m}\sum_{i \in \inliers^Q} \ell_{i,1}^2 &\geq   \frac{1-2\pfail}{2} - \frac{|\cIi^Q|}{m \abs{\inliers^Q} }\sum_{i \in \inliers^Q} \ell_{i,1}^2 \\
&\geq   \frac{1-2\pfail}{2} - \frac{631351}{1000000}\omegafail \geq 0.04644344.
\end{align*}
where the first three inequalities follow by the definition of the event $\mathcal{E}$. The fourth inequality follows by the definition of $\mathcal{E}$ and Lemma~\ref{lemma:monotonic_tails_condition}, which implies $\omegafail \leq .56$ when $\pfail = .1$ and that the difference is  minimized over $\pfail \in [0, .1]$ at the endpoint $\pfail = .1$.  To get the claimed probabilities, we note that by Lemma~\ref{lemma:monotonic_tails_condition}, we have $\omegafail \geq .5$ for any setting of $\pfail$.

Now we prove the claim. 

\begin{proof}[Proof of the Claim]
	We separate the proof into four parts. 
	
	Part~\ref{claim:init_dir_4:item0}. By definition, we have
	$$
	\frac{|\selinliers|}{m} = \frac{|\inliers \cap \sel |}{m} = \frac{|\sel| - |\outliers \cap \sel |}{m} \geq \frac{\frac{m}{2} - |\outliers \cap \sel |}{m} \geq \frac{\frac{m}{2} - m\pfail}{m} = \frac{1 - 2\pfail}{2}.
	$$
	
	Part~\ref{claim:init_dir_4:item1}. By the definitions of $\selinliers$ and $\inliers^Q$, the result will follow once we show that 
	\begin{align*}
	\Prob{\med(\{|y_i|\}_i^m) \geq \Qfail M} & \leq \exp\left(-\dfrac{3(1-2\pfail)}{160} m \right).
	\end{align*}
	To that end, first note that 
	\begin{align*}
	\med(\{|y_i|\}_i^m) &=\min\left\{ |y_j| \, \colon j \in [m], \, \sum_{i=1}^m \mathbf{1} \{|y_i| \leq |y_j|\} \geq \frac{m}{2} \right\}\\
	&=\min\left\{ |y_j| \, \colon j \in [m], \, \sum_{i=1}^m \mathbf{1} \{|y_i| \leq |y_j|\} \geq \frac{|\cIi|}{2(1-\pfail)}  \right\}\\
	&\leq \min\left\{ |y_j| \,  \colon j \in \inliers, \, \sum_{i=1}^m \mathbf{1} \{|y_i| \leq |y_j|\} \geq \frac{|\inliers|}{2(1-\pfail)} \right\}\\
	&\leq \min\left\{ |y_j| \,  \colon j \in \inliers, \, \sum_{i\in \inliers} \mathbf{1} \{|y_i| \leq |y_j|\} \geq \frac{|\inliers|}{2(1-\pfail)} \right\}\\
	&= \quant_{\frac{1}{2(1-\pfail)}}
	\left(\{|y_i|\}_{i\in \cIi}\right),
	\end{align*}
	where the first equality follows since $\tfrac{|\cIi|}{2(1-\pfail)} = \tfrac{(1-\pfail)m}{2(1-\pfail)} = m/2$, the first inequality follows since the minimum is taken over a smaller set, and the second inequality follows since the sum is taken over a smaller set of indices. Therefore, we find that 
	\begin{align*}
	\Prob{\med(\{|y_i|\}_i^m) \geq \Qfail M} &\leq
	\Prob{\quant_{\tfrac{1}{2(1-\pfail)}} \left(\{|y_i|\}_{i\in \cIi}\right)
		\geq \Qfail M}\\
	&= \Prob{\quant_{\tfrac{1}{2(1-\pfail)}} \left(\{|y_i|/M\}_{i\in \cIi}\right)
		\geq \Qfail },
	\end{align*}
	and our remaining task is to bound this probability. 
	
	To bound this probability, we apply Lemma~\ref{claim:quantiles} to the i.i.d.\ sample $\{ |y_i|/M \colon i \in \inliers\}$, which is sampled from the distribution of $\cD$ of $|ab|$ where $a,b \sim \normal(0, 1)$ and $a, b$ are independent. Therefore, using the identities (for $i \in \inliers$)
	$$
	q = \Prob{ |y_i|/M \leq \Qfail} = \qfail =  \dfrac{5-2\pfail}{8(1-\pfail)}
	$$
	and choosing $p := (2(1-\pfail))^{-1} < q$, we find that 
	\begin{align*}
	\Prob{\quant_{\tfrac{1}{2(1-\pfail)}} \left(\{|y_i|/M\}_{i\in \cIi}\right)
		\geq \Qfail }
	&\leq \exp\left(\dfrac{m(q-p)^2}{2(q-p)/3+ 2q(1-q)} \right) \\
	&=  \exp\left(\dfrac{m(q-p)}{2/3+ 6q}\right) \\
	&= \exp\left(-\dfrac{3(1-2\pfail) m}{8(1-\pfail)(2+ 18q)} \right)\\
	&\leq \exp\left(-\dfrac{3(1-2\pfail)}{160} m \right),
	\end{align*} 
	where we have used the identity $q - p = \frac{1-2\pfail}{8(1-\pfail)} = (1-q)/3$ in the first equality. This completes the bound and implies that $\inliers^Q \supseteq \selinliers$ with high probability, as desired.

	Part~\ref{claim:init_dir_4:item2}. Since  $\{ |y_i|/M \colon i \in \inliers\}$ is an i.i.d.\ sample from the distribution of $|ab|$ where $a, b \sim \normal(0, 1)$ are independent, we have for each $i \in \inliers$, that 
	$$
	\Prob{i \in \inliers^Q}  = \Prob{|y_i|/M \leq \Qfail} = \Prob{|ab| \leq \Qfail} = \qfail.
	$$
	Therefore, $\EE\left[ |\inliers^Q| \right] = \qfail |\inliers| \leq \tfrac{5-2\pfail}{8(1-\pfail)} (1-\pfail)m \leq \tfrac{5}{8} m $. Finally, we apply Hoeffding's inequality (Lemma~\ref{theo:sub_hoeffding}) to the i.i.d.\ Bernoulli random variables $\1\{i \in \cIi^q\} - \EE\left[\1\{i \in \cIi^q\}\right]$ ($i \in \inliers$) to deduce that 
	\begin{align*}
	\Prob{\frac{6251m}{10000} \leq |\inliers^Q|} = \Prob{\frac{m}{10000} \leq
		|\inliers^Q| - \frac{5m}{8}} &\leq \Prob{\frac{m}{10000} \leq |\inliers^Q|-\EE |\inliers^Q|
	} \\ & \leq \exp\left(-\frac{(1/10000)^2m}{2(1-\pfail)}\right) \leq \exp\left(-\frac{m}{2\cdot10^8}\right),
	\end{align*}
	as desired. 
	
	Part~\ref{claim:init_dir_4:item3}. First write 
	\begin{align*}
	&\Prob{\frac{1}{|\inliers^Q|} \sum_{i \in \inliers^Q}  \ell_{i,1}^2
		\geq \omegafail + t}\\
	&= \Prob{\frac{1}{|\inliers^Q|} \sum_{i \in \inliers^Q}  \ell_{i,1}^2
		\geq \omegafail + t \text{ and } |\inliers^Q| \supseteq \selinliers} + \Prob{\frac{1}{|\inliers^Q|} \sum_{i \in \inliers^Q}  \ell_{i,1}^2
		\geq \omegafail + t \text{ and } \inliers^Q \not\supseteq \selinliers}\\
	&\leq \Prob{\frac{1}{|\inliers^Q|} \sum_{i \in \inliers^Q}  \ell_{i,1}^2
		\geq \omegafail + t \text{ and } |\inliers^Q| \geq \frac{m(1-2\pfail)}{2}}  + \exp\left(-\tfrac{3(1-2\pfail)}{160} m \right),
	\end{align*}
	where first inequality follows from Part~\ref{claim:init_dir_4:item1} and the bound $\frac{|\selinliers|}{m} \geq \frac{1 - 2\pfail}{2}$. Thus, we focus on bounding the first term. 
	
	To that end, notice that 
	\begin{align*}
	&\Prob{\frac{1}{|\inliers^Q|} \sum_{i \in \inliers^Q}  \ell_{i,1}^2
		\geq \omegafail + t \text{ and } |\inliers^Q| \geq \frac{m(1-2\pfail)}{2}}  \\
	&= \Prob{\frac{1}{|\inliers^Q|} \sum_{i \in \inliers^Q}  \ell_{i,1}^2
		\geq \omegafail + t \Big| |\inliers^Q| \geq \frac{m(1-2\pfail)}{2}} \Prob{|\inliers^Q| \geq \frac{m(1-2\pfail)}{2}}.
	\end{align*}
	Observe that for any index $i\in \inliers$  and $t \geq 0$, we have $\Prob{\ell_{i,1}^2 \geq t\mid i \in \inliers^Q } = \Prob{ a^2 \geq t \mid |ab| \leq \Qfail} $, where $a, b \sim \normal(0, 1)$ are independent. In addition, we have $\qfail = P(|ab| \leq \Qfail) = \tfrac{5-2\pfail}{8(1-\pfail)} \geq 5/8 > 1/2$, where we have used the fact that $\qfail$ is an increasing function of $\pfail$. Therefore, applying Lemma~\ref{lemma:cond_sub_exp}, we have the following bound: 
	\[
	\Prob{\ell_{i,1}^2 \geq t\mid i \in \inliers^Q } \leq 2 \exp(-t/2K_1) \qquad \text{ for all $ t \geq 0$ and $i \in \inliers$},
	\]
	where $K_1$ is a numerical constant. In particular, by Theorem~\ref{theo:subexp_conc} and the identity $\omegafail = \EE\left[a^2 \geq t \mid |ab| \leq \Qfail\right]$, we have the following bound  
	\begin{align*}
	\Prob{\frac{1}{|\inliers^Q|} \sum_{i \in \inliers^Q}  \ell_{i,1}^2
		\geq \omegafail + t \Big | |\inliers^Q| > \frac{m(1-2\pfail)}{2}} \leq \exp\left(- c \min\left\{\frac{t^2}{K^2},
	\frac{t}{K}\right\} \frac{m(1-2\pfail)}{2}\right)
	\end{align*}
	for numerical constants $c$ and $K$, as desired.
	
\end{proof}
The proof is complete.
\qed

\subsubsection{Proof of Lemma~\ref{lemma:Z1_concentration}}
Our goal is to bound the operator norm of the following matrix: 
\begin{align*}
Y_1 = \sum_{i \in \selinliers} \left(P_\dirw
\ell_i \ell_i^\top P_\dirw^\perp + P_\dirw^\perp \ell_i \ell_i^\top P_\dirw
\right) =
\frac{1}{m} \sum_{i \in \cIis} \ell_{i, 1} \left(
e_1 \ell_{i, 2:d}^\top + \ell_{i, 2:d} e_1^\top
\right). 
\end{align*}
Simplifying, we find that 
$$
Y_1 =
\begin{bmatrix}
0 & \lambda_{2:{d_1}}^\top \\
\lambda_{2:{d_1}} & 0
\end{bmatrix} \qquad \text{ for  } \qquad \lambda :=
\begin{bmatrix}
0 \\ \frac{1}{m}\sum_{i \in \selinliers} \ell_{i, 1} \ell_{i, 2:d_1}
\end{bmatrix} \in \RR^{d_1}.
$$ 
Evidently, $\|Y_1\|_{\op} \leq \norm{\lambda_{2:d_1}}_2 $, so our focus will be to bound this quantity. We will bound this quantity through the following claim, which is based on Gaussian concentration for Lipschitz functions. 

\begin{claim}\label{claim:concentrationc2}
	Consider the (random) function $F : \RR^{m\times (d_1-1)} \rightarrow \RR$, given by
	\begin{align*}
	F(a_1, \ldots,a_m) = \left\|\frac{1}{m}  \sum_{i \in \selinliers} \ell_{i, 1}a_i\right\|_2.
	\end{align*}
	Then $F$ is $\widehat \eta = \frac{1}{m} \sqrt{ \sum_{i \in \selinliers} \ell_{i, 1}^2}$ Lipschitz continuous and 
	\begin{align*}
	\Prob{F(\ell_{1, 2:d}, \ldots, \ell_{m, 2:d}) \geq 2\sqrt{\frac{d_1 - 1}{m}} +  t \; \Bigg|\; \widehat\eta < \frac{2}{\sqrt{m}}, \{\ell_{1, i}\}_{i=1}^m,  \selinliers } \leq \exp\left(-\frac{mt^2}{8}\right).
	\end{align*}
	Moreover, the following bound holds: 
	\begin{align*}
	\Prob{\widehat \eta \geq \frac{2}{\sqrt{m}}} \leq \exp\left(-\frac{m}{2}\right).
	\end{align*}
\end{claim}
\begin{proof}[Proof of Claim]
	For any $A = \begin{bmatrix} a_1 & \ldots &  a_m\end{bmatrix} \in \RR^{m\times (d_1-1)}$ and $B = \begin{bmatrix} b_1 &  \ldots & b_m\end{bmatrix} \in \RR^{m\times (d_1-1)}$, we have
	$$
	|F(A) - F(B)| \leq \frac{1}{m}\|(A-B)(\ell_{i, 1}\1\{ i \in \selinliers\})_{i=1}^m \|_2 \leq \frac{1}{m}\|(A-B)\|_{\op} \|(\ell_{i, 1}\1\{ i \in \selinliers\})_{i=1}^m \|_2 \leq \widehat \eta \|A - B\|_F,
	$$
	which proves that $F$ is $\widehat \eta$-Lipschitz. Therefore, since for all $i$ the variables $\ell_{i,1}$ and $\ell_{i, 2:d_1}$ are  independent,  standard results on Gaussian concentration for Lipschitz functions (applied conditionally), Theorem~\ref{theo:lips_conc}, imply that 
	\begin{align*}
	&\Prob{F(\ell_{1, 2:d}, \ldots, \ell_{m, 2:d}) - \EE\left[ F(\ell_{1, 2:d}, \ldots, \ell_{m, 2:d}) \;\Bigg|\;\widehat\eta < \frac{2}{\sqrt{m}}, \{\ell_{1, i}\}_{i=1}^m,  \selinliers \right] \geq t \; \Bigg| \;\widehat\eta < \frac{2}{\sqrt{m}},  \{\ell_{1, i}\}_{i=1}^m,  \selinliers,  } \\
	&\leq \exp\left(-\frac{mt^2}{8}\right).
	\end{align*}
	Thus, the first part of the claim is a consequence of the following bound:
	\begin{align*}
	&\EE\left[ F(\ell_{1, 2:d}, \ldots, \ell_{m, 2:d}) \; \Bigg|\; \widehat \eta < \frac{2}{\sqrt{m}},  \{\ell_{1, i}\}_{i=1}^m,  \selinliers \right] \leq \sqrt{\EE\left[ F(\ell_{1, 2:d}, \ldots, \ell_{m, 2:d})^2 \;\Bigg| \; \hat\eta < \frac{2}{\sqrt{m}}, \{\ell_{1, i}\}_{i=1}^m,  \selinliers \right]}\\
	&\hspace{40pt}= \sqrt{\frac{1}{m^2}  \EE\left[\sum_{i \in \selinliers} \ell_{i, 1}^2 (d_1 - 1)\;\Big| \;  \hat\eta < \frac{2}{\sqrt{m}}\right]} \leq 2\sqrt{\frac{d_1 - 1}{m}}.
	\end{align*}
	We now turn our attention to the high probability bound on $\widehat \eta$.
	
	To that end, notice that the (random) function $E \colon \RR^{m} \rightarrow \RR$ given by 
	$$
	E(a) =  \frac{1}{m} \sqrt{ \sum_{i \in \selinliers} a_i^2} = \frac{1}{m} \| (a_i\1\{ i \in \selinliers\})_{i=1}^m\|_2.
	$$
	is $m^{-1}$-Lipschitz continuous.
	Moreover, we have that 
	$
	\EE\left[ E(\ell_{1, i}, \ldots, \ell_{1, d}) \right]  \leq \frac{1}{m} \EE\left[\| (\ell_{1, i})_{i=1}^m\|_2\right] \leq m^{-1/2}.
	$
	Therefore, by Gaussian concentration we have 
	\begin{align*}
	\Prob{\widehat \eta  \geq \frac{2}{\sqrt{m}}} \geq \Prob{E(\ell_{1, i}, \ldots, \ell_{1, d}) - \EE \left[ E(\ell_{1, i}, \ldots, \ell_{1, d}) \right] \geq \frac{1}{\sqrt{m}}} \leq \exp\left(-\frac{m}{2}\right),
	\end{align*}
	as desired.
\end{proof}

To complete the proof, observe that  
\begin{align*}
&\Prob{\norm{\lambda_{2:d_1}}_2\geq  2\sqrt{\frac{d_1 - 1}{m}}  + t } \\
&= \Prob{\left\|\frac{1}{m}\sum_{i \in \selinliers} \ell_{i, 1} \ell_{i, 2:d_1}\right\|_2 \geq  2\sqrt{\frac{d_1 - 1}{m}} +  t}\\
&\leq \Prob{\left\|\frac{1}{m}\sum_{i \in \selinliers} \ell_{i, 1} \ell_{i, 2:d_1}\right\|_2 \geq  2\sqrt{\frac{d_1 - 1}{m}} +  t \; \Bigg| \; \widehat \eta < \frac{2}{\sqrt{m}}}\Prob{\widehat \eta < \frac{2}{\sqrt{m}}} + \Prob{\widehat \eta \geq \frac{2}{\sqrt{m}}}\\
&\leq \Prob{F(\ell_{1, 2:d}, \ldots, \ell_{m, 2:d}) \geq 2\sqrt{\frac{d_1 - 1}{m}} + t \; \Bigg|\;  \widehat \eta < \frac{2}{\sqrt{m}} } +  \exp\left(-\frac{m}{2}\right),
\end{align*}
where the second inequality is due to Claim~\ref{claim:concentrationc2}. Finally, by Claim~\ref{claim:concentrationc2}, the conditional probability is bounded as follows 
\begin{align*}
&\Prob{F(\ell_{1, 2:d}, \ldots, \ell_{m, 2:d}) \geq  2\sqrt{\frac{d_1 - 1}{m}} + t \; \Bigg|\;  \widehat \eta < \frac{2}{\sqrt{m}} } \\
&= \EE_{\selinliers, \{\ell_{i, 1}\}_{i=1}^m}\left[\Prob{F(\ell_{1, 2:d}, \ldots, \ell_{m, 2:d}) \geq  2\sqrt{\frac{d_1 - 1}{m}} + t \; \Bigg|\;  \widehat \eta < \frac{2}{\sqrt{m}}, \{\ell_{1, i}\}_{i=1}^m,  \selinliers}\right]  \\
&\leq \exp\left(-\frac{mt^2}{8}\right),
\end{align*}
which completes the proof.

\subsubsection{Proof of Lemma~\ref{lemma:Z2_concentration}}
Observe the equality
\begin{align*}
Y_2 =
\frac{1}{m} \sum_{i \in \selinliers}
\matrx{0 \\ \ell_{i, 2:d_1}} \matrx{0 & \ell_{i, 2:d_1}^\top}.
\end{align*}
Therefore, we seek to bound the following operator norm: 
\begin{align*}
\opnorm{Y_2 - \Sin\left(I_{d_1} - e_1 
	e_1^\top\right)} &=
\opnorm{ \frac{1}{m}\sum_{ i \in \selinliers}
	(\ell_{i, 2:d_1} \ell_{i, 2:d_1}^\top - I_{d_1-1}) }.
\end{align*}
Using the tower rule for expectations and appealing to Corollary~\ref{cor:sub_concentr}, we therefore deduce     
\begin{align*}
& \Prob{ \opnorm{Y_2 - \Sin\left(I_{d_1} - e_1 
		e_1^\top\right)} \geq C \sqrt{\frac{d_1}{m} + t}} \\
&\leq \EE_{\selinliers} \left[ \Prob{ \opnorm{Y_2 - \Sin\left(I_{d_1} - e_1 
		e_1^\top\right)} \geq C \sqrt{\frac{d_1}{m} + t} \; \Bigg| \; \selinliers = \cI} \right]  \leq 2 \exp(-cmt),
\end{align*}
as desired.

\qed

\subsubsection{Proof of Lemma~\ref{lemma:Z3_concentration}}
\paragraph{Noise model~\ref{NModel:1}}
Under this noise model, we write
\begin{align*}
\opnorm{Y_3 - \Sout I_{d_1}} &= \opnorm{\frac{1}{m}
	\sum_{i \in \cIos} \ell_{i} \ell_{i}^\top - \Sout I_{d_1}}.
\end{align*}
The proof follows by repeating the conditioning argument
as in the proof of~\cref{lemma:Z2_concentration}.
\paragraph{Noise model~\ref{NModel:2}}
Observe that 
\begin{align*}
\opnorm{\frac{1}{m} \sum_{i \in \seloutliers} \ell_i \ell_i^\top}
\leq \opnorm{\frac{1}{m} \sum_{i \in \outliers} \ell_i \ell_i^\top}
&\leq \opnorm{\frac{1}{m} \sum_{i \in \outliers} (\ell_i \ell_i^\top -
	I_{d_1})}
+ \opnorm{\frac{1}{m} \sum_{i \in \outliers} I_{d_1}} \\
&= \opnorm{\frac{1}{m} \sum_{i \in \outliers} (\ell_i \ell_i^\top -
	I_{d_1})} + \pfail.
\end{align*}
Appealing to Corollary~\ref{cor:sub_concentr}, the result follows immediately.
\qed

\subsection{Proof of Proposition~\ref{prop:radius_estimate_ok}}
\label{appendix:init_radius}
We will assume that $\|\hat w \hat x^\top - \dirw \dirx^\top \|_F \leq \|\hat w \hat x^\top +\dirw \dirx^\top \|_F$. We will show that with high probability, $| \widehat M - M| \leq \delta M$, and moreover in this event if $\delta<1$, we have $\widehat M > 0$. The other setting $\|\hat w \hat x^T - \dirw \dirx^\top \|_F \geq \|\hat w \hat x^T +\dirw \dirx^\top \|_F$ can treated similarly. 

We will use the guarantees of Proposition~\ref{theo:RIP}. In particular, there exist numerical constants $c_1,\ldots, c_6 > 0$ so that as long as 
$m\geq  \frac{c_1(d_1+d_2+1)}{(1-\frac{2|\cI|}{m})^2}\ln\left(c_2+\frac{1}{1-2|\cI|/m}\right)$, then with probability at least $1-4\exp\left(-c_3(1-\frac{2|\cI|}{m})^2m\right)$
we have
\begin{align*}
c_4 \|X\|_F &\leq \frac{1}{m}\|\cA(X)\|_1 \leq c_5\|X\|_F \qquad \text{for all rank $\leq 2$ matrices $X \in \RR^{d_1 \times d_2}$,}
\end{align*} 
and 
$$
c_6 \left(1- 2\pfail\right)\|X\|_F \leq  \frac{1}{m}\sum_{i\in \inliers} |\ell_i^\top X r_i| -  \frac{1}{m}\sum_{i\in \outliers} |\ell_i^\top X r_i| \qquad \text{for all rank $\leq 2$ matrices $X \in \RR^{d_1 \times d_2}$.}
$$
Throughout the remainder of the proof, suppose we are in this event.  Define the two univariate functions
\[\widehat g(a )  := \frac{1}{m}\sum_{i=1}^m\Big| y_i - (1+a)M \ell_i^\top \widehat w \widehat x^\top r_i\Big|,  \]
\[g(a )  := \frac{1}{m}\sum_{i=1}^m\Big| y_i - (1+a)M \ell_i^\top \bar w
\bar x^\top r_i\Big| 
\]
By construction, if $a^\star$ minimizes $\widehat g
(\cdot)$ then $(1+a^\star)M$ minimizes $G.$ Thus, to prove the claim we need only show that any minimizer $a^\star$ of $\widehat g$ satisfies $-\delta \leq a^\star \leq  \delta$.

To that end, first note that $g(0)$ and $\widehat g(0)$ are close:
\begin{align}
|\widehat g (0) - g(0)| &\leq  \frac{M}{m} \sum_{i=1}^m|\ell_i^\top \widehat w \widehat x^\top r_i -
\ell_i^\top \dirw \dirx^\top r_i|  \leq c_5M\left\|\widehat w \widehat
x^\top  - \dirw \dirx^\top\right\|_F,\label{eq:close_g}
\end{align}
Therefore, setting $\mu_3 = c_6\left(1 - 2\pfail\right)$, we obtain
\begin{align*}
\hat g(a) &=  \frac{1}{m}\sum_{i=1}^m\Big| y_i - (1+a)M \ell_i^\top \widehat w \widehat x^\top r_i\Big|\\
&=  \frac{1}{m}\sum_{i\in \inliers}\Big| \ell_i^\top \bar w \bar x^\top r_i - (1+a)M \ell_i^\top \widehat w \widehat x^\top r_i\Big| + \frac{1}{m}\sum_{i \in \outliers}\Big| y_i - (1+a)M \ell_i^\top \widehat w \widehat x^\top r_i\Big|\\
&\geq  \frac{1}{m}\sum_{i\in \inliers}\Big| \ell_i^\top \bar w \bar x^\top r_i - (1+a)M \ell_i^\top \widehat w \widehat x^\top r_i\Big| - \frac{1}{m}\sum_{i \in \outliers} \Big| \ell_i^\top \bar w \bar x^\top r_i - (1+a)M \ell_i^\top \widehat w \widehat x^\top r_i\Big|\\
&\hspace{20pt}+  \frac{1}{m}\sum_{i \in \outliers} \Big| y_i - \ell_i^\top \bar w \bar x^\top r_i\Big|\\
&\geq  g(0) +  \mu_3\|(1+a)M \widehat w \widehat x^\top - \bar w \bar x^\top\|_F \\
&\geq \widehat g(0)+  \mu_3\|(1+a)M \widehat w \widehat x^\top - \bar w \bar x^\top\|_F- c_5M\left\|\widehat w \widehat
x^\top  - \dirw \dirx^\top\right\|_F\\
&\geq \widehat g(0)+  \mu_3|a|M -  \left(\mu_3M + c_5M\right) \left\|\widehat w \widehat
x^\top  - \dirw \dirx^\top\right\|_F,
\end{align*}
where the second inequality follows from Theorem~\ref{theo:RIP}, the third inequality follows from Equation~\eqref{eq:close_g}, and the fourth follows from the reverse triangle inequality. 
Thus, any minimizer $a^\star$ of $\hat g$ must satisfy 
$$
|a^\star| \leq  \left(1 + \frac{c_5}{\mu_3}\right) \left\|\widehat w \widehat
x^\top  - \dirw \dirx^\top\right\|_F=\delta,
$$
as desired. 
Finally suppose $\delta<1$. Then we deduce $\widehat{M}=(1+|a^\star|)M\geq (1-\delta)M> 0$. The proof is complete.

\section{Auxiliary Lemmas}

\subsection{Technical Results}
This subsection presents technical lemmas we employed in our proofs. The first result we need is a special case of the celebrated Davis-Kahan $\sin 
\theta$ Theorem (see~\cite{DavKah70}). For any two unit vectors $u_1,v_1\in \mathbb{S}^{d-1}$, define $\theta(u_1, v_1)=\cos^{-1}(|\langle u_1,v_1\rangle|)$.
\begin{lem}
	\label{lemma:davis_kahan_variant}
	Consider symmetric matrices $X, \Delta, Z \in \R^{n \times n}$, where $Z =
	X + \Delta$. Define $\delta$ to be the eigengap $\lambda_1(X) -
	\lambda_2(X)$, and denote the first eigenvectors of $X, Z$ by $u_1, v_1$,
	respectively.  Then
	\[
	\frac{1}{\sqrt{2}}\min\set{\norm{u - v}_2, \norm{u + v}_2} \leq
	\sqrt{1 - \ip{u_1, v_1}^2} =
	\abs{\sin \theta(u_1, v_1)}
	\leq \frac{ \opnorm{\Delta}}{\delta}.
	\]
\end{lem}

Additionally, we need the following fact about $\epsilon$-nets over low-rank
matrices, which we employ frequently to prove uniform concentration 
inequalities.
\begin{lem}[Lemma 3.1 in \cite{candes2011tight}]
	\label{lemma:eps_net}
	Let $S_r := \set{X \in \RR^{d_1 \times d_2} \mmid \rank(X) \leq r, 
		\norm{X}_F = 1}$. There exists an $\epsilon$-net $\mathcal{N}$ (with 
	respect to 
	$\|\cdot\|_F$) of $S_r$ obeying \[|\mathcal{N}| \leq 
	\left(\frac{9}{\epsilon}\right)^{(d_1+d_2+1)r}.\]
\end{lem}

\subsection{Concentration Inequalities}
In this subsection, we first provide a few well-known concentration inequalities about sub-gaussian and sub-exponential random variables.

\begin{theorem}[Hoeffding's Inequality - Theorem 2.2.2 in
	\cite{vershynin2016high}]\label{theo:sub_hoeffding}
	Let $X_1, \ldots, X_N$ be independent symmetric Bernoulli random variables. Then for any $t \geq 0$, we have
	\begin{align*}
	\Prob{ \sum_{i=1}^N X_i \geq t } \leq \exp\left(-\frac{t^2}{2N}\right).
	\end{align*}
\end{theorem}

\begin{theorem}[Bernstein's Inequality - Theorem 2.8.4 in
	\cite{vershynin2016high}]\label{theo:sub_berns}
	Let $X_1, \ldots, X_N$ be independent mean-zero random variables, such that for $|X_i| \leq K$ for all $i$. Then for any $t \geq 0$, we have
	\begin{align*}
	\Prob{ \abs{\sum_{i=1}^N X_i} \geq t } \leq 2\exp\left(-\frac{t^2}{2\left(\sigma^2 + Kt/3\right)}\right)
	\end{align*}
	here $\sigma^2 = \sum \EE [X_i^2]$ is the variance of the sum.
\end{theorem}

\begin{theorem}[Sub-gaussian Concentration - Theorem 2.6.3 in
	\cite{vershynin2016high}]
	\label{theo:subg_conc}
	Let $X_1, \dots, X_N$ be independent, mean zero, sub-gaussian random variables
	and $(a_1, \dots, a_N) \in \RR^N.$ Then, for every $t\geq 0,$ we have
	\[
	\Prob{\left| \sum_{i=1}^N a_i X_i \right|\geq t} \leq 2\exp
	\left(-\frac{ct^2}{K^2 \|a\|_2^2}\right)
	\]
	where $K = \max_i \|X_i\|_{\psi_2}.$
\end{theorem}

\begin{theorem}[Sub-exponential Concentration - Theorem 2.8.2 in \cite{vershynin2016high}]\label{theo:bern}
	Let  $Z_1, \dots, Z_m$ be an independent, mean zero, sub-exponential random
	variables and let $a \in \RR^m$ be a fixed vector. Then, for any $t \geq 0$ we
	have that
	\[\Prob{\sum_{i=1}^m a_i Z_i \leq -t}
	\leq \exp\left(-c \min
	\left\{\frac{t^2}{K^2\|a\|_2^2}, \frac{t}{K\|a\|_\infty}\right\}\right)
	\]
	where $K := \max_i \| Z_i\|_{\psi_1}$ and $c > 0$ is a numerical constant.
\end{theorem}

\begin{theorem}[Corollary 2.8.3 in \cite{vershynin2016high}]
	\label{theo:subexp_conc}
	Let $X_1, \dots, X_m$ be independent, mean zero, sub-exponential random
	variables. Then, for every $t \geq 0$, we have
	\[\Prob{\left| \dfrac{1}{m} \sum_{i=1}^m X_i \right| \geq t} \leq 2 \exp
	\left[-c m \min \left(\dfrac{t^2}{K^2}, \dfrac{t}{K}\right)\right]\]
	where $c > 0$ is a numerical constant and $K := \max_i \|X_i\|_{\psi_1}$.
\end{theorem}

\begin{theorem}[Theorem 5.6 in \cite{BouLugMas13}]
	\label{theo:lips_conc}
	Let $X = (X_1, \dots, X_m)$ be a vector of $n$ independent standard normal random variables. Let $f:\RR^n \rightarrow \RR$ denote an $L$-Lipschitz function. Then, for every $t \geq 0$, we have \[\Prob{f(X) - \EE f(X) \geq t} \leq  \exp\left(-\frac{t^2}{2L^2} \right).\]
\end{theorem}

The following concentration inequalities deal with quantiles of distributions:
\begin{lem}\label{claim:quantiles}
	Let $X_1, \dots, X_m$ be an i.i.d. sample with distribution $\cD$, choose
	$Q_q$ to be the $q$ population quantile of the distribution $\cD$, that is $q
	= \Prob{X_1 \leq Q_q}$, and let $p \in (0,1)$ be any probability with $p < q.$
	Then,
	\[\Prob{\quant_p(\{X_i\}_{i=1}^m) \geq Q_q} \leq
	\exp\left(\dfrac{m(q-p)^2}{2(q-p)/3+ 2q(1-q)} \right), \]
	where $\quant_p(\{X_i\}_{i=1}^m)$ denotes the $p$-th quantile of the sample $
	\set{X_i}$.
\end{lem}
\begin{proof} It is easy to see that the following holds,
	$\quant_p(\{X_i\}_i^m) \geq Q_q$ if, and only if, $\frac{1}{m}\sum_{i=1}^m
	\1 \{X_i \leq Q_q\} \leq p$. Notice that $\1\{X_i \leq Q_q\}\sim
	\mathrm{B}(q)$ are i.i.d. Bernoulli random variables and thus $\Var(\1\{X_i
	\leq Q_q\}) = q(1-q).$ Then, the result follows by applying Bernstein's
	inequality (Theorem~\ref{theo:bern}) to $\frac{1}{m}\sum \1\{X_i \leq Q_q\} - q$. 
\end{proof}

\begin{lem}\label{lemma:cond_sub_exp}
	Let $a, b$ be i.i.d.\ sub-gaussian random variables. For any $Q > 0$ such that $q
	:= \Prob{|ab| \leq Q} > 1/2$, consider the random variable $c^2$ defined as
	$a^2$ conditioned on the event $|ab| \leq Q,$  namely for all $t$
	\[\Prob{ c^2 \leq t } = \Prob{ a^2 \leq t \mid |ab| \leq Q }. \]
	Then, $c^2$ is a sub-exponential random variable, in other words for all $t
	\geq 0$ we have that
	\[\Prob{c^2 \geq t} \leq 2 \exp(-t/2K)\]
	where $K$ is the minimum scalar such that $\Prob{a^2 \geq t} \leq 2
	\exp(-t/K)$.
\end{lem}
\begin{proof}
	Let us consider two cases.
	Suppose first $ t \leq 2K \log 2$. Then we have that $1 \leq 2\exp(-t/2K)$
	and therefore the stated inequality is trivial.
	
	Suppose now $ t \geq 2K \log 2$. Then we have that
	\begin{align*}
	\frac{t}{2K} \geq \log 2 & \iff \exp(t/K - t/2K) \geq 2
	\iff \exp(- t/2K) \geq 2 \exp(-t/K).
	\end{align*} With this we can bound the probability
	\begin{align*}
	\Prob{c^2 \geq t} &= \frac{1}{q}
	\Prob{a^2 \1\{|ab| \leq Q\} \geq t}
	\leq \frac{1}{q}\Prob{a^2 \geq t} \leq \frac{2}{q}  \exp(-t/K) \\
	& \leq 4 \exp(-t/K) \leq 2 \exp(- t/2K),
	\end{align*}
	as claimed.
\end{proof}

The following Theorem from~\cite{Vershynin12} is especially useful in bounding
the operator norm of random matrices:
\begin{theorem}[Operator norm of random matrices]
	Consider an $m \times n$ matrix $A$ whose rows $A_i$ are independent,
	sub-gaussian, isotropic random vectors in $\RR^n$. Then, for every $t \geq
	0$, one has
	\[
	\Prob{\opnorm{\frac{1}{m} A A^\top - I_n} \leq
		C \sqrt{\dfrac{n}{m} + t}} \geq 1 - 2 \exp\left(-cmt\right),
	\]
	where $C$ depends only on $K := \max_i \norm{A_i}_{\psi_2}$.
	\label{theo:opnorm_conc}
\end{theorem}
\begin{proof}
	The Theorem is a direct Corollary of~\cite[Theorem 5.39]{Vershynin12}.
	Specifically, the concavity of the square root gives us $\sqrt{a} + \sqrt{b}
	\leq \sqrt{2} \sqrt{a + b}$, implying that
	\[
	C \sqrt{\dfrac{n}{m}} + \sqrt{\dfrac{t}{m}} \leq
	C \sqrt{2} \sqrt{\dfrac{n}{m} + \dfrac{t}{m}}.
	\]
	Additionally,~\cite[Theorem 5.39]{Vershynin12} gives us that
	\[
	\Prob{\opnorm{\frac{1}{m} A A^\top - I_n} \leq
		C \sqrt{\dfrac{n}{m}} + \dfrac{t}{\sqrt{m}}} \geq
	1 - 2 \exp\left(-ct^2\right).
	\]
	Setting $t' = C \sqrt{m t}$ and a bit of relabeling, along with the square 
	root inequality, gives us the desired inequality.
\end{proof}

Let us record the following elementary consequence.

\begin{corollary}\label{cor:sub_concentr}
	Let $a_1,\ldots, a_m\in\R^d$ be independent,
	sub-gaussian, isotropic random vectors in $\RR^n$ and let $\cI\subset\{1,\ldots,m\}$ be an arbitrary set.  
	Then, for every $t \geq 0$, one has
	\[
	\Prob{\opnorm{\frac{1}{m} \sum_{i\in \cI} (a_ia_i^\top - I_d)} \leq
		C \sqrt{\dfrac{d}{m} + t}} \geq 1 - 2 \exp\left(-cmt\right),
	\]
	where $C$ depends only on $K := \max_i \norm{A_i}_{\psi_2}$.
\end{corollary}
\begin{proof}
	Consider the matrix $A\in \R^{|\cI|\times d}$ whose rows are the vectors $a_i$ for $i\in\cI$. Then we deduce
	$$\opnorm{\frac{1}{m} \sum_{i\in \cI} (a_ia_i^\top - I_d)}=\frac{|\cI|}{m}\opnorm{\frac{1}{|\cI|} \sum_{i\in \cI} a_ia_i^\top - I_d}=\frac{|\cI|}{m}\opnorm{\frac{1}{|\cI|} AA^\top - I_d}.$$
	Appealing to \cref{theo:opnorm_conc}, we therefore deduce for any $\gamma>0$ the estimate 
	$$\opnorm{\frac{1}{m} \sum_{i\in \cI} (a_ia_i^\top - I_d)}\leq \frac{|\cI|}{m}\sqrt{\frac{d}{|\cI|}+\gamma}\leq C\sqrt{\frac{d|\cI|}{m^2}+\frac{\gamma|\cI|^2}{m^2}},$$
	holds with probability $1-2\exp(-c|\cI|\gamma)$.
	Now for any $t>0$, choose $\gamma$ such that,           
	$\frac{d|\cI|}{m^2}+\frac{\gamma|\cI|^2}{m^2}=\frac{d}{m}+t$, namely $\gamma=\frac{m^2}{|\cI|^2}[\frac{d}{m}(1-\frac{|\cI|}{m})+t].$ Noting $$|\cI| \gamma=m\cdot\frac{m}{|\cI|}\left[\frac{d}{m}\left(1-\frac{|\cI|}{m}\right)+t\right]\geq mt,$$    
	completes the proof.
\end{proof}

Recall that we defined the functions $\qfail(\pfail) = \frac{5-2\pfail}{8(1-\pfail)}$ and $\Qfail(\qfail)$ given as the $\qfail$-quantile of $|ab|$ where $a,b$ are i.i.d.  standard normal. Furthermore we defined $\omegafail =  \EE[a^2 \mid |ab|\leq \Qfail].$
\begin{lemma}
	\label{lemma:monotonic_tails_condition}
	The function $\omega: [0,1] \rightarrow \RR_+$ given by
	\[\pfail \mapsto \EE[a^2 \mid |ab|\leq \Qfail]\]
	is nondecreasing.
	In particular, there exist numerical constants $c_1, c_2 > 0$ such that for any $0 \leq \pfail \leq 0.1$ we have
	\[ c_1 \leq \omegafail \leq c_2, \]
	where the tightest constants are given by $c_1 = \omega(0) \geq 0.5$ and $c_2 = \omega(0.1) \leq 0.56.$
\end{lemma}
\begin{proof}
	The bulk of this result is contained in the following claim.
	\begin{claim}
		Let $0 \leq Q \leq Q'$ be arbitrary numbers, then
		\[
		\PP(a^2 \geq t \mid |ab|\leq Q) \leq \PP(a^2 \geq t \mid |ab|\leq Q') \qquad 
		\forall t\in \RR.
		\]
	\end{claim}
	We defer the proof of the claim and show how it implies the lemma. Observe that the functions $\pfail \mapsto \qfail$ and $\qfail \mapsto \Qfail$ are nondecreasing, thus it suffices to show that the function $Q \mapsto \EE[a^2 \mid |ab|\leq Q]$ is nondecreasing. Let $0 \leq Q \leq Q'$
	\begin{align*}
	\EE[a^2 \mid |ab|\leq Q] & =  \int_0^\infty \PP(a^2 \geq t \mid |ab| \leq Q) dt  \leq \int_0^\infty \PP(a^2 \geq t \mid |ab| \leq Q') dt  = \EE[a^2 \mid |ab|\leq Q'],
	\end{align*}
	where the inequality follows from the claim and the equalities follow from the identity $\EE [X] = \int_0^\infty \PP( X \geq t) dt$ for nonnegative random variables $X.$ Hence $\omega$ is a nondecreasing function. 
	
	The above implies that for any $\pfail \in [0,0.1]$ we have $\omega(0) \leq \omegafail \leq \omega(0.1)$. Note that $\omega(0)$ is positive since it is defined by a positive integrand on a set of non-negligible measure. The bounds on $\omega(0)$ and $\omega(0.1)$ follow by a numerical computation. In particular we obtain that with $Q = 0.6$ the probability $\PP(|ab| \leq Q) \geq 0.6679 \geq  2/3 = \qfail(0.1).$ Then computing numerically (with precision set to 32 digits) we obtain $\omega(0.1) \leq \EE[a^2 \mid |ab| \leq Q] \leq 0.56.$ Similarly we find that if we set $Q = 0.5$ we get $\PP(|ab| \leq Q) \leq 0.5903 \leq  5/8 = \qfail(0).$ Then evaluating we find $\omega(0) \geq  \EE[a^2 \mid |ab| \leq Q]  \geq 0.5.$ 
	
	\begin{proof}[Proof of the claim]
		The statement of the claim is equivalent to having that for any $t \in \RR_+$ the function $h_t:\RR_+ \rightarrow \RR$ given by 
		\[Q \mapsto \frac{\PP(a^2 \leq t ; |ab|\leq Q)}{\PP(|ab|\leq Q )}\]
		is nonincreasing. Our goal is to show that $h_t' \leq 0.$ In order to prove this result we proceed as follows.
		%will repetitively find sufficient conditions using derivatives until we get a simple condition that we can ensure holds.
		Define 
		\begin{align*}
		g(Q) :=\frac{\pi}{2}\PP(|ab|\leq Q) = \int_{0}^\infty \int_0^{Q/x} \exp(-(x^2+y^2)/2) \,dy \,dx,
		\end{align*}
		and 
		\begin{align*}
		f_t(Q) := \frac{\pi}{2} \PP(a^2 \leq t ; |ab|\leq Q) = \int_{0}^{\sqrt{t}} \int_0^{Q/x} \exp(-(x^2+y^2)/2) \,dy \,dx.
		\end{align*}
		Observe $h_t = f_t/g$. Thus it suffices to show $f_t' g - f_tg' \leq 0.$ Invoking Leibniz rule we get 
		\begin{align*}
		f_t'(Q)&= \frac{\partial}{\partial Q}\int_{0}^{\sqrt{t}} \int_0^{Q/x} \exp(-(x^2+y^2)/2) \,dy \,dx \\
		& = \int_{0}^{\sqrt{t}} \frac{\partial}{\partial Q} \int_0^{Q/x} \exp(-(x^2+y^2)/2) \,dy \,dx \\
		& = \int_{0}^{\sqrt{t}} \frac{1}{x} \exp(-(x^2+Q^2/x^2)/2) \,dx.
		\end{align*}
		Repeating the same procedure we get $g'(Q)  = \int_{0}^{\infty} \frac{1}{x} \exp(-(x^2+Q^2/x^2)/2) \,dx.$ Some algebra reveals we want to show
		\begin{align*}
		\xi(t) := \frac{\left(\int_{0}^{\sqrt{t}} \frac{1}{x} \exp(-(x^2+Q^2/x^2)/2) \,dx \right)}{\left(\int_{0}^{\sqrt{t}} \int_0^{Q/x} \exp(-(x^2+y^2)/2) \,dy \,dx\right)} \leq  \frac{\left(\int_{0}^{\infty} \frac{1}{x} \exp(-(x^2+Q^2/x^2)/2) \right)}{\left(\int_{0}^\infty \int_0^{Q/x} \exp(-(x^2+y^2)/2) \,dy \,dx\right)}. 
		\end{align*}
		It is enough to show that the function $\xi(t)$ is monotonically increasing. 
		Define 
		\[\zeta_Q(t) = \int_{0}^{\sqrt{t}} \frac{1}{x} \exp(-(x^2+Q^2/x^2)/2) \,dx  \qquad \text{ and } \qquad \psi_Q(t)=\int_{0}^{\sqrt{t}} \int_0^{Q/x} \exp(-(x^2+y^2)/2) \,dy \,dx,\]
		Thus we have 
		\begin{align*}
		\zeta_Q'(t) = \frac{1}{2 t} \exp(-(t + Q^2/t)/2) \qquad \text{and} \qquad \psi_Q'(t) = \frac{1}{2\sqrt{t}} \int_0^{Q/\sqrt{t}} \exp(-(t+y^2)/2) dy.  
		\end{align*}
		Again, $\xi(t) = \zeta_Q(t)/\psi_Q(t)$, hence we need to show $\zeta_Q' \psi_Q \geq \zeta_Q \psi_Q'.$ After some algebra, this amounts to proving
		\begin{align*}
		& \left(\int_{0}^{\sqrt{t}} \int_0^{Q/x} \exp(-(x^2+y^2)/2) \,dy \,dx\right)\\ &\hspace{2cm} \geq \left(\int_{0}^{\sqrt{t}} \frac{\sqrt{t}}{x} \exp(-(Q^2/x^2-Q^2/t)/2)  \int_0^{Q/\sqrt{t}} \exp(-(x^2+y^2)/2) dy\,dx\right).
		\end{align*}
		The inequality is true if in particular the same holds for the integrands, i.e.
		\[\int_0^{Q/x} \exp(-y^2/2) \,dy \geq \frac{\sqrt{t}}{x}\exp\left(-\left(\frac{Q^2}{x^2}-\frac{Q^2}{t} \right)/2\right) \int_0^{Q/\sqrt{t}} \exp(-y^2/2) \,dy.\]
		Since $x \leq \sqrt{t},$ the previous inequality holds if
		\[x \mapsto \frac{1}{x} \frac{\exp\left(-\frac{Q^2}{2x^2} \right)}{\int_0^{Q/x} \exp(-y^2/2)dy}\]
		is increasing. By taking derivatives and reordering terms we see that this is equivalent to
		\[\frac{Q-x^2}{Qx} \int_0^{Q/x} \exp(-y^2/2)dy + \exp(-Q^2/2x^2)\geq 0.\]
		Since $\exp(-y^2/2)$ is decreasing, we have 
		\begin{align*}
		\frac{Q-x^2}{qx} \int_0^{Q/x} \exp(-y^2/2)dy \geq \frac{Q-x^2}{Qx} \frac{Q}{x}\exp(-Q^2/2x^2) \geq -\exp(-Q^2/2x^2)
		\end{align*}
		proving the claim.
	\end{proof}
	Thus the proof is complete.
\end{proof}
\end{appendices}

\end{document}

%% file: phases_gaussian_cropped.pdf_tex
%% Creator: Inkscape inkscape 0.92.2, www.inkscape.org
%% Accompanies image file 'phases_gaussian_cropped.pdf' (pdf, eps, ps)
\begingroup%
  \makeatletter%
  \providecommand\color[2][]{%
    \errmessage{(Inkscape) Color is used for the text in Inkscape, but the package 'color.sty' is not loaded}%
    \renewcommand\color[2][]{}%
  }%
  \providecommand\transparent[1]{%
    \errmessage{(Inkscape) Transparency is used (non-zero) for the text in Inkscape, but the package 'transparent.sty' is not loaded}%
    \renewcommand\transparent[1]{}%
  }%
  \providecommand\rotatebox[2]{#2}%
  \newcommand*\fsize{\dimexpr\f@size pt\relax}%
  \newcommand*\lineheight[1]{\fontsize{\fsize}{#1\fsize}\selectfont}%
  \ifx\svgwidth\undefined%
    \setlength{\unitlength}{348.74475098bp}%
    \ifx\svgscale\undefined%
      \relax%
    \else%
      \setlength{\unitlength}{\unitlength * \real{\svgscale}}%
    \fi%
  \else%
    \setlength{\unitlength}{\svgwidth}%
  \fi%
  \global\let\svgwidth\undefined%
  \global\let\svgscale\undefined%
  \makeatother%
  \begin{picture}(1,0.88859416)%
    \lineheight{1}%
    \setlength\tabcolsep{0pt}%
    \put(0,0){\includegraphics[width=\unitlength,page=1]{phases_gaussian_cropped.pdf}}%
    \put(0.13718929,0.52059068){\makebox(0,0)[t]{\lineheight{1.25}\smash{\begin{tabular}[t]{c}0.0\end{tabular}}}}%
    \put(0,0){\includegraphics[width=\unitlength,page=2]{phases_gaussian_cropped.pdf}}%
    \put(0.30103178,0.52059068){\makebox(0,0)[t]{\lineheight{1.25}\smash{\begin{tabular}[t]{c}0.1\end{tabular}}}}%
    \put(0,0){\includegraphics[width=\unitlength,page=3]{phases_gaussian_cropped.pdf}}%
    \put(0.46487425,0.52059068){\makebox(0,0)[t]{\lineheight{1.25}\smash{\begin{tabular}[t]{c}0.2\end{tabular}}}}%
    \put(0,0){\includegraphics[width=\unitlength,page=4]{phases_gaussian_cropped.pdf}}%
    \put(0.62871669,0.52059068){\makebox(0,0)[t]{\lineheight{1.25}\smash{\begin{tabular}[t]{c}0.3\end{tabular}}}}%
    \put(0,0){\includegraphics[width=\unitlength,page=5]{phases_gaussian_cropped.pdf}}%
    \put(0.79255917,0.52059068){\makebox(0,0)[t]{\lineheight{1.25}\smash{\begin{tabular}[t]{c}0.4\end{tabular}}}}%
    \put(0,0){\includegraphics[width=\unitlength,page=6]{phases_gaussian_cropped.pdf}}%
    \put(0.92363316,0.52059068){\makebox(0,0)[t]{\lineheight{1.25}\smash{\begin{tabular}[t]{c}0.48\end{tabular}}}}%
    \put(0,0){\includegraphics[width=\unitlength,page=7]{phases_gaussian_cropped.pdf}}%
    \put(0.10073305,0.79732034){\makebox(0,0)[rt]{\lineheight{1.25}\smash{\begin{tabular}[t]{r}1\end{tabular}}}}%
    \put(0,0){\includegraphics[width=\unitlength,page=8]{phases_gaussian_cropped.pdf}}%
    \put(0.10073305,0.76455184){\makebox(0,0)[rt]{\lineheight{1.25}\smash{\begin{tabular}[t]{r}2\end{tabular}}}}%
    \put(0,0){\includegraphics[width=\unitlength,page=9]{phases_gaussian_cropped.pdf}}%
    \put(0.10073305,0.73178336){\makebox(0,0)[rt]{\lineheight{1.25}\smash{\begin{tabular}[t]{r}3\end{tabular}}}}%
    \put(0,0){\includegraphics[width=\unitlength,page=10]{phases_gaussian_cropped.pdf}}%
    \put(0.10073305,0.69901486){\makebox(0,0)[rt]{\lineheight{1.25}\smash{\begin{tabular}[t]{r}4\end{tabular}}}}%
    \put(0,0){\includegraphics[width=\unitlength,page=11]{phases_gaussian_cropped.pdf}}%
    \put(0.10073305,0.66624636){\makebox(0,0)[rt]{\lineheight{1.25}\smash{\begin{tabular}[t]{r}5\end{tabular}}}}%
    \put(0,0){\includegraphics[width=\unitlength,page=12]{phases_gaussian_cropped.pdf}}%
    \put(0.10073305,0.63347786){\makebox(0,0)[rt]{\lineheight{1.25}\smash{\begin{tabular}[t]{r}6\end{tabular}}}}%
    \put(0,0){\includegraphics[width=\unitlength,page=13]{phases_gaussian_cropped.pdf}}%
    \put(0.10073305,0.60070937){\makebox(0,0)[rt]{\lineheight{1.25}\smash{\begin{tabular}[t]{r}7\end{tabular}}}}%
    \put(0,0){\includegraphics[width=\unitlength,page=14]{phases_gaussian_cropped.pdf}}%
    \put(0.10073305,0.56794089){\makebox(0,0)[rt]{\lineheight{1.25}\smash{\begin{tabular}[t]{r}8\end{tabular}}}}%
    \put(0,0){\includegraphics[width=\unitlength,page=15]{phases_gaussian_cropped.pdf}}%
    \put(0.5304112,0.84180313){\makebox(0,0)[t]{\lineheight{1.25}\smash{\begin{tabular}[t]{c}\( d_1 = d_2 = 100 \)\end{tabular}}}}%
    \put(0,0){\includegraphics[width=\unitlength,page=16]{phases_gaussian_cropped.pdf}}%
    \put(0.13718929,0.10437798){\makebox(0,0)[t]{\lineheight{1.25}\smash{\begin{tabular}[t]{c}0.0\end{tabular}}}}%
    \put(0,0){\includegraphics[width=\unitlength,page=17]{phases_gaussian_cropped.pdf}}%
    \put(0.30103178,0.10437798){\makebox(0,0)[t]{\lineheight{1.25}\smash{\begin{tabular}[t]{c}0.1\end{tabular}}}}%
    \put(0,0){\includegraphics[width=\unitlength,page=18]{phases_gaussian_cropped.pdf}}%
    \put(0.46487425,0.10437798){\makebox(0,0)[t]{\lineheight{1.25}\smash{\begin{tabular}[t]{c}0.2\end{tabular}}}}%
    \put(0,0){\includegraphics[width=\unitlength,page=19]{phases_gaussian_cropped.pdf}}%
    \put(0.62871669,0.10437798){\makebox(0,0)[t]{\lineheight{1.25}\smash{\begin{tabular}[t]{c}0.3\end{tabular}}}}%
    \put(0,0){\includegraphics[width=\unitlength,page=20]{phases_gaussian_cropped.pdf}}%
    \put(0.79255917,0.10437798){\makebox(0,0)[t]{\lineheight{1.25}\smash{\begin{tabular}[t]{c}0.4\end{tabular}}}}%
    \put(0,0){\includegraphics[width=\unitlength,page=21]{phases_gaussian_cropped.pdf}}%
    \put(0.92363316,0.10437798){\makebox(0,0)[t]{\lineheight{1.25}\smash{\begin{tabular}[t]{c}0.48\end{tabular}}}}%
    \put(0,0){\includegraphics[width=\unitlength,page=22]{phases_gaussian_cropped.pdf}}%
    \put(0.10073305,0.38110763){\makebox(0,0)[rt]{\lineheight{1.25}\smash{\begin{tabular}[t]{r}1\end{tabular}}}}%
    \put(0,0){\includegraphics[width=\unitlength,page=23]{phases_gaussian_cropped.pdf}}%
    \put(0.10073305,0.34833913){\makebox(0,0)[rt]{\lineheight{1.25}\smash{\begin{tabular}[t]{r}2\end{tabular}}}}%
    \put(0,0){\includegraphics[width=\unitlength,page=24]{phases_gaussian_cropped.pdf}}%
    \put(0.10073305,0.31557063){\makebox(0,0)[rt]{\lineheight{1.25}\smash{\begin{tabular}[t]{r}3\end{tabular}}}}%
    \put(0,0){\includegraphics[width=\unitlength,page=25]{phases_gaussian_cropped.pdf}}%
    \put(0.10073305,0.28280214){\makebox(0,0)[rt]{\lineheight{1.25}\smash{\begin{tabular}[t]{r}4\end{tabular}}}}%
    \put(0,0){\includegraphics[width=\unitlength,page=26]{phases_gaussian_cropped.pdf}}%
    \put(0.10073305,0.25003364){\makebox(0,0)[rt]{\lineheight{1.25}\smash{\begin{tabular}[t]{r}5\end{tabular}}}}%
    \put(0,0){\includegraphics[width=\unitlength,page=27]{phases_gaussian_cropped.pdf}}%
    \put(0.10073305,0.21726514){\makebox(0,0)[rt]{\lineheight{1.25}\smash{\begin{tabular}[t]{r}6\end{tabular}}}}%
    \put(0,0){\includegraphics[width=\unitlength,page=28]{phases_gaussian_cropped.pdf}}%
    \put(0.10073305,0.18449665){\makebox(0,0)[rt]{\lineheight{1.25}\smash{\begin{tabular}[t]{r}7\end{tabular}}}}%
    \put(0,0){\includegraphics[width=\unitlength,page=29]{phases_gaussian_cropped.pdf}}%
    \put(0.10073305,0.15172815){\makebox(0,0)[rt]{\lineheight{1.25}\smash{\begin{tabular}[t]{r}8\end{tabular}}}}%
    \put(0,0){\includegraphics[width=\unitlength,page=30]{phases_gaussian_cropped.pdf}}%
    \put(0.5304112,0.42559041){\makebox(0,0)[t]{\lineheight{1.25}\smash{\begin{tabular}[t]{c}\( d_1 = d_2 = 200 \)\end{tabular}}}}%
    \put(0.61629639,0.02660883){\makebox(0,0)[t]{\lineheight{1.25}\smash{\begin{tabular}[t]{c}Corruption level\end{tabular}}}}%
    \put(0.04243343,0.34940553){\rotatebox{90}{\makebox(0,0)[lt]{\lineheight{1.25}\smash{\begin{tabular}[t]{l}\( m / (d_1 + d_2) \)\end{tabular}}}}}%
  \end{picture}%
\endgroup%

%% file: phases_dethadm_cropped.pdf_tex
%% Creator: Inkscape inkscape 0.92.2, www.inkscape.org
%% Accompanies image file 'phases_dethadm_cropped.pdf' (pdf, eps, ps)
\begingroup%
  \makeatletter%
  \providecommand\color[2][]{%
    \errmessage{(Inkscape) Color is used for the text in Inkscape, but the package 'color.sty' is not loaded}%
    \renewcommand\color[2][]{}%
  }%
  \providecommand\transparent[1]{%
    \errmessage{(Inkscape) Transparency is used (non-zero) for the text in Inkscape, but the package 'transparent.sty' is not loaded}%
    \renewcommand\transparent[1]{}%
  }%
  \providecommand\rotatebox[2]{#2}%
  \newcommand*\fsize{\dimexpr\f@size pt\relax}%
  \newcommand*\lineheight[1]{\fontsize{\fsize}{#1\fsize}\selectfont}%
  \ifx\svgwidth\undefined%
    \setlength{\unitlength}{348.74475996bp}%
    \ifx\svgscale\undefined%
      \relax%
    \else%
      \setlength{\unitlength}{\unitlength * \real{\svgscale}}%
    \fi%
  \else%
    \setlength{\unitlength}{\svgwidth}%
  \fi%
  \global\let\svgwidth\undefined%
  \global\let\svgscale\undefined%
  \makeatother%
  \begin{picture}(1,0.88859413)%
    \lineheight{1}%
    \setlength\tabcolsep{0pt}%
    \put(0,0){\includegraphics[width=\unitlength,page=1]{phases_dethadm_cropped.pdf}}%
    \put(0.13718928,0.52059067){\makebox(0,0)[t]{\lineheight{1.25}\smash{\begin{tabular}[t]{c}0.0\end{tabular}}}}%
    \put(0,0){\includegraphics[width=\unitlength,page=2]{phases_dethadm_cropped.pdf}}%
    \put(0.30103177,0.52059067){\makebox(0,0)[t]{\lineheight{1.25}\smash{\begin{tabular}[t]{c}0.1\end{tabular}}}}%
    \put(0,0){\includegraphics[width=\unitlength,page=3]{phases_dethadm_cropped.pdf}}%
    \put(0.46487423,0.52059067){\makebox(0,0)[t]{\lineheight{1.25}\smash{\begin{tabular}[t]{c}0.2\end{tabular}}}}%
    \put(0,0){\includegraphics[width=\unitlength,page=4]{phases_dethadm_cropped.pdf}}%
    \put(0.62871666,0.52059067){\makebox(0,0)[t]{\lineheight{1.25}\smash{\begin{tabular}[t]{c}0.3\end{tabular}}}}%
    \put(0,0){\includegraphics[width=\unitlength,page=5]{phases_dethadm_cropped.pdf}}%
    \put(0.79255914,0.52059067){\makebox(0,0)[t]{\lineheight{1.25}\smash{\begin{tabular}[t]{c}0.4\end{tabular}}}}%
    \put(0,0){\includegraphics[width=\unitlength,page=6]{phases_dethadm_cropped.pdf}}%
    \put(0.92363312,0.52059067){\makebox(0,0)[t]{\lineheight{1.25}\smash{\begin{tabular}[t]{c}0.48\end{tabular}}}}%
    \put(0,0){\includegraphics[width=\unitlength,page=7]{phases_dethadm_cropped.pdf}}%
    \put(0.10073305,0.79732032){\makebox(0,0)[rt]{\lineheight{1.25}\smash{\begin{tabular}[t]{r}1\end{tabular}}}}%
    \put(0,0){\includegraphics[width=\unitlength,page=8]{phases_dethadm_cropped.pdf}}%
    \put(0.10073305,0.76455183){\makebox(0,0)[rt]{\lineheight{1.25}\smash{\begin{tabular}[t]{r}2\end{tabular}}}}%
    \put(0,0){\includegraphics[width=\unitlength,page=9]{phases_dethadm_cropped.pdf}}%
    \put(0.10073305,0.73178334){\makebox(0,0)[rt]{\lineheight{1.25}\smash{\begin{tabular}[t]{r}3\end{tabular}}}}%
    \put(0,0){\includegraphics[width=\unitlength,page=10]{phases_dethadm_cropped.pdf}}%
    \put(0.10073305,0.69901485){\makebox(0,0)[rt]{\lineheight{1.25}\smash{\begin{tabular}[t]{r}4\end{tabular}}}}%
    \put(0,0){\includegraphics[width=\unitlength,page=11]{phases_dethadm_cropped.pdf}}%
    \put(0.10073305,0.66624635){\makebox(0,0)[rt]{\lineheight{1.25}\smash{\begin{tabular}[t]{r}5\end{tabular}}}}%
    \put(0,0){\includegraphics[width=\unitlength,page=12]{phases_dethadm_cropped.pdf}}%
    \put(0.10073305,0.63347785){\makebox(0,0)[rt]{\lineheight{1.25}\smash{\begin{tabular}[t]{r}6\end{tabular}}}}%
    \put(0,0){\includegraphics[width=\unitlength,page=13]{phases_dethadm_cropped.pdf}}%
    \put(0.10073305,0.60070936){\makebox(0,0)[rt]{\lineheight{1.25}\smash{\begin{tabular}[t]{r}7\end{tabular}}}}%
    \put(0,0){\includegraphics[width=\unitlength,page=14]{phases_dethadm_cropped.pdf}}%
    \put(0.10073305,0.56794088){\makebox(0,0)[rt]{\lineheight{1.25}\smash{\begin{tabular}[t]{r}8\end{tabular}}}}%
    \put(0,0){\includegraphics[width=\unitlength,page=15]{phases_dethadm_cropped.pdf}}%
    \put(0.53041117,0.84180311){\makebox(0,0)[t]{\lineheight{1.25}\smash{\begin{tabular}[t]{c}\( d_1 = d_2 = 100 \)\end{tabular}}}}%
    \put(0,0){\includegraphics[width=\unitlength,page=16]{phases_dethadm_cropped.pdf}}%
    \put(0.13718928,0.10437799){\makebox(0,0)[t]{\lineheight{1.25}\smash{\begin{tabular}[t]{c}0.0\end{tabular}}}}%
    \put(0,0){\includegraphics[width=\unitlength,page=17]{phases_dethadm_cropped.pdf}}%
    \put(0.30103177,0.10437799){\makebox(0,0)[t]{\lineheight{1.25}\smash{\begin{tabular}[t]{c}0.1\end{tabular}}}}%
    \put(0,0){\includegraphics[width=\unitlength,page=18]{phases_dethadm_cropped.pdf}}%
    \put(0.46487423,0.10437799){\makebox(0,0)[t]{\lineheight{1.25}\smash{\begin{tabular}[t]{c}0.2\end{tabular}}}}%
    \put(0,0){\includegraphics[width=\unitlength,page=19]{phases_dethadm_cropped.pdf}}%
    \put(0.62871666,0.10437799){\makebox(0,0)[t]{\lineheight{1.25}\smash{\begin{tabular}[t]{c}0.3\end{tabular}}}}%
    \put(0,0){\includegraphics[width=\unitlength,page=20]{phases_dethadm_cropped.pdf}}%
    \put(0.79255914,0.10437799){\makebox(0,0)[t]{\lineheight{1.25}\smash{\begin{tabular}[t]{c}0.4\end{tabular}}}}%
    \put(0,0){\includegraphics[width=\unitlength,page=21]{phases_dethadm_cropped.pdf}}%
    \put(0.92363312,0.10437799){\makebox(0,0)[t]{\lineheight{1.25}\smash{\begin{tabular}[t]{c}0.48\end{tabular}}}}%
    \put(0,0){\includegraphics[width=\unitlength,page=22]{phases_dethadm_cropped.pdf}}%
    \put(0.10073305,0.38110763){\makebox(0,0)[rt]{\lineheight{1.25}\smash{\begin{tabular}[t]{r}1\end{tabular}}}}%
    \put(0,0){\includegraphics[width=\unitlength,page=23]{phases_dethadm_cropped.pdf}}%
    \put(0.10073305,0.34833913){\makebox(0,0)[rt]{\lineheight{1.25}\smash{\begin{tabular}[t]{r}2\end{tabular}}}}%
    \put(0,0){\includegraphics[width=\unitlength,page=24]{phases_dethadm_cropped.pdf}}%
    \put(0.10073305,0.31557064){\makebox(0,0)[rt]{\lineheight{1.25}\smash{\begin{tabular}[t]{r}3\end{tabular}}}}%
    \put(0,0){\includegraphics[width=\unitlength,page=25]{phases_dethadm_cropped.pdf}}%
    \put(0.10073305,0.28280214){\makebox(0,0)[rt]{\lineheight{1.25}\smash{\begin{tabular}[t]{r}4\end{tabular}}}}%
    \put(0,0){\includegraphics[width=\unitlength,page=26]{phases_dethadm_cropped.pdf}}%
    \put(0.10073305,0.25003365){\makebox(0,0)[rt]{\lineheight{1.25}\smash{\begin{tabular}[t]{r}5\end{tabular}}}}%
    \put(0,0){\includegraphics[width=\unitlength,page=27]{phases_dethadm_cropped.pdf}}%
    \put(0.10073305,0.21726515){\makebox(0,0)[rt]{\lineheight{1.25}\smash{\begin{tabular}[t]{r}6\end{tabular}}}}%
    \put(0,0){\includegraphics[width=\unitlength,page=28]{phases_dethadm_cropped.pdf}}%
    \put(0.10073305,0.18449665){\makebox(0,0)[rt]{\lineheight{1.25}\smash{\begin{tabular}[t]{r}7\end{tabular}}}}%
    \put(0,0){\includegraphics[width=\unitlength,page=29]{phases_dethadm_cropped.pdf}}%
    \put(0.10073305,0.15172816){\makebox(0,0)[rt]{\lineheight{1.25}\smash{\begin{tabular}[t]{r}8\end{tabular}}}}%
    \put(0,0){\includegraphics[width=\unitlength,page=30]{phases_dethadm_cropped.pdf}}%
    \put(0.53041117,0.42559041){\makebox(0,0)[t]{\lineheight{1.25}\smash{\begin{tabular}[t]{c}\( d_1 = d_2 = 200 \)\end{tabular}}}}%
    \put(0.61629637,0.02660885){\makebox(0,0)[t]{\lineheight{1.25}\smash{\begin{tabular}[t]{c}Corruption level\end{tabular}}}}%
    \put(0.04243343,0.34940553){\rotatebox{90}{\makebox(0,0)[lt]{\lineheight{1.25}\smash{\begin{tabular}[t]{l}\( m / (d_1 + d_2) \)\end{tabular}}}}}%
  \end{picture}%
\endgroup%

%% file: phases_gaussian_adv_cropped.pdf_tex
%% Creator: Inkscape inkscape 0.92.2, www.inkscape.org
%% Accompanies image file 'phases_gaussian_adv_cropped.pdf' (pdf, eps, ps)
\begingroup%
  \makeatletter%
  \providecommand\color[2][]{%
    \errmessage{(Inkscape) Color is used for the text in Inkscape, but the package 'color.sty' is not loaded}%
    \renewcommand\color[2][]{}%
  }%
  \providecommand\transparent[1]{%
    \errmessage{(Inkscape) Transparency is used (non-zero) for the text in Inkscape, but the package 'transparent.sty' is not loaded}%
    \renewcommand\transparent[1]{}%
  }%
  \providecommand\rotatebox[2]{#2}%
  \newcommand*\fsize{\dimexpr\f@size pt\relax}%
  \newcommand*\lineheight[1]{\fontsize{\fsize}{#1\fsize}\selectfont}%
  \ifx\svgwidth\undefined%
    \setlength{\unitlength}{348.74474823bp}%
    \ifx\svgscale\undefined%
      \relax%
    \else%
      \setlength{\unitlength}{\unitlength * \real{\svgscale}}%
    \fi%
  \else%
    \setlength{\unitlength}{\svgwidth}%
  \fi%
  \global\let\svgwidth\undefined%
  \global\let\svgscale\undefined%
  \makeatother%
  \begin{picture}(1,0.92136263)%
    \lineheight{1}%
    \setlength\tabcolsep{0pt}%
    \put(0,0){\includegraphics[width=\unitlength,page=1]{phases_gaussian_adv_cropped.pdf}}%
    \put(0.14128535,0.48782218){\makebox(0,0)[t]{\lineheight{1.25}\smash{\begin{tabular}[t]{c}0.0\end{tabular}}}}%
    \put(0,0){\includegraphics[width=\unitlength,page=2]{phases_gaussian_adv_cropped.pdf}}%
    \put(0.34608845,0.48782218){\makebox(0,0)[t]{\lineheight{1.25}\smash{\begin{tabular}[t]{c}0.1\end{tabular}}}}%
    \put(0,0){\includegraphics[width=\unitlength,page=3]{phases_gaussian_adv_cropped.pdf}}%
    \put(0.55089151,0.48782218){\makebox(0,0)[t]{\lineheight{1.25}\smash{\begin{tabular}[t]{c}0.2\end{tabular}}}}%
    \put(0,0){\includegraphics[width=\unitlength,page=4]{phases_gaussian_adv_cropped.pdf}}%
    \put(0.75569458,0.48782218){\makebox(0,0)[t]{\lineheight{1.25}\smash{\begin{tabular}[t]{c}0.3\end{tabular}}}}%
    \put(0,0){\includegraphics[width=\unitlength,page=5]{phases_gaussian_adv_cropped.pdf}}%
    \put(0.91953701,0.48782218){\makebox(0,0)[t]{\lineheight{1.25}\smash{\begin{tabular}[t]{c}0.38\end{tabular}}}}%
    \put(0,0){\includegraphics[width=\unitlength,page=6]{phases_gaussian_adv_cropped.pdf}}%
    \put(0.10073305,0.82599276){\makebox(0,0)[rt]{\lineheight{1.25}\smash{\begin{tabular}[t]{r}1\end{tabular}}}}%
    \put(0,0){\includegraphics[width=\unitlength,page=7]{phases_gaussian_adv_cropped.pdf}}%
    \put(0.10073305,0.78503214){\makebox(0,0)[rt]{\lineheight{1.25}\smash{\begin{tabular}[t]{r}2\end{tabular}}}}%
    \put(0,0){\includegraphics[width=\unitlength,page=8]{phases_gaussian_adv_cropped.pdf}}%
    \put(0.10073305,0.74407152){\makebox(0,0)[rt]{\lineheight{1.25}\smash{\begin{tabular}[t]{r}3\end{tabular}}}}%
    \put(0,0){\includegraphics[width=\unitlength,page=9]{phases_gaussian_adv_cropped.pdf}}%
    \put(0.10073305,0.70311091){\makebox(0,0)[rt]{\lineheight{1.25}\smash{\begin{tabular}[t]{r}4\end{tabular}}}}%
    \put(0,0){\includegraphics[width=\unitlength,page=10]{phases_gaussian_adv_cropped.pdf}}%
    \put(0.10073305,0.6621503){\makebox(0,0)[rt]{\lineheight{1.25}\smash{\begin{tabular}[t]{r}5\end{tabular}}}}%
    \put(0,0){\includegraphics[width=\unitlength,page=11]{phases_gaussian_adv_cropped.pdf}}%
    \put(0.10073305,0.62118967){\makebox(0,0)[rt]{\lineheight{1.25}\smash{\begin{tabular}[t]{r}6\end{tabular}}}}%
    \put(0,0){\includegraphics[width=\unitlength,page=12]{phases_gaussian_adv_cropped.pdf}}%
    \put(0.10073305,0.58022908){\makebox(0,0)[rt]{\lineheight{1.25}\smash{\begin{tabular}[t]{r}7\end{tabular}}}}%
    \put(0,0){\includegraphics[width=\unitlength,page=13]{phases_gaussian_adv_cropped.pdf}}%
    \put(0.10073305,0.53926845){\makebox(0,0)[rt]{\lineheight{1.25}\smash{\begin{tabular}[t]{r}8\end{tabular}}}}%
    \put(0,0){\includegraphics[width=\unitlength,page=14]{phases_gaussian_adv_cropped.pdf}}%
    \put(0.53041117,0.8745716){\makebox(0,0)[t]{\lineheight{1.25}\smash{\begin{tabular}[t]{c}\( d_1 = d_2 = 100 \)\end{tabular}}}}%
    \put(0,0){\includegraphics[width=\unitlength,page=15]{phases_gaussian_adv_cropped.pdf}}%
    \put(0.14128535,0.07160949){\makebox(0,0)[t]{\lineheight{1.25}\smash{\begin{tabular}[t]{c}0.0\end{tabular}}}}%
    \put(0,0){\includegraphics[width=\unitlength,page=16]{phases_gaussian_adv_cropped.pdf}}%
    \put(0.34608845,0.07160949){\makebox(0,0)[t]{\lineheight{1.25}\smash{\begin{tabular}[t]{c}0.1\end{tabular}}}}%
    \put(0,0){\includegraphics[width=\unitlength,page=17]{phases_gaussian_adv_cropped.pdf}}%
    \put(0.55089151,0.07160949){\makebox(0,0)[t]{\lineheight{1.25}\smash{\begin{tabular}[t]{c}0.2\end{tabular}}}}%
    \put(0,0){\includegraphics[width=\unitlength,page=18]{phases_gaussian_adv_cropped.pdf}}%
    \put(0.75569458,0.07160949){\makebox(0,0)[t]{\lineheight{1.25}\smash{\begin{tabular}[t]{c}0.3\end{tabular}}}}%
    \put(0,0){\includegraphics[width=\unitlength,page=19]{phases_gaussian_adv_cropped.pdf}}%
    \put(0.91953701,0.07160949){\makebox(0,0)[t]{\lineheight{1.25}\smash{\begin{tabular}[t]{c}0.38\end{tabular}}}}%
    \put(0,0){\includegraphics[width=\unitlength,page=20]{phases_gaussian_adv_cropped.pdf}}%
    \put(0.10073305,0.40978005){\makebox(0,0)[rt]{\lineheight{1.25}\smash{\begin{tabular}[t]{r}1\end{tabular}}}}%
    \put(0,0){\includegraphics[width=\unitlength,page=21]{phases_gaussian_adv_cropped.pdf}}%
    \put(0.10073305,0.36881942){\makebox(0,0)[rt]{\lineheight{1.25}\smash{\begin{tabular}[t]{r}2\end{tabular}}}}%
    \put(0,0){\includegraphics[width=\unitlength,page=22]{phases_gaussian_adv_cropped.pdf}}%
    \put(0.10073305,0.32785884){\makebox(0,0)[rt]{\lineheight{1.25}\smash{\begin{tabular}[t]{r}3\end{tabular}}}}%
    \put(0,0){\includegraphics[width=\unitlength,page=23]{phases_gaussian_adv_cropped.pdf}}%
    \put(0.10073305,0.2868982){\makebox(0,0)[rt]{\lineheight{1.25}\smash{\begin{tabular}[t]{r}4\end{tabular}}}}%
    \put(0,0){\includegraphics[width=\unitlength,page=24]{phases_gaussian_adv_cropped.pdf}}%
    \put(0.10073305,0.24593757){\makebox(0,0)[rt]{\lineheight{1.25}\smash{\begin{tabular}[t]{r}5\end{tabular}}}}%
    \put(0,0){\includegraphics[width=\unitlength,page=25]{phases_gaussian_adv_cropped.pdf}}%
    \put(0.10073305,0.20497699){\makebox(0,0)[rt]{\lineheight{1.25}\smash{\begin{tabular}[t]{r}6\end{tabular}}}}%
    \put(0,0){\includegraphics[width=\unitlength,page=26]{phases_gaussian_adv_cropped.pdf}}%
    \put(0.10073305,0.16401636){\makebox(0,0)[rt]{\lineheight{1.25}\smash{\begin{tabular}[t]{r}7\end{tabular}}}}%
    \put(0,0){\includegraphics[width=\unitlength,page=27]{phases_gaussian_adv_cropped.pdf}}%
    \put(0.10073305,0.12305577){\makebox(0,0)[rt]{\lineheight{1.25}\smash{\begin{tabular}[t]{r}8\end{tabular}}}}%
    \put(0,0){\includegraphics[width=\unitlength,page=28]{phases_gaussian_adv_cropped.pdf}}%
    \put(0.53041117,0.4583589){\makebox(0,0)[t]{\lineheight{1.25}\smash{\begin{tabular}[t]{c}\( d_1 = d_2 = 200 \)\end{tabular}}}}%
    \put(0.61629637,0.0266088){\makebox(0,0)[t]{\lineheight{1.25}\smash{\begin{tabular}[t]{c}Corruption level\end{tabular}}}}%
    \put(0.04243343,0.34940553){\rotatebox{90}{\makebox(0,0)[lt]{\lineheight{1.25}\smash{\begin{tabular}[t]{l}\( m / (d_1 + d_2) \)\end{tabular}}}}}%
  \end{picture}%
\endgroup%

%% file: phases_dethadm_adv_cropped.pdf_tex
%% Creator: Inkscape inkscape 0.92.2, www.inkscape.org
%% Accompanies image file 'phases_dethadm_adv_cropped.pdf' (pdf, eps, ps)
\begingroup%
  \makeatletter%
  \providecommand\color[2][]{%
    \errmessage{(Inkscape) Color is used for the text in Inkscape, but the package 'color.sty' is not loaded}%
    \renewcommand\color[2][]{}%
  }%
  \providecommand\transparent[1]{%
    \errmessage{(Inkscape) Transparency is used (non-zero) for the text in Inkscape, but the package 'transparent.sty' is not loaded}%
    \renewcommand\transparent[1]{}%
  }%
  \providecommand\rotatebox[2]{#2}%
  \newcommand*\fsize{\dimexpr\f@size pt\relax}%
  \newcommand*\lineheight[1]{\fontsize{\fsize}{#1\fsize}\selectfont}%
  \ifx\svgwidth\undefined%
    \setlength{\unitlength}{348.74474823bp}%
    \ifx\svgscale\undefined%
      \relax%
    \else%
      \setlength{\unitlength}{\unitlength * \real{\svgscale}}%
    \fi%
  \else%
    \setlength{\unitlength}{\svgwidth}%
  \fi%
  \global\let\svgwidth\undefined%
  \global\let\svgscale\undefined%
  \makeatother%
  \begin{picture}(1,0.92136263)%
    \lineheight{1}%
    \setlength\tabcolsep{0pt}%
    \put(0,0){\includegraphics[width=\unitlength,page=1]{phases_dethadm_adv_cropped.pdf}}%
    \put(0.14128535,0.48782218){\makebox(0,0)[t]{\lineheight{1.25}\smash{\begin{tabular}[t]{c}0.0\end{tabular}}}}%
    \put(0,0){\includegraphics[width=\unitlength,page=2]{phases_dethadm_adv_cropped.pdf}}%
    \put(0.34608845,0.48782218){\makebox(0,0)[t]{\lineheight{1.25}\smash{\begin{tabular}[t]{c}0.1\end{tabular}}}}%
    \put(0,0){\includegraphics[width=\unitlength,page=3]{phases_dethadm_adv_cropped.pdf}}%
    \put(0.55089151,0.48782218){\makebox(0,0)[t]{\lineheight{1.25}\smash{\begin{tabular}[t]{c}0.2\end{tabular}}}}%
    \put(0,0){\includegraphics[width=\unitlength,page=4]{phases_dethadm_adv_cropped.pdf}}%
    \put(0.75569458,0.48782218){\makebox(0,0)[t]{\lineheight{1.25}\smash{\begin{tabular}[t]{c}0.3\end{tabular}}}}%
    \put(0,0){\includegraphics[width=\unitlength,page=5]{phases_dethadm_adv_cropped.pdf}}%
    \put(0.91953701,0.48782218){\makebox(0,0)[t]{\lineheight{1.25}\smash{\begin{tabular}[t]{c}0.38\end{tabular}}}}%
    \put(0,0){\includegraphics[width=\unitlength,page=6]{phases_dethadm_adv_cropped.pdf}}%
    \put(0.10073305,0.82599276){\makebox(0,0)[rt]{\lineheight{1.25}\smash{\begin{tabular}[t]{r}1\end{tabular}}}}%
    \put(0,0){\includegraphics[width=\unitlength,page=7]{phases_dethadm_adv_cropped.pdf}}%
    \put(0.10073305,0.78503214){\makebox(0,0)[rt]{\lineheight{1.25}\smash{\begin{tabular}[t]{r}2\end{tabular}}}}%
    \put(0,0){\includegraphics[width=\unitlength,page=8]{phases_dethadm_adv_cropped.pdf}}%
    \put(0.10073305,0.74407152){\makebox(0,0)[rt]{\lineheight{1.25}\smash{\begin{tabular}[t]{r}3\end{tabular}}}}%
    \put(0,0){\includegraphics[width=\unitlength,page=9]{phases_dethadm_adv_cropped.pdf}}%
    \put(0.10073305,0.70311091){\makebox(0,0)[rt]{\lineheight{1.25}\smash{\begin{tabular}[t]{r}4\end{tabular}}}}%
    \put(0,0){\includegraphics[width=\unitlength,page=10]{phases_dethadm_adv_cropped.pdf}}%
    \put(0.10073305,0.6621503){\makebox(0,0)[rt]{\lineheight{1.25}\smash{\begin{tabular}[t]{r}5\end{tabular}}}}%
    \put(0,0){\includegraphics[width=\unitlength,page=11]{phases_dethadm_adv_cropped.pdf}}%
    \put(0.10073305,0.62118967){\makebox(0,0)[rt]{\lineheight{1.25}\smash{\begin{tabular}[t]{r}6\end{tabular}}}}%
    \put(0,0){\includegraphics[width=\unitlength,page=12]{phases_dethadm_adv_cropped.pdf}}%
    \put(0.10073305,0.58022908){\makebox(0,0)[rt]{\lineheight{1.25}\smash{\begin{tabular}[t]{r}7\end{tabular}}}}%
    \put(0,0){\includegraphics[width=\unitlength,page=13]{phases_dethadm_adv_cropped.pdf}}%
    \put(0.10073305,0.53926845){\makebox(0,0)[rt]{\lineheight{1.25}\smash{\begin{tabular}[t]{r}8\end{tabular}}}}%
    \put(0,0){\includegraphics[width=\unitlength,page=14]{phases_dethadm_adv_cropped.pdf}}%
    \put(0.53041117,0.8745716){\makebox(0,0)[t]{\lineheight{1.25}\smash{\begin{tabular}[t]{c}\( d_1 = d_2 = 100 \)\end{tabular}}}}%
    \put(0,0){\includegraphics[width=\unitlength,page=15]{phases_dethadm_adv_cropped.pdf}}%
    \put(0.14128535,0.07160949){\makebox(0,0)[t]{\lineheight{1.25}\smash{\begin{tabular}[t]{c}0.0\end{tabular}}}}%
    \put(0,0){\includegraphics[width=\unitlength,page=16]{phases_dethadm_adv_cropped.pdf}}%
    \put(0.34608845,0.07160949){\makebox(0,0)[t]{\lineheight{1.25}\smash{\begin{tabular}[t]{c}0.1\end{tabular}}}}%
    \put(0,0){\includegraphics[width=\unitlength,page=17]{phases_dethadm_adv_cropped.pdf}}%
    \put(0.55089151,0.07160949){\makebox(0,0)[t]{\lineheight{1.25}\smash{\begin{tabular}[t]{c}0.2\end{tabular}}}}%
    \put(0,0){\includegraphics[width=\unitlength,page=18]{phases_dethadm_adv_cropped.pdf}}%
    \put(0.75569458,0.07160949){\makebox(0,0)[t]{\lineheight{1.25}\smash{\begin{tabular}[t]{c}0.3\end{tabular}}}}%
    \put(0,0){\includegraphics[width=\unitlength,page=19]{phases_dethadm_adv_cropped.pdf}}%
    \put(0.91953701,0.07160949){\makebox(0,0)[t]{\lineheight{1.25}\smash{\begin{tabular}[t]{c}0.38\end{tabular}}}}%
    \put(0,0){\includegraphics[width=\unitlength,page=20]{phases_dethadm_adv_cropped.pdf}}%
    \put(0.10073305,0.40978005){\makebox(0,0)[rt]{\lineheight{1.25}\smash{\begin{tabular}[t]{r}1\end{tabular}}}}%
    \put(0,0){\includegraphics[width=\unitlength,page=21]{phases_dethadm_adv_cropped.pdf}}%
    \put(0.10073305,0.36881942){\makebox(0,0)[rt]{\lineheight{1.25}\smash{\begin{tabular}[t]{r}2\end{tabular}}}}%
    \put(0,0){\includegraphics[width=\unitlength,page=22]{phases_dethadm_adv_cropped.pdf}}%
    \put(0.10073305,0.32785884){\makebox(0,0)[rt]{\lineheight{1.25}\smash{\begin{tabular}[t]{r}3\end{tabular}}}}%
    \put(0,0){\includegraphics[width=\unitlength,page=23]{phases_dethadm_adv_cropped.pdf}}%
    \put(0.10073305,0.2868982){\makebox(0,0)[rt]{\lineheight{1.25}\smash{\begin{tabular}[t]{r}4\end{tabular}}}}%
    \put(0,0){\includegraphics[width=\unitlength,page=24]{phases_dethadm_adv_cropped.pdf}}%
    \put(0.10073305,0.24593757){\makebox(0,0)[rt]{\lineheight{1.25}\smash{\begin{tabular}[t]{r}5\end{tabular}}}}%
    \put(0,0){\includegraphics[width=\unitlength,page=25]{phases_dethadm_adv_cropped.pdf}}%
    \put(0.10073305,0.20497699){\makebox(0,0)[rt]{\lineheight{1.25}\smash{\begin{tabular}[t]{r}6\end{tabular}}}}%
    \put(0,0){\includegraphics[width=\unitlength,page=26]{phases_dethadm_adv_cropped.pdf}}%
    \put(0.10073305,0.16401636){\makebox(0,0)[rt]{\lineheight{1.25}\smash{\begin{tabular}[t]{r}7\end{tabular}}}}%
    \put(0,0){\includegraphics[width=\unitlength,page=27]{phases_dethadm_adv_cropped.pdf}}%
    \put(0.10073305,0.12305577){\makebox(0,0)[rt]{\lineheight{1.25}\smash{\begin{tabular}[t]{r}8\end{tabular}}}}%
    \put(0,0){\includegraphics[width=\unitlength,page=28]{phases_dethadm_adv_cropped.pdf}}%
    \put(0.53041117,0.4583589){\makebox(0,0)[t]{\lineheight{1.25}\smash{\begin{tabular}[t]{c}\( d_1 = d_2 = 200 \)\end{tabular}}}}%
    \put(0.61629637,0.0266088){\makebox(0,0)[t]{\lineheight{1.25}\smash{\begin{tabular}[t]{c}Corruption level\end{tabular}}}}%
    \put(0.04243343,0.34940553){\rotatebox{90}{\makebox(0,0)[lt]{\lineheight{1.25}\smash{\begin{tabular}[t]{l}\( m / (d_1 + d_2) \)\end{tabular}}}}}%
  \end{picture}%
\endgroup%